\documentclass[11pt]{amsart}
\usepackage{hyperref}
\usepackage{a4wide}
\usepackage{amsmath}
\usepackage[utf8]{inputenc}
\usepackage{amssymb}
\usepackage{amsopn}
\usepackage{epsfig}
\usepackage{amsfonts}
\usepackage{latexsym}
\usepackage{graphicx}
\usepackage{enumerate}
\usepackage[dvipsnames]{xcolor}
\usepackage{mathrsfs }
\usepackage{tikz}
\usetikzlibrary{arrows,automata}

\newtheorem{theorem}{Theorem}[section]
\newtheorem{lemma}[theorem]{Lemma}
\newtheorem{proposition}[theorem]{Proposition}
\newtheorem{corollary}[theorem]{Corollary}

\theoremstyle{definition}
\newtheorem{definition}[theorem]{Definition}
\newtheorem{example}[theorem]{Example}

\theoremstyle{remark}
\newtheorem{remark}[theorem]{Remark}

\numberwithin{equation}{section}

\newcommand{\R}{\ensuremath{\mathbb{R}}}
\newcommand{\N}{\ensuremath{\mathbb{N}}}
\renewcommand{\L}{\ensuremath{\mathbb{L}}}
\renewcommand{\S}{\ensuremath{\mathbb{S}}}

\newcommand{\BB}{\mathscr{B}}
\newcommand{\EE}{\mathscr{E}}
\newcommand{\F}{F^*}
\newcommand{\cF}{\mathcal{F}}

\newcommand{\II}{\mathcal{I}}
\newcommand{\Kt}{\widetilde{\mathcal{K}}}
\newcommand{\LL}{L^*}

\renewcommand{\c}{ {\mathbf{c}}}
\renewcommand{\a}{ {\mathbf{a}}}
\renewcommand{\d}{ {\mathbf{d}}}
\renewcommand{\b}{ {\mathbf{b}}}

\renewcommand{\t}{ {\mathbf{t}}}
\newcommand{\s}{ {\mathbf{s}}}
\newcommand{\bR}{\mathbf{R}}
\newcommand{\cs}{ {\mathbf{S}}}

\renewcommand{\r}{ {\mathbf{r}}}

\newcommand{\cL}{\mathcal{L}}
\newcommand{\XX}{\mathcal{X}}
\renewcommand{\u}{\mathbf{u}}
\renewcommand{\v}{\mathbf{v}}
\newcommand{\w}{\mathbf{w}}
\newcommand{\x}{\mathbf{x}}
\newcommand{\y}{\mathbf{y}}
\newcommand{\z}{\mathbf{z}}

\newcommand{\set}[1]{\left\{#1\right\}}

\newcommand{\ga}{\gamma}

\newcommand{\Ga}{\Gamma}
\newcommand{\ep}{\varepsilon}
\newcommand{\f}{\infty}

\newcommand{\al}{\alpha}
\newcommand{\lle}{\preccurlyeq}
\newcommand{\lge}{\succcurlyeq}
\renewcommand{\a}{ \mathbf{a}}
\newcommand{\si}{\sigma}

\newcommand{\La}{\Lambda}

\begin{document}
\title[The $\beta$-transformation with a hole at $0$]{Entropy plateaus, transitivity and bifurcation sets for the $\beta$-transformation with a hole at $0$}

\author[P. Allaart]{Pieter Allaart}
\address[P. Allaart]{Mathematics Department, University of North Texas, 1155 Union Cir \#311430, Denton, TX 76203-5017, U.S.A.}
\email{allaart@unt.edu}

\author[D. Kong]{Derong Kong}
\address[D. Kong]{College of Mathematics and Statistics, Mathematical Research Center, Chongqing University, Chongqing 401331, People's Republic of China.}
\email{derongkong@126.com}

\date{\today}
\dedicatory{}


\begin{abstract}
Given $\beta{>1}$, let $T_\beta$ be the $\beta$-transformation on the unit circle $[0,1)$ such that $T_\beta(x)=\beta x\pmod 1$. For each $t\in[0,1)$ let $K_\beta(t)$ be the survivor set consisting of all $x\in[0,1)$ whose orbit $\{T^n_\beta(x): n\ge 0\}$ never enters the interval $[0,t)$. Letting $\EE_\beta$ denote the bifurcation set of the set-valued map $t\mapsto K_\beta(t)$, Kalle et al. [{\em Ergodic Theory Dynam. Systems}, 40 (9): 2482--2514, 2020] conjectured that
\[
\dim_H\big(\EE_\beta\cap[t,1]\big)=\dim_H K_\beta(t) \qquad \forall\,t\in(0,1).
\]
The main purpose of this article is to prove this conjecture. We do so by investigating dynamical properties of the symbolic equivalent of the survivor set $K_\beta(t)$, in particular its entropy and topological transitivity. In addition, we compare $\EE_\beta$ with the bifurcation set $\BB_\beta$ of the map $t\mapsto \dim_H K_\beta(t)$ (which is a decreasing devil's staircase by a theorem of Kalle et al.), and show that, for Lebesgue-almost every $\beta{>1}$, the difference $\EE_\beta\backslash\BB_\beta$ has positive Hausdorff dimension, but for every $k\in\{0,1,2,\dots\}\cup\{\aleph_0\}$, there are infinitely many values of $\beta$ such that the cardinality of $\EE_\beta\backslash\BB_\beta$ is exactly $k$. For a countable but dense subset of $\beta$'s, we also determine the intervals of constancy of the function $t\mapsto \dim_H K_\beta(t)$.

Some connections with other topics in dynamics, such as kneading invariants of Lorenz maps and the doubling map with an arbitrary hole, are also discussed.
\end{abstract}


\maketitle

\tableofcontents

\section{Introduction}\label{sec:introduction}

Let ${\beta>1}$ and define the map $T_\beta:[0,1)\to[0,1)$ by $T_\beta(x):=\beta x\!\!\mod 1$. Given a number $t\in(0,1)$, let
\[
K_\beta(t):=\{x\in[0,1): T_\beta^n(x)\geq t\ \forall n\geq 0\}.
\]
Up to countably many points, $K_\beta(t)$ is the {\em survivor set} of the open dynamical system $(T_\beta,O)$ with the {\em hole} $O=(0,t)$. Such open dynamical systems were considered first by Urba\'{n}ski \cite{Urbanski_1986,Urbanski-87}, who proved, among many other things, that the map $\eta_2: t\mapsto \dim_H K_2(t)$ is a decreasing devil's staircase. This result was extended more recently by Kalle {\em et al.} \cite{Kalle-Kong-Langeveld-Li-18} to all $\beta\in(1,2]$.

Other authors have studied the $\beta$-transformation with a hole. For instance, Glendinning and Sidorov \cite{Glendinning-Sidorov-2015} considered the doubing map $T_2$ with an arbitrary hole $(a,b)$ and proved several results about the size of the survivor set. Their work was partially extended by Clark \cite{Lyndsey-2016}, who characterized for the general $\beta$-transformation with $\beta\in(1,2]$ the holes $(a,b)$ with an uncountable survivor set. Carminati and Tiozzo \cite{Carminati-Tiozzo-2017} showed that the local H\"older exponent of $\eta_2$ at any bifurcation point $t$ is precisely $\eta_2(t)$. {In two recent papers \cite{Allaart-Kong-2021,Allaart-Kong-2024}}, the present authors determined for each $\beta{>1}$ the critical value
\begin{equation} \label{eq:critical-value}
\tau(\beta):=\min\{t>0:\dim_H K_\beta(t)=0\}
\end{equation}
and showed that the functon $\beta\mapsto\tau(\beta)$ has an infinite set of discontinuities.

Note that the set-valued map $t\mapsto K_\beta(t)$ is non-increasing. We define the {\em bifurcation set}
\begin{equation} \label{eq:EE-beta}
\EE_\beta:=\{t\in[0,1): K_\beta(t')\neq K_\beta(t)\ \forall\,t'>t\}.
\end{equation}
It was shown by Urba\'{n}ski \cite{Urbanski-87} that
\begin{equation*}
\dim_H\big(\EE_2\cap[t,1]\big)=\dim_H K_2(t) \qquad \forall\,t\in(0,1).
\end{equation*}
Kalle {\em et al.} \cite{Kalle-Kong-Langeveld-Li-18} investigated the topological structure of $\EE_\beta$, and conjectured that a similar identity should hold for all $\beta\in(1,2]$, namely,
\begin{equation} \label{eq:dimension-identity}
\dim_H\big(\EE_\beta\cap[t,1]\big)=\dim_H K_\beta(t) \qquad \forall\,t\in(0,1).
\end{equation}
Baker and Kong \cite{Baker-Kong-2020} proved the conjecture for the special case when $\beta$ is a {\em multinacci number}, i.e. the positive root of $x^{n+1}=x^n+x^{n-1}+\dots+x+1$, where $n\in\N$. The main purpose of this paper is to prove the conjecture for all $\beta{>1}$.

\begin{theorem} \label{thm:main}
The equation \eqref{eq:dimension-identity} holds for all $\beta{>1}$.
\end{theorem}

The identity \eqref{eq:dimension-identity} is an instance of the interplay between the ``parameter space" (in this case, $\EE_\beta$) and the ``dynamical space" (in our case $K_\beta(t)$) which was first observed by Douady \cite{Douady-1995} in the context of dynamics of real quadratic polynomials. Theorems of this type have frequently occurred in the literature in a variety of settings. For instance, Tiozzo \cite{Tiozzo_2015} considers for $c\in\R$ the set of angles of external rays which ``land" on the real slice of the Mandelbrot set to the right of $c$ (parameter space) and the set of external angles which land on the real slice of the Julia set of the map $z\mapsto z^2+c$ (dynamical space), showing that these two sets have the same Hausdorff dimension. More recently, Carminati and Tiozzo \cite{CT-2021} proved an identity analogous to \eqref{eq:dimension-identity} in the setting of continued fractions, and the present authors proved a similar result in the context of densities of Cantor measure (see \cite{Allaart-Kong-2022}). {This phenomenon has also been observed in the context of unique non-integer base expansions and some other settings \cite{Huang-Kong-23, Jiang-Kong-Li-22, Kong-Li-Lv-Wang-Xu-20}.}

Since the map $\eta_\beta: t\mapsto \dim_H K_\beta(t)$ is a decreasing devil's staircase, it is natural to consider also the {\em dimension bifurcation set}
\[
\BB_\beta:=\set{t\ge 0: \dim_H K_\beta(t')<\dim_H K_\beta(t)~\forall t'>t}, \qquad \beta{>1}.
\]
Clearly $\BB_\beta\subseteq \EE_\beta$. As noted in \cite{Kalle-Kong-Langeveld-Li-18}, the inclusion is proper for most values of $\beta$, since $\EE_\beta$ typically has isolated points whereas $\BB_\beta$ does not. However, Baker and Kong \cite{Baker-Kong-2020} showed that $\BB_\beta=\EE_\beta$ when $\beta$ is a multinacci number. We extend this result here as follows:

\begin{theorem} \label{thm:size-of-E-minus-B}
The difference $\EE_\beta\backslash\BB_\beta$ has positive Hausdorff dimension for Lebesgue-almost every $\beta{>1}$. On the other hand, for each $k\in\{0,1,2,\dots\}\cup\{\aleph_0\}$, there are infinitely many values of $\beta$ such that $|\EE_\beta\backslash\BB_\beta|=k$. However, there is no $\beta{>1}$ such that $\EE_\beta\backslash\BB_\beta$ is uncountable but of zero Hausdorff dimension.
\end{theorem}

Regarding the local dimension of $\EE_\beta$ and $\BB_\beta$ we can prove the following.

\begin{theorem} \label{thm:local-dimension}
Let $\beta{>1}$.
\begin{enumerate}[{\rm(i)}]
\item For each $t\in\BB_\beta$, we have
\[
\lim_{\ep\searrow 0}\dim_H\big(\EE_\beta\cap(t-\ep,t+\ep)\big)=\lim_{\ep\searrow 0}\dim_H\big(\EE_\beta\cap[t,t+\ep)\big)=\dim_H K_\beta(t).
\]
\item If $\al(\beta)$ is eventually periodic, then for each $t\in\BB_\beta$, we also have
\[
\lim_{\ep\searrow 0}\dim_H\big(\BB_\beta\cap(t-\ep,t+\ep)\big)=\lim_{\ep\searrow 0}\dim_H\big(\BB_\beta\cap[t,t+\ep)\big)=\dim_H K_\beta(t).
\]
\end{enumerate}
\end{theorem}

Statement (i) is an easy consequence of Theorem \ref{thm:main}, but (ii) will follow only after we prove several other results. Therefore, we defer the proof until the end of the paper. At present, we do not know how to compute the local dimension of $\EE_\beta$ at points of $\EE_\beta\backslash\BB_\beta$. Note that, since $\EE_\beta$ and $\BB_\beta$ are defined as {\em right} bifurcation sets, it is much less clear how to compute their {\em left} local dimension functions.

We prove the main theorems above using tools from symbolic dynamics. A critical role is played here by the {\em greedy expansion} $b(t,\beta)$, defined as the lexicographically largest expansion of the number $t$ in base $\beta$, and by the {\em quasi-greedy expansion} $\al(\beta)$ of 1 in base $\beta$, defined as the lexicographically largest $\beta$-expansion of 1 not ending in $0^\f$.

We first define the symbolic equivalent of $K_\beta(t)$, that is,
\[
\mathcal{K}_\beta(t):=\{\z\in {A_\beta^\N}: b(t,\beta)\lle \sigma^n(\z)\prec \alpha(\beta)\ \forall n\geq 0\},
\]
where {$A_\beta:=\{0,1,\dots,\lceil\beta\rceil-1\}$ and} $\si$ denotes the left shift map on {$A_\beta^{\N}$}.
Note that $K_\beta(t)=\pi_\beta(\mathcal{K}_\beta(t))$, where $\pi_\beta:{A_\beta}^\N\to\R$ is the projection map given by
\[
\pi_\beta((d_i)):=\sum_{i=1}^\f \frac{d_i}{\beta^i}.
\]
It is convenient to slightly enlarge the set $\mathcal{K}_\beta(t)$ and consider the closed set
\[
\Kt_\beta(t):=\{\z\in {A_\beta^\N}: b(t,\beta)\lle \sigma^n(\z)\lle \alpha(\beta)\ \forall n\geq 0\}.
\]
Clearly $\mathcal{K}_\beta(t)\subseteq \Kt_\beta(t)$, and the difference $\Kt_\beta(t)\backslash \mathcal{K}_\beta(t)$ is at most countable. The point is that $\Kt_\beta(t)$ is always a subshift, hence we can apply theorems from symbolic dynamics. It was shown in \cite{Kalle-Kong-Langeveld-Li-18} that
\begin{equation} \label{eq:dimension-formula}
\dim_H K_\beta(t)=\frac{h(\Kt_\beta(t))}{\log\beta} \qquad\forall \beta\in(1,2], \quad \forall t\in(0,1),
\end{equation}
where for a subshift $\mathcal{X}$ of ${A_\beta}^\N$, $h(\mathcal{X})$ denotes the {\em (topological) entropy} of $\mathcal{X}$, that is,
\[
h(\mathcal{X}):=\lim_{n\to\f} {\frac{\log\#\mathcal{B}_n(\mathcal{X})}{ n}}.
\]
Here $\mathcal{B}_n(\mathcal{X})$ is the set of all words of length $n$ occurring in some sequence in $\mathcal{X}$, and $\#B$ denotes the number of elements of the finite set $B$. {The argument in \cite{Kalle-Kong-Langeveld-Li-18} extends readily to all $\beta>1$; hence \eqref{eq:dimension-formula} holds for all $\beta>1$.}

\medskip
The reader may observe that $\Kt_\beta(t)$ is a special case of the subshifts
\begin{equation} \label{eq:Sigma-ab}
\Sigma_{\a,\b}:=\{\z\in{A_\beta}^\N: \a\lle \si^n(\z)\lle \b\ \forall\,n\geq 0\},
\end{equation}
where $\a$ and $\b$ are sequences in ${A_\beta}^\N$. These subshifts have many important applications in dynamical systems and number theory. They arose in the 1990s in connection with kneading sequences of Lorenz maps \cite{Glendinning-Hall-1996,Glendinning-Sparrow-1993,Hubbard-Sparrow-1990}; see the next subsection. The symmetric case $\Sigma_{\a,\tilde{\a}}$, where $\tilde{\a}$ is the {\em reflection} of $\a$ obtained by interchanging zeros and ones, plays a prominent role in the study of unique expansions in non-integer bases (see \cite{AlcarazBarrera-Baker-Kong-2016, Allaart-2017, Allaart-Baker-Kong-17, Darczy_Katai_1995, DeVries_Komornik_2008, Erdos_Joo_Komornik_1990, Glendinning_Sidorov_2001, Komornik-Kong-Li-17}).  The general subshifts $\Sigma_{\a,\b}$ were studied in detail by Labarca and Moreira \cite{Labarca-Moreira-2006} and more recently by Komornik, Steiner and Zou \cite{Komornik-Steiner-Zou-2022}. Several of our results here have new implications for these subshifts.

The complement of the bifurcation set $\BB_\beta$ consists of intervals on which $h(\Kt_\beta(t))$ (and hence $\dim_H K_\beta(t)$) is constant. For $\beta=2$ these intervals, called {\em entropy plateaus}, were characterized by Nilsson \cite{Nilsson-2009}. Say a finite word $\w$ is {\em Lyndon} if it is aperiodic and lexicographically smallest among all its cyclic permutations. Nilsson proved that the plateaus of $\eta_2: t\mapsto h(\Kt_2(t))$ are precisely the intervals $[\pi_2(\w0^\f),\pi_2(\w^\f)]$, as $\w$ ranges over all Lyndon words. It is easy to see that these intervals are pairwise disjoint.

For {non-integer $\beta$} the situation is more complex. First of all the sequences $\w0^\f$ and $\w^\f$ may not be valid greedy $\beta$-expansions. We say a Lyndon word $\w$ is {\em $\beta$-Lyndon} if $\w^\f=b(t,\beta)$ for some $t\in[0,1)$. By Lemma \ref{lem:greedy-expansion} below this is the case if and only if $\si^n(\w^\f)\prec\al(\beta)$ for all $n\geq 0$. If $\w$ is a $\beta$-Lyndon word, we call the interval $[\pi_\beta(\w0^\f),\pi_\beta(\w^\f)]$ a {\em $\beta$-Lyndon interval}; see \cite[Definition 1.2]{Baker-Kong-2020}. Furthermore, for each $\beta$-Lyndon word $\w$ we define the {\em extended $\beta$-Lyndon interval} (EBLI for short) by
\[
I_\w:=\begin{cases}
[\pi_\beta(\w0^\f),\pi_\beta(\w^\f)] & \mbox{if $\si^n(\al(\beta))\succ \w^\f$ for all $n\geq 0$},\\
[\pi_\beta\big(\w^-(\al_1\dots\al_m^-)^\f\big),\pi_\beta(\w^\f)] & \mbox{otherwise},
\end{cases}
\]
where in the second case, $m:=\min\{n\geq 1: \si^n(\al(\beta))\lle \w^\f\}$, and for any word $\w=w_1\dots w_n$ {not ending in 0, $\w^-:=w_1\dots w_{n-1}(w_n-1)$}. In general, the EBLIs need not be disjoint. However, we will show that if two EBLIs intersect then one contains the other. Our last main result is the following generalization of Nilsson's theorem:

\begin{theorem} \label{thm:entropy-plateaus}
Assume $\al(\beta)$ is eventually periodic. Then the plateaus of $\eta_\beta: t\mapsto h(\Kt_\beta(t))$ (hence of $t\mapsto \dim_H K_\beta(t)$) are precisely $[\tau(\beta),1)$ and the maximal (with respect to set inclusion) EBLIs in $[0,\tau(\beta)]$.
\end{theorem}

We will prove this theorem in Section \ref{sec:EBLI}, where we also indicate how to determine whether a given EBLI is maximal.

Since every subshift $\Sigma_{\a,\b}$ is essentially (up to a countable set) of the form $\Kt_\beta(t)$ for some parameter pair $(\beta,t)$ (see Subsection \ref{subsec:Lorenz} below), Theorem \ref{thm:entropy-plateaus} also gives the entropy plateaus for the map $\a\mapsto h(\Sigma_{\a,\b})$ for fixed $\b$, or, by symmetry, for the map $\b\mapsto h(\Sigma_{\a,\b})$ for fixed $\a$. For completeness, we state a precise result for the former. {Here and later on, we use the notation $\N_0:=\N\cup\{0\}$.}

\begin{corollary} \label{cor:symbolic-plateaus}
Let $\b\succ 10^\f$ be a sequence in ${\N_0}^\N$, and let $\b'$ be the lexicographically largest sequence less than or equal to $\b$ such that $0^\f\prec\si^n(\b')\lle \b'$ for all $n\geq 0$. Then $\b'=\al(\beta)$ for some $\beta{>1}$, and if $\b'$ is eventually periodic, the entropy plateaus of $\a\mapsto h(\Sigma_{\a,\b})$ with strictly positive entropy value are the lexicographic intervals
\[
\tilde{I}_\w:=\begin{cases}
[\w^-\b',\w^\f] & \mbox{if}\ I_\w=[\pi_\beta(\w0^\f),\pi_\beta(\w^\f)],\\
[\w^-(\al_1\dots\al_m^-)^\f,\w^\f] & \mbox{if}\ I_\w=[\pi_\beta\big(\w^-(\al_1\dots\al_m^-)^\f\big),\pi_\beta(\w^\f)],
\end{cases}
\]
as $\w$ ranges over all $\beta$-Lyndon words such that $I_\w\subseteq[0,\tau(\beta)]$.
\end{corollary}

Note that the symbolic plateaus in the first case extend a bit further to the left, due to the fact that no sequence in $[\w^-\b',\w0^\f)$ is a greedy $\beta$-expansion. Further, in the above corollary we did not describe the $0$-plateau (that is, the plateau with entropy value $0$) of $\a\mapsto h(\Sigma_{\a,\b})$, as doing so requires some notation to be developed later; see Remark \ref{rem:zero-plateau}. In any case, the $0$-plateau was already determined previously in \cite{Labarca-Moreira-2006}; see also \cite{Komornik-Steiner-Zou-2022}.

\medskip
In order to prove the identity \eqref{eq:dimension-identity}, it is necessary to better understand the dynamics of the subshifts $\Kt_\beta(t)$. We say a subshift $\mathcal{X}$ of ${A_\beta}^\N$ is {\em (topologically) transitive}\footnote{Our definition is slightly stronger than the usual one, which requires that for any two finite words $\u$ and $\w$ in $\cL(\mathcal{X})$, there is a word $\v$ such that $\u\v\w\in\cL(\mathcal{X})$. However, it is possible to show that for the subshifts $\Kt_\beta(t)$, the two definitions are equivalent. We adopt the stronger definition here for convenience.} if for any word $\u\in\cL(\mathcal{X})$ and any sequence $\z\in\mathcal{X}$ there is a word $\v$ such that $\u\v\z\in\mathcal{X}$.
Here $\cL(\mathcal{X})$ denotes the {\em language} of $\mathcal{X}$; that is, the set of all finite words occurring in sequences in $\mathcal{X}$. Unfortunately, $\Kt_\beta(t)$ may fail to be transitive; in that case, we try to identify a subshift of full entropy that contains the sequence $b(t,\beta)$. Even this is not always possible; in Section \ref{sec:basic-interiors} we will identify so-called {\em non-transitivity windows}, i.e., intervals inside of which no subshift of full entropy containing the sequence $b(t,\beta)$ exists. What comes to the rescue in that situation is that the entropy $h(\Kt_\beta(t))$ is constant on these intervals. (The closures of these non-transitivity windows are in fact EBLIs.)

The topological structure of the subshift $\Kt_\beta(t)$ depends very heavily on $\beta$ (and, to a lesser extent, on $t$). We decompose the interval {$(1,\f)$} into countably many disjoint subsets which are of four essentially different types: {\em basic intervals}, the {\em exceptional set}, the set of {\em finitely renormalizable $\beta$}, and the set of {\em infinitely renormalizable $\beta$}. Different sections of the paper deal with different parts of this decomposition. In the course of this work, we completely characterize transitivity for all points $t$ whose greedy $\beta$-expansion is purely periodic. Note that $\Kt_\beta(t_R)$ is a subshift of the {\em $\beta$-shift}
\begin{equation} \label{eq:beta-shift}
\Sigma_\beta:=\{\z\in{A_\beta}^\N: \sigma^n(\z)\lle \alpha(\beta)\ \forall n\geq 0\},
\end{equation}
and therefore, this paper also characterizes transitivity for a large class of subshifts of $\beta$-shifts. We are not aware of any earlier work of this nature.

\subsection{Connection with kneading sequences of Lorenz maps} \label{subsec:Lorenz}
\ 
{Here we assume $\beta\in(1,2]$, so $A_\beta=\{0,1\}$.}
The symbolic survivor set $\mathcal{K}_\beta(t)$ and the subshift $\Kt_\beta(t)$ are closely related to the set of kneading sequences of Lorenz maps. Below we follow roughly the notion and terminology of Hubbard and Sparrow \cite{Hubbard-Sparrow-1990}. A function $f:[0,1]\to[0,1]$ is called a {\em Lorenz map} if there is a point $c\in(0,1)$ such that $f$ is continuous and strictly increasing on $[0,c)$ and $(c,1]$, $f(c-)=1$ and $f(c+)=0$. Such maps arise in the study of a geometric model of the Lorenz differential equations. The function $T_\beta$ for $\beta\in(1,2]$ is a Lorenz map, with $c=1/\beta$. For a point $x\in[0,1]$ that is not a pre-image of $c$, the {\em kneading sequence} of $x$ is the sequence $\mathbf{k}_f(x)=\ep_1\ep_2\dots\in\{0,1\}^\N$ given by
\[
\ep_n=\begin{cases}
0 & \mbox{if $f^{n-1}(x)<c$},\\
1 & \mbox{if $f^{n-1}(x)>c$}.
\end{cases}
\]
If $x$ is a pre-image of $c$, we define two kneading sequences $\mathbf{k}_f^+(x)$ and $\mathbf{k}_f^-(x)$ by
\[
\mathbf{k}_f^+(x):=\lim_{y\searrow x} \mathbf{k}_f(y), \qquad \mathbf{k}_f^-(x):=\lim_{y\nearrow x} \mathbf{k}_f(y),
\]
where $y$ runs through points which are not pre-images of $c$, and the limits are with respect to the product topology on $\{0,1\}^\N$. For $f=T_\beta$, $\mathbf{k}_f^+(x)$ is precisely the greedy expansion of $x$ in base $\beta$, whereas $\mathbf{k}_f^-(x)$ is the quasi-greedy expansion. The {\em kneading invariant} of a Lorenz map $f$ is the pair $(\a(f),\b(f))$ given by $\a(f)=\mathbf{k}_f^+(0)$ and $\b(f)=\mathbf{k}_f^-(1)$. For $f=T_\beta$, note that $\a(f)=0^\f$ because $T_\beta(0)=0$, whereas $\b(f)=\al(\beta)$.

Hubbard and Sparrow \cite{Hubbard-Sparrow-1990} consider {\em topologically expansive} Lorenz maps, a relatively weak notion of expansive which is equivalent to the set of pre-images of $c$ being dense in $[0,1]$. The map $T_\beta$ is topologically expansive, as is any Lorenz map with $C^1$ branches and derivative bounded above 1. Hubbard and Sparrow  prove that a pair $(\a,\b)$ of sequences in $\{0,1\}^\N$ is the kneading invariant of some topologically expansive Lorenz map if and only if
\begin{equation} \label{eq:kneading-inequalities}
\a\lle \si^n(\a)\prec\b \qquad\mbox{and} \qquad \a\prec\si^n(\b)\lle \b \qquad \forall\,n\geq 0.
\end{equation}
Moreover, the corresponding map $f$ is unique up to conjugacy. Now observe that, for any pair of sequences $(\a,\b)$ satisfying \eqref{eq:kneading-inequalities}, there is a parameter pair $(\beta,t)\in(1,2]\times[0,1)$ such that $\a=b(t,\beta)$ and $\b=\al(\beta)$. (This follows from Lemmas \ref{lem:quasi-greedy expansion-alpha-q} and \ref{lem:greedy-expansion} below.) By the second main result of \cite{Hubbard-Sparrow-1990}, the set of all (upper or lower) kneading sequences of points in $[0,1]$ for the map $f$ is then precisely $\Kt_\beta(t)$.

Vice versa, suppose $\a$ and $\b$ are two $0-1$ sequences beginning with $0$ and $1$, respectively, and consider the set $\Sigma_{\a,\b}$ from \eqref{eq:Sigma-ab}.
If $\Sigma_{\a,\b}\neq\emptyset$, then there is a pair of sequences $(\a',\b')$ such that $\Sigma_{\a,\b}=\Sigma_{\a',\b'}$ and
\begin{equation} \label{eq:LW}
\a'\lle \si^n(\a')\lle\b' \qquad\mbox{and} \qquad \a'\lle\si^n(\b')\lle \b' \qquad \forall\,n\geq 0.
\end{equation}
(Simply take $\a':=\min\Sigma_{\a,\b}$ and $\b':=\max\Sigma_{\a,\b}$.) Labarca and Moreira \cite{Labarca-Moreira-2006} call the set of pairs $(\a',\b')$ satisfying \eqref{eq:LW} the {\em lexicographical world}, and investigate it in detail. They show moreover that, if $\Sigma_{\a,\b}$ is uncountable, then there is a pair $(\tilde{\a},\tilde{\b})$ of sequences satisfying the kneading inequalities \eqref{eq:kneading-inequalities} such that $\Sigma_{\tilde{\a},\tilde{\b}}\subseteq\Sigma_{\a,\b}$ and $\Sigma_{\a,\b}\backslash\Sigma_{\tilde{\a},\tilde{\b}}$ is countable. Namely, let $\tilde{\a}$ be the smallest, and $\tilde{\b}$ be the largest condensation point of $\Sigma_{\a,\b}$; see \cite[p.~689]{Labarca-Moreira-2006}. (Recall that a point $x$ is a {\em condensation point} of a set $A$ in a topological space if every neighborhood of $x$ contains uncountably many points of $A$.) The sequences $\tilde{\a}$ and $\tilde{\b}$ are well defined and have the required properties because the set of condensation points of $\Sigma_{\a,\b}$, like $\Sigma_{\a,\b}$ itself, is closed and invariant under $\si$. While this characterization of $\tilde{\a}$ and $\tilde{\b}$ is somewhat abstract, it is also not too difficult to devise an algorithm for their construction.

From these observations, we conclude that
\begin{enumerate}[(i)]
\item For each pair $(\a,\b)$ such that $\a$ begins with $0$ and $\b$ begins with $1$, there is a parameter pair $(\beta,t)$ such that $\Kt_\beta(t)\subseteq\Sigma_{\a,\b}$ and $\Sigma_{\a,\b}\backslash \Kt_\beta(t)$ is countable;
\item For each pair $(\beta,t)$ such that $\Kt_\beta(t)$ is uncountable, there is a parameter pair $(\beta',t')$ such that $(\a,\b):=(b(t,\beta),\al(\beta))$ satisfies \eqref{eq:kneading-inequalities},
\[
\Kt_{\beta'}(t')\subseteq \Kt_\beta(t), \qquad\mbox{and} \qquad \Kt_\beta(t)\backslash\Kt_{\beta'}(t')\ \mbox{is countable}.
\]
\end{enumerate}
Thus, loosely speaking, each $\Sigma_{\a,\b}$ is almost realized as $\Kt_\beta(t)$ for some pair $(\beta,t)$, and each set $\Kt_\beta(t)$ is almost  the set of kneading sequences for some topologically expansive Lorenz map, where ``almost" here means, up to a countable set.

Labarca and Moreira \cite{Labarca-Moreira-2006} give a further analysis of the sets $\Sigma_{\a,\b}$. For instance, they (implicitly) determine for a given sequence $\a$ the critical $\b$ at which $\Sigma_{\a,\b}$ becomes uncountable, and show that when $\Sigma_{\a,\b}$ is uncountable, it has positive entropy and hence positive Hausdorff dimension. {This implies that $K_\beta(\tau_\beta)$ is at most countable, which was implicitly proven in \cite{Allaart-Kong-2022}.}  More recently, their results have been generalized by Komornik, Steiner and Zou \cite{Komornik-Steiner-Zou-2022}. Both of these works are also very closely related to the recent paper \cite{Allaart-Kong-2021} by the present authors.

\subsection{Connection with the doubling map with a general hole}

In 2015, Glendinning and Sidorov \cite{Glendinning-Sidorov-2015} considered the doubling map $T_2$ on the circle $[0,1)$ with an arbitrary hole $(a,b)$, and determined for which pairs $(a,b)$ with $0\leq a<b<1$ the survivor set
\[
K_2(a,b):=\{x\in[0,1): T_2^n(x)\not\in(a,b)\ \forall\,n\geq 0\}
\]
is uncountable. Replacing the numbers $a$ and $b$ with their greedy binary expansions $\a$ and $\b$, respectively, the set $K_2(a,b)$ can be identified with
\[
\Omega_{\a,\b}:=\{\z\in\{0,1\}^\N: \si^n(\z)\lle \a\ \mbox{or}\ \si^n(\z)\lge \b\ \forall\,n\geq 0\}.
\]
Since it is easy to see that $\Omega_{\a,\b}$ is uncountable when $\a$ and $\b$ have the same first digit, Glendinning and Sidorov focused on the case when $\a$ begins with $0$ and $\b$ begins with $1$, say $\a=0\a'$ and $\b=1\b'$. It can further be assumed that $\a'$ and $\b'$ begin with $1$ and $0$, respectively; {otherwise, either the word $01$ or the word $10$ cannot occur in $\Omega_{\a,\b}$ and so $\Omega_{\a,\b}$ is countable.}
Now for such $\a'$ and $\b'$, we have the relationship
\begin{equation} \label{eq:Omega-and-Sigma}
\Omega_{0\a',1\b'}=\bigcup_{n=0}^\f 0^n\Sigma_{\b',\a'}\cup \bigcup_{n=0}^\f 1^n\Sigma_{\b',\a'}\cup\{0^\f,1^\f\}
\end{equation}
(see \cite[Theorem 2.5 (vi)]{Komornik-Steiner-Zou-2022}). Hence $\Omega_{\a,\b}$ has the same cardinality and entropy as $\Sigma_{\b',\a'}$, except in trivial cases where the latter is finite and the former countably infinite. We can thus relate the set $\Omega_{\a,\b}$ to a symbolic survivor set $\Kt_\beta(t)$ for a suitable pair $(\beta,t)$, as explained in the previous subsection.

\subsection{Organization of the paper}

After developing some notation and conventions, Section \ref{sec:prelim} introduces the main building blocks of this article: Farey words and substitutions on such words. Associated with these concepts are several important sets, not least of which is the set $E$ of bases $\beta$ for which $\al(\beta)$ is {\em Sturmian}, i.e. aperiodic and balanced. Here we also define three collections of intervals: {\em basic intervals}, {\em Farey intervals}, and (higher order) {\em Lyndon intervals} arising from the substitutions on Farey words. We see that the interval $(1,2]$ is decomposed into several pieces, namely the set $E$; finitely renormalizable $\beta$'s; infinitely renormalizable $\beta$'s; and the basic intervals.

In Section \ref{sec:E_beta} we state and prove a useful characterization of the bifurcation set $\EE_\beta$. Section \ref{sec:key} then outlines the main ideas of the proofs of our results; we state a more technical theorem whose proof takes up much of the rest of the paper, and use it to prove Theorem \ref{thm:main}.

Section \ref{sec:beta-in-E} deals with the simplest case, when $\beta$ lies in the above-mentioned set $E$. (In fact we include here also the left endpoints of first-order basic intervals.) We show that here the subshift $\Kt_\beta(t_R)$ is transitive for every $\beta$-Lyndon interval $[t_L,t_R]$, and that the $\beta$-Lyndon intervals are dense in $[0,1-1/\beta]$. Section \ref{sec:right-endpoints} then proves the same facts for the right endpoints of first-order basic intervals, which require a somewhat separate argument but behave much more like points in $E$ than like points in the interiors of basic intervals.

Section \ref{sec:relative-exceptional} treats the case of finitely renormalizable $\beta$; that is, those $\beta$'s that can be renormalized to a point of $E$ in a finite number of steps. In this case we show inductively that $\Kt_\beta(t_R)$ is transitive only for $t_R$ below a certain threshold, but for $t_R$ above this threshold there still exists a transitive subshift of full entropy containing the sequence $b(t_R,\beta)$. We show furthermore that the $\beta$-Lyndon intervals remain dense in this case.

Section \ref{sec:comparison} considers the cardinality of $\EE_\beta\backslash\BB_\beta$ when $\beta$ is an endpoint of a basic interval or one of the special Lyndon intervals; the main result here is a crucial element of the proof of Theorem \ref{thm:size-of-E-minus-B}.

Section \ref{sec:E-infinity} deals with the infinitely renormalizable case; the results are much like those in Section \ref{sec:relative-exceptional}, but the induction argument proceeds slightly differently.

Sections \ref{sec:basic-interiors}, \ref{sec:basic-interval-proof} and \ref{sec:higher-order-basic}, the most complex of the paper, deal with the interiors of basic intervals. Section \ref{sec:basic-interiors} introduces a collection of intervals of $t$-values called {\em non-transitivity windows}, inside of which $\Kt_\beta(t_R)$ is not transitive and does not contain a transitive subshift of full entropy containing the sequence $b(t_R,\beta)$. These intervals are generated by a sequence $(\v_k)$ of Lyndon words which are extracted from the sequence $\al(\beta)$. Most of this section is devoted to proving several necessary facts about these words. Section \ref{sec:basic-interval-proof} then proves the main result, which is that (for $\beta$ in a first-order basic interval), $\Kt_\beta(t_R)$ is transitive if and only if $t_R$ does not lie in any non-transitivity window, and the entropy of $\Kt_\beta(t)$ is constant in each non-transitivity window. Section \ref{sec:higher-order-basic} then extends this result to the higher order basic intervals, by combining the methods of the previous two sections with the renormalization technique from Section \ref{sec:relative-exceptional}.

Section \ref{sec:gaps} gives a brief treatment of gaps between $\beta$-Lyndon intervals; its main purpose is to show that for certain values of $\beta$, the difference $\EE_\beta\backslash\BB_\beta$ is countably infinite, thus establishing part of Theorem \ref{thm:size-of-E-minus-B}.

Section \ref{sec:EBLI} develops several properties of extended $\beta$-Lyndon intervals (EBLIs), and proves Theorem \ref{thm:entropy-plateaus} and Corollary \ref{cor:symbolic-plateaus}. Finally, Section \ref{sec:other-proofs} gives the proofs of the remaining theorems.

\subsection{List of notation}

\begin{itemize}
\item $T_\beta(x):=\beta x\!\!\mod 1$ is the $\beta$-transformation on $[0,1)$,
\item $K_\beta(t):=\{x\in[0,1): T_\beta^n(x)\geq t\ \forall n\geq 0\}$,
\item {$M_\beta:=\lceil\beta\rceil-1$},
\item {$A_\beta:=\{0,1,\dots,M_\beta\}$},
\item $\pi_\beta((c_i)):=\sum_{i=1}^\f c_i\beta^{-i}, \quad (c_i)\in{A_\beta}^\N$,
\item $\al(\beta)$ is the quasi-greedy expansion of $1$ in base $\beta$,
\item $b(t,\beta)$ is the greedy expansion of $t$ in base $\beta$,
\item $\si:{A_\beta}^\N\to{A_\beta}^\N$ is the left shift map,
\item $\Sigma_\beta:=\{\z\in{A_\beta}^\N: \sigma^n(\z)\lle \alpha(\beta)\ \forall n\geq 0\}$ is the $\beta$-shift,
\item $\Sigma_{\a,\b}:=\{\z\in{A_\beta}^\N: \a\lle\si^n(\z)\lle \b\ \forall n\geq 0\}$,
\item $\mathcal{K}_\beta(t):=\{\z\in{A_\beta}^\N: b(t,\beta)\lle \sigma^n(\z)\prec \alpha(\beta)\ \forall n\geq 0\}$,
\item $\Kt_\beta(t):=\{\z\in{A_\beta}^\N: b(t,\beta)\lle \sigma^n(\z)\lle \alpha(\beta)\ \forall n\geq 0\}$,
\item $h(\mathcal{X})$ is the topological entropy of the subshift $\mathcal{X}$,
\item $\cL(\mathcal{X})$ is the language of the subshift $\XX$; i.e. the set of all finite words that occur in some sequence in $\XX$,
\item $\EE_\beta:=\{t\in[0,1): K_\beta(t')\neq K_\beta(t)\ \forall\,t'>t\}$,
\item $\EE_\beta^+:=\{t\in[0,1): T_\beta^n(t)\geq t\ \forall n\geq 0\}=\EE_\beta$,
\item $\BB_\beta:=\set{t\ge 0: \dim_H K_\beta(t')<\dim_H K_\beta(t)~\forall t'>t}$,
\item $\F$ is the set of all Farey words of length at least 2,
\item {$F_e$ is the set of all extended Farey words},
\item $L^*$ is the set of all Lyndon words {in $\{0,1\}^*$} of length at least 2,
\item {$L_e$ is the set of all Lyndon words in $\N_0^*$ except $0$,}
\item $\L(\s)$ is the lexicographically largest cyclic permutation of $\s$,
\item $w_1\dots w_m^+:=w_1\dots w_{m-1}(w_m+1)$ if {$w_m<M_\beta$},
\item $w_1\dots w_m^-:=w_1\dots w_{m-1}(w_m-1)$ if {$w_m\geq 1$},
\item $\beta_\ell^{\s}$ is the base such that $\alpha(\beta_\ell^{\s})=\L(\s)^\f$,
\item $\beta_r^{\s}$ is the base such that $\alpha(\beta_r^{\s})=\L(\s)^+\s^\f$,
\item $\beta_*^{\s}$ is the base such that $\alpha(\beta_*^{\s})=\L(\s)^+\s^-\L(\s)^\f$,
\item $E:={(1,\f)}\backslash \bigcup_{s\in\F} [\beta_\ell^{\s},\beta_r^{\s}]$,
\item $E_L:={(1,\f)}\backslash \bigcup_{s\in\F} (\beta_\ell^{\s},\beta_r^{\s}]$,
\item $U_0$ is the substitution $0\mapsto 0$, $1\mapsto 01$,
\item $U_1$ is the substitution $0\mapsto 01$, $1\mapsto 1$,
\item $\tau(\beta):=\min\{t>0: \dim_H K_\beta(t)=0\}$,
\item $\mathcal{T}_R(\beta)$ is the set of all right endpoints $t_R$ of $\beta$-Lyndon intervals $[t_L,t_R]$ in $[0,\tau(\beta)]$,
\item $\s\bullet\r:=\Phi_\s(\r)$ is the substitution defined in \eqref{eq:substitution},
\item $\La_k:=\set{\cs=\s_1\bullet\s_2\bullet\cdots\bullet\s_k:\ {\s_1\in F_e, \ \s_i\in\F\textrm{ for any }2\le i\le k}}$,
\item $\La:=\bigcup_{k=1}^\f \La_k$,
\item $I^\cs=[\beta_\ell^\cs,\beta_*^\cs]$ is the basic interval generated by $\cs\in\La$,
\item $J^\cs=[\beta_\ell^{\cs},\beta_r^{\cs}]$ is the Lyndon interval generated by $\cs\in\La$,
\item $\Ga(\cs):=\{\z\in{A_\beta}^\N: \cs^\f\lle \sigma^n(\z)\lle \L(\cs)^\f\ \ \forall n\geq 0\}$.
\end{itemize}

\section{Preliminaries} \label{sec:prelim}

First we introduce some notation and conventions. By a \emph{word} we mean a finite string of {elements of $A_\beta$}. Let {$A_\beta^*$} be the set of all words over the alphabet {$A_\beta$} together with the empty word $\epsilon$. For a word $\c\in{A_\beta}^*$ we denote its length by $|\c|$, and for a digit $a\in{A_\beta}$ we denote by $|\c|_a$ the number of occurrences of $a$ in the word $\c$. The concatenation of two words $\mathbf c=c_1\ldots c_m$ and $\mathbf d=d_1\ldots d_n$ in ${A_\beta}^*$ is denoted by $\mathbf{cd}=c_1\ldots c_md_1\ldots d_n$. Similarly, $\mathbf c^n$ denotes the $n$-fold concatenation of $\mathbf c$ with itself, and $\mathbf c^\f$ denotes the periodic sequence with period block $\mathbf c$. If $\mathbf c=c_1\ldots c_m$ with ${c_m<M_\beta}$, then we define $\mathbf c^+:=c_1\ldots c_{m-1}(c_m+1)$; and if $\mathbf c=c_1\ldots c_m$ with $c_m\geq 1$, we set $\mathbf c^-:=c_1\ldots c_{m-1}{(c_m-1)}$.

Throughout the paper we will use the lexicographical order `$\prec, \lle, \succ$' or `$\lge$' between sequences and words. For example, for two sequences $(c_i), (d_i)\in{A_\beta}^\N$, we say $(c_i)\prec (d_i)$ if $c_1<d_1$, or there exists $n\in\N$ such that $c_1\ldots c_n=d_1\ldots d_n$ and $c_{n+1}<d_{n+1}$. For two words $\mathbf c, \mathbf d$, we say $\mathbf c\prec \mathbf d$ if $\mathbf c {M_\beta}^\f\prec \mathbf d 0^\f$. Finally, for a word $\mathbf{c}$ and an infinite sequence $\mathbf{d}$, we say $\mathbf c\prec \mathbf d$ if $\mathbf c {M_\beta}^\f\prec \mathbf d$, and define $\mathbf c\succ \mathbf d$ similarly. Thoughout this paper, when discussing words and sequences, adjectives like {\em smallest}, {\em greater}, etc. will always mean smallest, greater, etc. with respect to the lexicographical order.

Recall that $\alpha(\beta)$ is the quasi-greedy expansion of $1$ in base $\beta$. The following useful result is well known (cf.~\cite{Baiocchi_Komornik_2007}).

\begin{lemma} \label{lem:quasi-greedy expansion-alpha-q}
{Let $k\in\N$. The restriction of the map $\beta\mapsto\al(\beta)$ to $(k,k+1]$ is an increasing bijection from $\beta\in (k,k+1]$ to the set of sequences $\a=\al_1\al_2\dots\in\set{0,1,\dots,k}^\N$ such that $\al_1=k$ and}
\[
0^\f\prec\si^n((\al_i))\lle (\al_i)\quad\forall n\ge 1.
\]
\end{lemma}

We recall the following characterization of greedy $\beta$-expansions due to Parry \cite{Parry_1960}.

\begin{lemma} \label{lem:greedy-expansion}
Let $\beta{>1}$. The map $t\mapsto b(t,\beta)$ is an increasing bijection from $[0,1)$ to
\[
\big\{\z\in {A_\beta}^\N: \si^n(\z)\prec \al(\beta)~\forall n\ge 0\big\}.
\]
Furthermore, the map $t\mapsto b(t,\beta)$ is right-continuous everywhere in $[0,1)$ with respect to the order topology in ${A_\beta}^\N$.
\end{lemma}

\subsection{Farey words, Farey intervals and the set $E$}

 Farey words have attracted much attention in the literature due to their intimate connection with rational rotations on the circle (see \cite[Chapter 2]{Lothaire-2002}) and their one-to-one correspondence with the rational numbers in $[0,1]$ (see {(\ref{eq:kk-6})} below). In the following we adopt the definition from \cite{Carminati-Isola-Tiozzo-2018}.

 First we recursively define a sequence of ordered sets $F_n, n=0,1,2,\ldots$. Let $F_0=(0,1)$; and for $n\ge 0$ the ordered set $F_{n+1}=(\v_1, \ldots, \v_{2^{n+1}+1})$ is obtained from $F_{n}=(\w_1,\ldots, \w_{2^n+1})$ by
 \[
 \left\{
 \begin{array}
   {lll}
   \v_{2i-1}=\w_i&\textrm{for}& 1\le i\le 2^{n}+1,\\
   \v_{2i}=\w_i\w_{i+1}&\textrm{for}& 1\le i\le 2^n.
 \end{array}\right.
 \]
In other words, $F_{n+1}$ is obtained from $F_n$ by inserting for each $1\le j\le 2^{n}$ the new word $\w_j\w_{j+1}$ between the two neighboring words $\w_j$ and $\w_{j+1}$. So,
 \begin{equation*} 
 \begin{split}
 &F_1=(0,01,1),\qquad F_2=(0,001,01,011,1),\\
  F_3&=(0,0001,001,00101,01,01011,011,0111,1), \quad \dots
 \end{split}
 \end{equation*}
Note that for each $n\ge 0$ the ordered set $F_n$ consists of $2^n+1$ words which are listed from the left to the right in lexicographically increasing order.
We call $\w\in\set{0,1}^*$ a \emph{Farey word} if $\w\in F_n$ for some $n\ge 0$. We denote by $F:=\bigcup_{n=1}^\f F_n$ the set of all Farey words, and by $\F:=F\backslash\{0,1\}$ the set of all Farey words of length at least two. As shown in \cite[Proposition 2.3]{Carminati-Isola-Tiozzo-2018}, the set $F$ can be bijectively mapped to $\mathbb{Q}\cap[0,1]$ via the map
\begin{equation} \label{eq:kk-6}
\xi: F\to\mathbb Q\cap[0,1];\quad \s\mapsto\frac{|\s|_1}{|\s|}.
\end{equation}
So, $\xi(\s)$ is the frequency of the digit $1$ in $\s$.



 The Farey words can also be obtained recursively via the substitutions
 \begin{equation} \label{eq:subs-01}
 U_0:\left\{\begin{array}{lll}
 0&\mapsto&0\\
 1&\mapsto&01,
 \end{array}\right.\quad\textrm{and}\quad U_1:\left\{
 \begin{array}{lll}
 0&\mapsto&01\\
 1&\mapsto&1.
 \end{array}\right.
 \end{equation}
We extend the definitions of $U_0$ and $U_1$ to $\set{0,1}^*$ and $\set{0,1}^\N$ via homomorphism; that is, $U_0(c_1\dots c_m)=U_0(c_1)\dots U_0(c_m)$, etc. Observe that the maps $U_0$ and $U_1$, when viewed as functions on $\set{0,1}^\N$, are both strictly increasing with respect to the lexicographical order.


The following lemma can be deduced from \cite[Propositions 2.3 and 2.9]{Carminati-Isola-Tiozzo-2018}. (See \cite[Section 2]{Allaart-Kong-2021} for more details.)

\begin{lemma} \label{lem:characterization-Farey-words}
   Let $\s\in\F$; then one of the following holds:
   \begin{itemize}
    \item[{\rm(i)}] $\s=01$;
		\item[{\rm(ii)}] $\s=U_0(\hat{\s})$ for some Farey word $\hat{\s}\in\F$;
    \item[{\rm(iii)}] $\s=U_1(\hat{\s})$ for some Farey word $\hat{\s}\in\F$.
	\end{itemize}
	Vice versa, if $\s$ is a Farey word, then both $U_0(\s)$ and $U_1(\s)$ are Farey.
\end{lemma}

Observe that for any word $\w$ which begins with $0$ and ends with $1$, both $U_0(\w)$ and $U_1(\w)$ begin with $0$ and end with $1$. Furthermore, for any word $\w$, the word $U_0(\w)$ does not contain $11$ as a subword, and $U_1(\w)$ does not contain $00$.

For a word $\mathbf c=c_1\ldots c_m\in\set{0,1}^*$, let $\S(\mathbf c)$ and $\L(\mathbf c)$ be the smallest and largest cyclic permutations of $\mathbf c$, respectively; that is, the smallest and largest words among
\[
c_1c_2\ldots c_m, \quad c_2\ldots c_m c_1,\quad c_3\ldots c_m c_1c_2,\quad \cdots,\quad c_mc_1\ldots c_{m-1}.
\]
The following properties of Farey words are well known (see, e.g., \cite[Proposition 2.5]{Carminati-Isola-Tiozzo-2018}).

\begin{lemma}  \label{lem:Farey-property}
Let $\s=s_1\ldots s_m\in\F$. Then
\begin{enumerate}[{\rm (i)}]
\item $\S(\s)=\s$ and $\L(\s)=s_ms_{m-1}\ldots s_1$.
\item ${\s^-}$ is a palindrome; that is, $s_1\ldots s_{m-1}(s_m-1)=(s_m-1)s_{m-1}s_{m-2}\ldots s_1$.
\end{enumerate}
\end{lemma}

{Next, we define a map $\theta$ on finite words and infinite sequences that simply increments each coordinate by 1. Thus,
\begin{equation} \label{eq:theta}
\theta(c_1,\dots,c_n):=(c_1+1,\dots,c_n+1), \qquad \theta(c_1,c_2,\dots):=(c_1+1,c_2+1,\dots).
\end{equation}
We define the {\em extended Farey set} $\cF_e$ by
\[
\cF_e:=\{\theta^k(\w): \w\in \F\cup\{1\}, k=0,1,2,\dots\}.
\]
Thus, $\cF_e$ contains all the Farey words except $0$, and in addition, it contains all words derived from such Farey words by incrementing all digits by the same amount. For example, applying $\theta$ repeatedly to the Farey word $001$ yields the words $112, 223, 334, \dots$ in $\cF_e$.

The map $\theta$ induces a function $\phi:(1,\f)\to(2,\f)$ defined by
\[
\phi(\beta):=\al^{-1}\circ \theta \circ \al(\beta).
\]
Note by Lemma \ref{lem:quasi-greedy expansion-alpha-q} that the map $\phi$ is well defined, and it is strictly increasing.
Observe that for each $k\in\N$, $\phi$ maps the interval $(k,k+1]$ into $(k+1,k+2]$. 
However, $\phi$ should not be confused with the map $\beta\mapsto \beta+1$. For instance,
\[
\phi\left(\frac{1+\sqrt{5}}{2}\right)=\al^{-1}\circ\theta((10)^\f)=\al^{-1}((21)^\f)=1+\sqrt{3}.
\]
In fact, for each $k\in\N$ we have
\[
 \lim_{\beta\searrow k} \phi(\beta)=\al^{-1}\circ\theta(k0^\f)=\al^{-1}((k+1)1^\f)=\frac{k+2+\sqrt{k^2+4}}{2},
\]
and hence, since $\phi$ is clearly increasing, $\phi$ maps $(k,k+1]$ into $(\frac{k+2+\sqrt{k^2+4}}{2},k+2]$. In particular, $\phi$ does not map $(k,k+1]$ {\em onto} $(k+1,k+2]$.
}

We recall from \cite{Kalle-Kong-Langeveld-Li-18} {and \cite{Allaart-Kong-2024}} that a {\em Farey interval} is an interval $J^\s=[\beta_\ell^{\s},\beta_r^{\s}]$ defined by
\[
\alpha(\beta_\ell^{\s})=\L(\s)^\f, \qquad \alpha(\beta_r^{\s})=\L(\s)^+\s^\f,
\]
where $\s\in{\cF_e}$. We define the sets
\[
E:={(1,\f)}\backslash \bigcup_{\s\in\cF_e} J^\s={(1,\f)}\backslash \bigcup_{\s\in\cF_e} [\beta_\ell^{\s},\beta_r^{\s}]
\]
and
\begin{equation} \label{eq:bifurction-E-L}
E_L:={(1,\f)}\backslash \bigcup_{\s\in\cF_e} (\beta_\ell^{\s},\beta_r^{\s}].
\end{equation}
We let $\overline{E}$ denote the topological closure of $E$, and observe that
\[
\overline{E}=\overline{E_L}={{[1,\f)}}\backslash \bigcup_{\s\in\cF_e} (\beta_\ell^{\s},\beta_r^{\s}).
\]

{
\begin{lemma} \label{lem:two-digits}
Let $k\in\N$ and $\beta\in \overline{E}\cap(k,k+1]$. Then $\al(\beta)\in \{k,k+1\}^\N$.
\end{lemma}

\begin{proof}
This was proved in \cite[Lemma 5.2]{Allaart-Kong-2024} for $\beta\in E\cap(k,k+1]$. The proof easily extends to $\overline{E}$.
\end{proof}

The next lemma is an easy consequence of the previous one; see \cite[Lemma 5.3]{Allaart-Kong-2024}.

\begin{lemma} \label{prop:image-of-E}
For each $k\in\N_{\geq 2}$,
\[
E\cap(k,k+1]=\phi^{k-1}(E\cap(1,2]).
\]
\end{lemma}
}

A word or sequence $\w$ is said to be {\em balanced} if for any two subwords $\u$ and $\v$ of $\w$ with equal length, {and for any digit $d$, $\big||\u|_d-|\v|_d\big|\leq 1$.}
All {extended} Farey words are balanced, and any subword of a balanced word or sequence is again balanced.

\begin{lemma} \label{lem:balanced-expansion}
For each $\beta\in \overline{E}{\cap{(1,\f)}}$, the sequence $\al(\beta)$ is balanced.
\end{lemma}

\begin{proof}
{Take first $\beta\in(1,2]$.}
Write $\alpha(\beta)=\alpha_1\alpha_2\dots\in\{0,1\}^\N$. 
Since $\beta\in \overline{E}$, $\alpha(\beta)$ is the component-wise limit of the sequence $(\L(\s_n)^\f)$ for some sequence $(\s_n)$ of Farey words. Without loss of generality we may assume that the length of $\s_n$ strictly increases with $n$. Then for each $k\in\N$ there is an integer $N_k$ such that $\alpha_1\dots\alpha_k$ is a prefix of $\L(\s_n)$ for all $n\geq N_k$. Since Farey words are balanced, it follows that $\al(\beta)$ is balanced as well.

{For $\beta>2$, the result now follows from Lemma \ref{prop:image-of-E}.}
\end{proof}

{The remainder of this section is specific to $\beta\in(1,2]$.}

\begin{lemma} \label{lem:eta-and-xi}
Let $\s$ and $\hat{\s}$ be Farey words and suppose $\s=U_0(\hat{\s})$ or $\s=U_1(\hat{\s})$.
\begin{enumerate}[{\rm(i)}]
\item If $\s=U_0(\hat{\s})$, then
\[
0\L(\s)=U_0(\L(\hat{\s}))0, \qquad 0\L(\s)^+=U_0(\L(\hat{\s})^+), \qquad \s^-=U_0(\hat{\s}^-)0.
\]
\item If $\s=U_1(\hat{\s})$, then
\[
\L(\s)1=1U_1(\L(\hat{\s})), \qquad \L(\s)^+=1U_1(\L(\hat{\s})^+), \qquad \s^-1=U_1(\hat{\s}^-).
\]
\end{enumerate}
\end{lemma}

\begin{proof}
We prove (i); the proof of (ii) is similar.

Assume $\s=U_0(\hat{\s})$. By Lemma \ref{lem:Farey-property}, $\L(\hat{\s})$ is the word obtained from $\hat{\s}$ by changing the first $0$ to a $1$ and the last $1$ to a $0$. Hence $U_0(\L(\hat{\s}))$ is obtained from $U_0(\hat{\s})$ by replacing the first block, $U_0(0)=0$, with $U_0(1)=01$, and the last block, $U_0(1)=01$ with $U_0(0)=0$. On the other hand, $\L(\s)=\L(U_0(\hat{\s}))$ is obtained from $U_0(\hat\s)$ by replacing the first digit $0$ with $1$ and the last block $01$ with $00$. It follows that
\begin{equation}\label{eq:lem18-1}
0\L(\s)=U_0(\L(\hat{\s}))0.
\end{equation}
Next, since $\L(\hat{\s})$ ends in a $0$ and $U_0(1)=01=U_0(0)1$, by (\ref{eq:lem18-1}) we have
\[
U_0(\L(\hat{\s})^+)=U_0(\L(\hat{\s}))1=U_0(\L(\hat{\s}))0^+=0\L(\s)^+.
\]
Finally, $\hat{\s}^-$ is obtained from $\hat{\s}$ by changing the last digit 1 to a 0, so $U_0(\hat{\s}^-)$ is obtained from $U_0(\hat{\s})$ by removing the last 1, whereas $\s^-$ is obtained from $U_0(\hat{\s})$ by changing the last 1 to a 0. Thus, $\s^-=U_0(\hat{\s}^-)0$.
\end{proof}

\begin{lemma} \label{lem:not-0-1-repeated}
\begin{enumerate}[{\rm(i)}]
\item Let $\beta\in \overline{E}\cap(1,2]$, and assume $\al(\beta)\not\in\{(10)^\f,11(01)^\f\}$. Then $\al(\beta)$ does not end in $(01)^\f$. 
\item If $\beta\in E_L{\cap(1,2]}$ and $\al(\beta)\neq (10)^\f$, then $\al(\beta)$ does not end in $(01)^\f$.
\end{enumerate}
\end{lemma}

\begin{proof}
{ (i) Suppose, by way of contradiction, that $\al(\beta)$ does end in $(10)^\f$. Then, since $\al(\beta)\neq (10)^\f$,  Lemma \ref{lem:quasi-greedy expansion-alpha-q} implies that $\al(\beta)\succ (10)^\f$. Hence, $\al(\beta)$ begins with $11$. Since $\al(\beta)$ is balanced by Lemma \ref{lem:balanced-expansion}, it then cannot contain the word $00$. Since $\al(\beta)\neq 11(01)^\f$, it must be the case that $\al(\beta)$ contains a block $11(01)^k1$ for some $k\geq 0$. This block has length $2k+3$ and contains $k$ zeros, whereas the block $(01)^{k+1}0$ also has length $2k+3$ but contains $k+2$ zeros. This contradicts that $\al(\beta)$ is balanced.

(ii) If $\beta\in E_L{\cap(1,2]}$, then $\al(\beta)\neq 11(01)^\f$ because $11(01)^\f=\al(\beta_r^\s)$ with $\s=01$. Apply part (i).}
\end{proof}


\begin{lemma} \label{lem:alpha-renormalizing}
{
Let $\beta\in \overline{E}\cap(1,2]$ such that $\al(\beta)\neq 11(01)^\f$. Then there is a base $\hat{\beta}\in \overline{E}{\cap(1,2]}$ such that
\[
0\al(\beta)=U_0(\al(\hat{\beta})) \qquad \mbox{or} \qquad 0\al(\beta)=U_1(0\al(\hat{\beta})).
\]
Furthermore, if $\beta\in E_L{\cap(1,2]}$, then $\hat{\beta}\in E_L{\cap(1,2]}$.
}
\end{lemma}

Note that in the second case, it is possible that $\beta=\hat{\beta}$; this happens when $\beta=2$, in which case $\al(\beta)=\al(\hat{\beta})=1^\f$.

\begin{proof}
If $\al(\beta)=(10)^\f$, then $0\al(\beta)=U_0(1^\f)=U_0(\al(2))$, and $2\in E_L\subseteq \overline{E}$. So assume $\al(\beta)\neq(10)^\f$. Then $\al(\beta)$ does not end in $(01)^\f$ by Lemma \ref{lem:not-0-1-repeated}.
Furthermore, since $\al(\beta)$ is balanced by Lemma \ref{lem:balanced-expansion}, it either does not contain the word $00$ or does not contain the word $11$.

\medskip
{\em Case 1}: $\al(\beta)$ does not contain $11$. Then $0\alpha(\beta)=U_0(\mathbf{d})$ for some sequence $\mathbf{d}=d_1d_2\dots\in\{0,1\}^\N$ with $d_1=1$. Since $\al(\beta)$ does not end in $0^\f=U_0(0^\f)$, $\mathbf{d}$ does not end in $0^\f$ either. We claim that $\si^k(\mathbf{d})\lle \mathbf{d}$ for all $k\geq 1$. Fix $k\in\N$. We can choose $n\geq 1$ so that
\[
U_0(\si^k(\mathbf{d}))=\si^n(U_0(\mathbf{d}))=\si^n(0\alpha(\beta)),
\]
and moreover, $\si^{n}(0\alpha(\beta))$ begins with a $0$. Hence,
\[
U_0(\si^k(\mathbf{d}))=\si^{n}(0\alpha(\beta))=0\si^n(\al(\beta))\lle 0\alpha(\beta)=U_0(\mathbf{d}),
\]
and since $U_0$ is strictly increasing, it follows that $\si^k(\mathbf{d})\lle \mathbf{d}$. Thus, by Lemma \ref{lem:quasi-greedy expansion-alpha-q}, $\mathbf{d}=\alpha(\hat{\beta})$ for some base $\hat{\beta}$. In other words,
\begin{equation} \label{eq:quasi-greedy-relationship-b}
0\alpha(\beta)=U_0(\alpha(\hat{\beta})).
\end{equation}

We next verify that $\hat{\beta}\in \overline{E}$. Suppose, by way of contradiction, that $\hat{\beta}\in (\beta_\ell^{\hat{\s}},\beta_r^{\hat{\s}})$ for some Farey word $\hat{\s}$. That is,
\begin{equation} \label{eq:in-Farey-interval}
\L(\hat{\s})^\f\prec \alpha(\hat{\beta}) \prec \L(\hat{\s})^+\hat{\s}^\f.
\end{equation}
Set $\s:=U_0(\hat{\s})$. Then $\s$ is Farey, and since $U_0$ is strictly increasing, by (\ref{eq:quasi-greedy-relationship-b}) and Lemma \ref{lem:eta-and-xi} (i) we obtain
\[
0\L(\s)^\f=U_0(\L(\hat{\s})^\f)\prec U_0(\alpha(\hat{\beta}))=0\alpha(\beta)\prec U_0\big(\L(\hat{\s})^+\hat{\s}^\f\big)=0\L(\s)^+\s^\f,
\]
Thus, $\beta\in(\beta_\ell^{\s},\beta_r^{\s})$, contradicting that $\beta\in \overline{E}$. We conclude that $\hat{\beta}\in \overline{E}$.

\medskip
{\em Case 2}: $\al(\beta)$ does not contain $00$. Then $0\al(\beta)=U_1(0\mathbf{d})$ for some sequence $\mathbf{d}=d_1d_2\dots\in\{0,1\}^\N$ with $d_1=1$. Since $\al(\beta)$ does not end in $(01)^\f=U_1(0^\f)$, $\mathbf{d}$ does not end in $0^\f$.

We next show that $\si^k(\mathbf{d})\lle \mathbf{d}$ for all $k\geq 1$. {This is obvious if $\mathbf{d}=1^\f$. So assume $\mathbf{d}$ contains at least one $0$, and fix $k\in\N$.} Since $\mathbf{c}\lle 1\mathbf{c}$ for any sequence $\mathbf{c}$, {we may assume without loss of generality that $d_k=0$.} 
We can choose $n\geq 1$ so that $U_1(0\si^k(\mathbf{d}))=U_1(\si^{k-1}(\mathbf{d}))=\si^n(0\alpha(\beta))$, and this expression begins with a $U_1(0)=01$. Hence,
\[
U_1(0\si^k(\mathbf{d}))=\si^{n}(0\alpha(\beta))=0\si^{n-1}(\alpha(\beta))\lle 0\alpha(\beta)=U_1(0\mathbf{d}).
\]
Since $U_1$ is strictly increasing, it follows that $\si^k(\mathbf{d})\lle \mathbf{d}$, as desired. Thus, $\mathbf{d}=\alpha(\hat{\beta})$ for some $\hat{\beta}\in(1,2]$, and then
$0\alpha(\beta)=U_1(0\mathbf d)=U_1(0\alpha(\hat{\beta}))$.

We next verify that $\hat{\beta}\in \overline{E}$. Suppose, by way of contradiction, that $\hat{\beta}\in (\beta_\ell^{\hat{\s}},\beta_r^{\hat{\s}})$ for some Farey word $\hat{\s}$. That is, \eqref{eq:in-Farey-interval} holds.
Set $\s:=U_1(\hat{\s})$. Then $\s$ is Farey, and since $U_1$ is strictly increasing, we obtain from \eqref{eq:in-Farey-interval} and Lemma \ref{lem:eta-and-xi} (ii) that
\begin{align*}
0\L(\s)^\f&=01U_1(\L(\hat{\s})^\f)=U_1(0\L(\hat{\s})^\f)\\
&\prec U_1(0\al(\hat{\beta}))=0\al(\beta)\prec U_1(0\L(\hat\s)^+\hat\s^\f)\\
&=01U_1(\L(\hat\s)^+\hat\s^\f)=0\L(\s)^+\s^\f.
\end{align*}
Hence $\L(\s)^\f\prec \alpha(\beta)\prec \L(\s)^+\s^\f$. But then $\beta\in(\beta_\ell^{\s},\beta_r^{\s})$, contradicting that $\beta\in \overline{E}$. Thus, $\hat{\beta}\in \overline{E}$.

That $\beta\in E_L$ implies $\hat{\beta}\in E_L$ follows in the same way.
\end{proof}

The last result in this subsection may be of independent interest. Recall that a sequence is {\em Sturmian} if it is balanced and not eventually periodic (cf.~\cite{Lothaire-2002}).

\begin{proposition} \label{prop:Sturmian}
Let $\beta\in(1,2)$.
\begin{enumerate}[{\rm(i)}]
\item $\al(\beta)$ is balanced if and only if $\beta\in\overline{E}$.
\item $\al(\beta)$ is Sturmian if and only if $\beta\in E$.
\end{enumerate}
\end{proposition}

\begin{proof}
(i) We have already seen from Lemma \ref{lem:balanced-expansion} that, if $\beta\in\overline{E}$, then $\al(\beta)$ is balanced. Conversely, suppose $\beta\not\in\overline{E}$. Then $\beta_\ell^{\s}<\beta<\beta_r^{\s}$ for some Farey word $\s=s_1\dots s_m$, and so $\L(\s)^\f\prec\al(\beta)\prec \L(\s)^+\s^\f$. By Lemma \ref{lem:quasi-greedy expansion-alpha-q} this implies that $\al(\beta)$ begins with $\L(\s)^+$ and there is an $n\in\N$ such that $\al_{n+1}\dots\al_{n+m}\prec\s$, where $(\al_i):=\al(\beta)$. Let $k$ be the smallest integer such that $\al_{n+k}\neq s_k$. 
Then $\al_{n+1}\dots\al_{n+k}=s_1\dots s_k^-=0s_2\dots s_k^-$. By Lemma \ref{lem:Farey-property}, $\L(\s)^+$ is obtained from $\s$ by changing the first 0 to a 1, so $\L(\s)^+$ begins with $1s_2\dots s_k$. Thus, $\al_{n+1}\dots\al_{n+k}$ has two fewer 1's than $\al_1\dots\al_k=1s_2\ldots s_k$, and we conclude that $\al(\beta)$ is not balanced.

(ii) If $\al(\beta)$ is Sturmian, then it is balanced so $\beta\in\overline{E}$ by part (i). But $\beta\neq \beta_\ell^{\s}$ or $\beta_r^{\s}$ for any Farey word $\s$, because $\al(\beta_\ell^{\s})$ and $\al(\beta_r^{\s})$ are (eventually) periodic. Thus, $\beta\in E$.

Vice versa, take $\beta\in E$. Then $\al(\beta)$ is balanced. 
Suppose, by way of contradiction, that $\al(\beta)$ is eventually periodic. Since $\beta<2$, $\al(\beta)\neq 1^\f$. By cyclical shifting, we may assume that the minimal period block of $\al(\beta)$ is a Lyndon word $\w$. Note that $\al(\beta)\ne (10)^\f$, since otherwise $\beta=\beta_\ell^{01}\notin E$. By Lemma \ref{lem:not-0-1-repeated}, $\w\neq 01$. Since $\al(\beta)$ is balanced, so is $\w$, and therefore either $\w=U_0(\hat{\w})$ or $\w=U_1(\hat{\w})$ for some Lyndon word $\hat{\w}$, as $\w$ cannot contain both of the words $00$ and $11$.  (Note $\hat{\w}$ is also Lyndon because $U_0$ and $U_1$ are strictly increasing.) If $\w=U_0(\hat{\w})$, then $0\al(\beta)=0U_0(\al(\hat{\beta}))$ for some $\hat{\beta}\in E$, since Lemma \ref{lem:alpha-renormalizing} obviously holds also for $E$ instead of $E_L$. But then $\al(\hat{\beta})$ is eventually periodic with the shorter period block $\hat{\w}$. A similar argument holds when $\w=U_1(\hat{\w})$. We can now repeat the same argument over and over again, to obtain an infinite sequence of Lyndon words $\w_1,\w_2,\dots$, each one shorter than the last. But this is absurd. Therefore, $\al(\beta)$ is not eventually periodic.
\end{proof}

{
\begin{remark} \label{rem:no-converses}
{\rm
For $\beta>2$, the ``only if" part of both statements in Proposition \ref{prop:Sturmian} fails in general: For instance, take $\beta>2$ such that $\al(\beta)=(210)^\f$. Then $\al(\beta)$ is balanced, but $\beta\in(\beta_\ell^{(1)},\beta_r^{(1)})$ and hence $\beta\not\in\overline{E}$. By slightly perturbing $\al(\beta)$, one can similarly construct a base $\beta'>2$ such that $\al(\beta')$ is Sturmian but $\beta'\not\in E$. We leave this as a small exercise for the reader. 
}
\end{remark}
}

\subsection{The substitution operator} \label{subsec:substitution}
\ 
{Let $L^*$ denote the set of all Lyndon words in $\{0,1\}^*$ of length at least 2, and let $L_e$ denote the set of all Lyndon words in $\N_0^*$ except $0$.}

\begin{definition} \label{def:substitution}
For a Lyndon word $\s\in {L_e}$, we define the substitution map $\Phi_\s: \{0,1\}^\N\to {\N_0}^\N$ by
\begin{gather} \label{eq:block-map}
\begin{split}
\Phi_\s(0^{k_1}1^{l_1}0^{k_2}1^{l_2}\dots)=\s^-\L(\s)^{k_1-1}\L(\s)^+\s^{l_1-1}\s^-\L(\s)^{k_2-1}\L(\s)^+\s^{l_2-1}\dots,\\
\Phi_s(1^{k_1}0^{l_1}1^{k_2}0^{l_2}\dots)=\L(\s)^+\s^{k_1-1}\s^-\L(\s)^{l_1-1}\L(\s)^+\s^{k_2-1}\s^-\L(\s)^{l_2-1}\dots,
\end{split}
\end{gather}
where $1\leq k_i, l_i\leq \f$ for all $i$.
We allow one of the exponents $k_i$ or $l_i$ to take the value $+\f$, in which case we ignore the remainder of the sequence. We denote the range of $\Phi_\s$ by $X(\s)$.
We define $\Phi_\s(\r)$ for a finite word $\r$ in the same way, and set $X^*(\s):=\Phi_\s\big(\{0,1\}^*\big)$.
\end{definition}

The order in which the four blocks $\s,\s^-,\L(\s)$ and $\L(\s)^+$ can appear in $\Phi_\s(\r)$ is illustrated in Figure \ref{fig:directed-graph}.

   \begin{figure}[h!]
  \centering
  \begin{tikzpicture}[->,>=stealth',shorten >=1pt,auto,node distance=3cm, semithick,scale=3]

  \tikzstyle{every state}=[minimum size=0pt,fill=none,draw=black,text=black]

  \node[state] (A)                    { $\s$};
  \node[state]         (B) [ right of=A] {$\s^-$ };
  \node[state]         (C) [ above of=A] {$\L(\s)^+$};
  \node[state] (E)[left of=C]{Start-$1$};
  \node[state](D)[right of=C]{$\L(\s)$};
  \node[state](F)[right of=B]{Start-$0$};

  \path[->,every loop/.style={min distance=0mm, looseness=25}]
  (E) edge[->] node{$1$} (C)
  (C) edge[->,left] node{$1$} (A)
  (C) edge[bend right,->,right] node{$0$} (B)
(D) edge[loop right,->] node{$0$} (D)
(D) edge[->,above] node{$1$} (C)
(A) edge[loop left,->,looseness=55] node{$1$} (A)
(A)edge[->,below] node{$0$} (B)
(B) edge[bend right,->,right] node{$1$} (C)
(B) edge[->,right] node{$0$} (D)
(F) edge[->] node{$0$} (B)
;
\end{tikzpicture}
\caption{The directed graph illustrating the map $\Phi_\s$.}
\label{fig:directed-graph}
\end{figure}

For example,
\[
\Phi_\s(0110^\f)=\s^-\L(\s)^+\s\s^-\L(\s)^\f, \qquad \Phi_\s(11100)=\L(\s)^+\s^2\s^-\L(\s).
\]

Now for any two words $\s\in L_e$ and $\r\in\set{0,1}^*$ we define the substitution operation
\begin{equation} \label{eq:substitution}
\s\bullet\r:=\Phi_\s(\r).
\end{equation}

\begin{example} \label{ex:1}
\begin{enumerate}[(a)]
\item Let $\s=01$ and $\r=011$. Then
  \[
   \s\bullet\r=\Phi_\s(\r)=\Phi_\s(011)={\s^-}\L(\s)^+\s=001101.
	\]
\item Let $\s=1$ and $\r=011$. Then $\s\bullet\r={\s^-}\L(\s)^+\s=021$. This shows that the operator $\Phi_1$ maps words from $\{0,1\}^*$ to words over the larger alphabet $\{0,1,2\}$. 
\item Similarly, we have, for example, $23\bullet 01=(23)^-\L(23)^+=2233$.
\end{enumerate}
\end{example}

Note that we have not defined expressions such as $01\bullet 021$; we shall have no need for them.

The following lemma collects properties of the map $\Phi_\s$ and the substitution operator $\bullet$. The proofs can be found in \cite[Section 3]{Allaart-Kong-2021}. To summarize, $\Phi_\s$ is increasing and commutes with the operators $\cdot^+,\cdot^-$ and $\L$; the set $\LL$ of Lyndon words is closed under $\bullet$, and $\bullet$ is associative.

\begin{lemma} \label{lem:substitution-properties}
Let $\s\in {L_e}$.
\begin{enumerate}[{\rm(i)}]
\item The map $\Phi_\s$ is strictly increasing on $\set{0,1}^\N$.
\item For any  word $\d=d_1\ldots d_k\in \{0,1\}^*$ with $k\ge 2$, we have
\[
\left\{
\begin{array}{lll}
\Phi_\s(\d^-)=\Phi_\s(\d)^-&\textrm{if}& d_k=1,\\
\Phi_\s(\d^+)=\Phi_\s(\d)^+&\textrm{if}& d_k=0.
\end{array}\right.
\]
\item For any two sequences $\c,\d\in\set{0,1}^\N$, we have the {equivalences}
\[
\si^n(\c)\prec \d~\forall n\ge 0\quad{\Longleftrightarrow}\quad \si^n(\Phi_\s(\c))\prec \Phi_\s(\d)~\forall n\ge 0
\]
and
\[
\si^n(\c)\succ \d~\forall n\ge 0\quad{\Longleftrightarrow}\quad \si^n(\Phi_\s(\c))\succ \Phi_\s(\d)~\forall n\ge 0.
\]
\item For any $\r\in L^*$, we have $\s\bullet\r\in {L_e}$ and $\L(\s\bullet\r)=\s\bullet\L(\r)$.
\item For any {$\r\in L_e$ and $\s,\t\in L^*$, we have $(\r\bullet\s)\bullet\t=\r\bullet(\s\bullet\t)$.}
\end{enumerate}
\end{lemma}

Observe that the map $\Phi_\s$ is not a homomorphism on $\{0,1\}^*$: For instance, $\Phi_\s(00)=\s^-\L(\s)\neq \s^-\s^-=\Phi_\s(0)\Phi_\s(0)$. Nonetheless, $\Phi_\s$ is a partial homomorphism in the sense described below.

Say a finite or infinite sequence of words $\b_1,\dots, \b_n$ or $\b_1,\b_2,\dots$ is {\em connectible} if for each $i$, the last digit of $\b_i$ differs from the first digit of $\b_{i+1}$. Thus, for instance, the sequence $1101, 00111$ is connectible whereas the sequence $11010, 0111$ is not. 

\begin{lemma} \label{lem:connectible}
{\rm (\cite[Section 3]{Allaart-Kong-2021})}
\begin{enumerate}[{\rm (i)}]
\item Let $\b_1,\b_2,\dots$ be a (finite or infinite) connectible sequence of words. Then for any $\s\in {L_e}$,
\[
\Phi_\s(\b_1\b_2\dots)=\Phi_\s(\b_1)\Phi_\s(\b_2)\dots.
\]
\item Let {$\s\in L_e$ and} $\r\in\LL$. Then $\Phi_\s(\r^\f)=\Phi_\s(\r)^\f$ and $\Phi_\s(\L(\r)^\f)=\Phi_\s(\L(\r))^\f$.
\end{enumerate}
\end{lemma}

{Note that (ii) is a consequence of (i), since $\r\in\LL$ implies that $\r$ begins with $0$ and ends with $1$, whereas $\L(\r)$ begins with $1$ and ends with $0$.}

Now for $k\in\N$ we define
\begin{equation} \label{eq:set-substitution-Farey-words}
\La_k:=\set{\cs=\s_1\bullet\s_2\bullet\cdots\bullet\s_k:\ {\s_1\in F_e, \s_2,\dots,\s_k\in\F}},
\end{equation}
and set
\[
\La:=\bigcup_{k=1}^\f \La_k.
\]
If $\cs\in\La_k$, we call $k$ the {\em degree} of $\cs$. Note that $F_e\subseteq\La\subseteq L_e$ by Lemma \ref{lem:substitution-properties} (iv). Both inclusions are proper, e.g. $001011=01\bullet 001\in\La\backslash F_e$ and $0010111\in L_e\backslash\La$.

\medskip
We recall some terminology and facts from \cite{Allaart-Kong-2021}.
For $\cs\in\La$, we call the interval $I^\cs:=[\beta_\ell^\cs,\beta_*^\cs]$ given implicitly by
\[
\al(\beta_\ell^\cs)=\L(\cs)^\f \qquad\mbox{and}\qquad \al(\beta_*^\cs)=\L(\cs)^+\cs^-\L(\cs)^\f,
\]
a {\em basic interval} generated by the Lyndon word $\cs$. We also define an interval $J^\cs:=[\beta_\ell^\cs,\beta_r^\cs]$, where $\beta_r^\cs$ is given by
\[
\al(\beta_r^\cs)=\L(\cs)^+\cs^\f.
\]
The interval $J^\cs$ is called a {\em Lyndon interval} generated by $\cs$.
For any $\cs\in\La$, the intervals $J^{\cs\bullet\r}, \r\in\F$ are pairwise disjoint and contained in $J^\cs\setminus I^\cs$; let
\begin{equation} \label{eq:relative-E}
E^\cs:=(J^\cs\setminus I^\cs)\setminus\bigcup_{\r\in\F}J^{\cs\bullet\r}.
\end{equation}
It is a consequence of \cite[Proposition 5.4]{Allaart-Kong-2021} that $\beta\in E^\cs$ if and only if $\al(\beta)=\Phi_\cs(\al(\hat{\beta}))$ for some $\hat{\beta}\in E$. In this case we say that $\beta$ is {\em renormalizable}\footnote{Our definition differs slightly from that of Hubbard and Sparrow \cite{Hubbard-Sparrow-1990}, who call the above notion {\em properly renormalizable} and consider $\beta\in E$ to be {\em trivially renormalizable}.} by the word $\cs$.

We further define the set
\begin{equation} \label{eq:E_f}
E_\f:=\bigcap_{k=1}^\f\bigcup_{\cs\in\La_k}J^{\cs}.
\end{equation}
These are the values of $\beta$ that lie in infinitely many of the intervals $J^{\cs}$, and are hence {\em infinitely renormalizable}. (In \cite{Allaart-Kong-2021}, $E_\f\cap(1,2]$ is called the {\em infinitely Farey set}, because its elements arise from substitutions of an infinite sequence of Farey words.) The set $E_\f$ is uncountable but of zero Hausdorff dimension.
By \cite[Theorem 3]{Allaart-Kong-2021} {(see also \cite{Allaart-Kong-2024})}, we have the decomposition in disjoint sets,
\begin{equation} \label{eq:decomposition}
{(1,\f)}=E\cup E_\f\cup\bigcup_{\cs\in\La}E^{\cs}\cup\bigcup_{\cs\in\La}I^{\cs}.
\end{equation}
In this paper, different parts of this decomposition require different analysis, with the case of $\beta$ in a basic interval being the most involved.

\begin{remark}
Our presentation of the substitution map $\Phi_\s$ and its properties differs somewhat from the substitutions used by most authors, but can be seen to be equivalent. Precisely, let $\w_+$ (resp. $\w_-$) be the word $\s$ (resp. $\L(\s)$) shifted cyclically one place to the right.
Alternatively, $\w_-$ is the largest cyclic permutation of $\s$ beginning with $0$, and $\w_+$ is the smallest cyclic permutation of $\s$ beginning with $1$; see \cite{Glendinning-Sidorov-2015}. For example, if $\s=01011$, then $\w_+={10101}$ and $\w_-=01101$.
If $\y=y_1y_2\dots$ is a sequence beginning with $1$, then $0\Phi_\s(\y)$ is a concatenation of words from $\{\w_+,\w_-\}$, where the first block is $\w_-$ and for each $i\geq 2$, the $i$th block is $\w_-$ if $y_{i-1}=0$, and $\w_+$ if $y_{i-1}=1$. Similarly, if $\y$ is a sequence beginning with $0$, then $1\Phi_\s(\y)$ is also such a concatenation, but beginning instead with $\w_+$.

Although the approach using the words $\w_+$ and $\w_-$ has been the traditional way to express renormalizations (see \cite{Hubbard-Sparrow-1990}), we prefer here the map $\Phi_\s$ because (i) it emphasizes that the renormalization is parametrized by a single word rather than two different ones; (ii) it emphasizes the importance of Lyndon and Farey words; and (iii) it greatly facilitates operations on finite words, which are used frequently in future sections. On the other hand, our map $\Phi_\s$ is essentially equivalent (up to the above-mentioned cyclical shift) to the map $\rho_r$ of Glendinning and Sidorov \cite{Glendinning-Sidorov-2015}, where $r=|\s|_1/|\s|$. The maps $\rho_r$ are useful for studying the sets $\Omega_{\a,\b}$, but for dealing with the sets $\Sigma_{\a,\b}$ we find the maps $\Phi_\s$ more practical, despite their somewhat clunkier definition.

We point out also that it is possible to write the substitution maps $\Phi_\cs$, $\cs\in\La$ completely in terms of just three basic substitutions, namely
\[
L: \begin{cases}0 &\mapsto 0\\ 1 &\mapsto 10\end{cases}, \qquad
M: \begin{cases}0 &\mapsto 01\\ 1 &\mapsto 10\end{cases}, \qquad
R: \begin{cases}0 &\mapsto 01\\ 1 &\mapsto 1\end{cases},
\]
as was done by Komornik, Steiner and Zou \cite{Komornik-Steiner-Zou-2022} and before that (with different notation) by Labarca and Moreira \cite{Labarca-Moreira-2006}.
\end{remark}

\section{Characterization of the bifurcation set $\EE_\beta$} \label{sec:E_beta}

In the introduction we have followed the notation and definitions of Baker and Kong \cite{Baker-Kong-2020}. However, we point out that $K_\beta(t)$ is slightly different from the set denoted $K_\beta(t)$ in \cite{Kalle-Kong-Langeveld-Li-18}, and as a result, the bifurcation set $\EE_\beta$ defined in \eqref{eq:EE-beta} is slightly different from the set called $E_\beta$ in \cite{Kalle-Kong-Langeveld-Li-18}. In fact, setting $\check{K}_\beta(t):=\{x\in[0,1): T_\beta^n(x)\not\in(0,t)\ \forall n\geq 0\}$, we have
\begin{equation}\label{eq:E-beta}
E_\beta=\set{t\in[0,1): \check K_\beta(t')\ne \check K_\beta(t)~\forall t'>t}.
\end{equation}

It is also convenient to define the set
\[
\EE_\beta^+:=\{t\in[0,1): T_\beta^n(t)\geq t\ \forall n\geq 0\}=\{t\in[0,1): t\in K_\beta(t)\}.
\]
(This is the set denoted $E_\beta^+$ in \cite{Kalle-Kong-Langeveld-Li-18}.) The following characterization will play an important role throughout this article.

\begin{lemma} \label{lem:E-beta-characterizations}
For each $\beta{>1}$, we have
\[
\EE_\beta=\set{t\in[0,1): K_\beta(t-\ep)\ne K_\beta(t+\ep)~\forall \ep>0}=\EE_\beta^+.
\]
\end{lemma}

\begin{proof}
If $t\in\EE_\beta^+$, the definition of $K_\beta(t)$ implies that $t\in K_\beta(t)\setminus K_\beta(t+\ep)$ for any $\ep>0$, and so $t\in\EE_\beta$.

Next, if $t\in\EE_\beta$, then $K_\beta(t+\ep)\neq K_\beta(t)$ for all $\ep>0$, so certainly $K_\beta(t-\ep)\neq K_\beta(t+\ep)$ for all $\ep>0$.

Finally, suppose $t\not\in \EE_\beta^+$. Then there is an $n_0\in\N$ such that $T_\beta^{n_0}(t)<t$. By the continuity of $T_\beta^{n_0}$ there is a sufficiently small $\ep>0$ such that
  \begin{equation} \label{eq:Ebeta-2}
    T_\beta^{n_0}(x)<t-\ep\quad\forall x\in[t-\ep, t+\ep].
  \end{equation}
  Suppose there exists $x_0\in K_\beta(t-\ep)\setminus K_\beta(t+\ep)$. Then there is an integer $n_1$ such that $t-\ep\le T_\beta^{n_1}(x_0)<t+\ep$. Thus, $T_\beta^{n_0+n_1}(x_0)<t-\ep$ by (\ref{eq:Ebeta-2}), contradicting that $x_0\in K_\beta(t-\ep)$. Hence, $K_\beta(t-\beta)=K_\beta(t+\ep)$.
\end{proof}

By a similar argument it can be shown that $E_\beta=\set{t\in[0,1): T_\beta^n(t)\notin(0,t)~\forall n\ge 0}$. Thus, $\EE_\beta\subseteq E_\beta$, and $E_\beta\backslash \EE_\beta$ is at most countable. As a result, when considering questions of Hausdorff dimension the two sets may be freely interchanged. Because of this, \eqref{eq:dimension-identity} is indeed equivalent to the conjecture of Kalle {\em et al.}

\section{The key ideas} \label{sec:key}

In this section we outline the main idea of the proof of Theorem \ref{thm:main}. The key is to prove the more technical Theorem \ref{thm:general-transitivity} below; its proof takes up much of the rest of the paper. First, we define the concepts appearing in the theorem.

The metric $d_2$ on ${A_\beta}^\N$ defined by
\[
d_2((x_i),(y_i)):=2^{-\inf\{i:x_i\neq y_i\}}, \qquad {(x_i),(y_i)}\in{A_\beta}^\N,
\]
induces a Hausdorff dimension on ${A_\beta}^\N$, which we also denote by $\dim_H$.

Next, we recall the definition of the $\beta$-shift from \eqref{eq:beta-shift}.
It was shown recently by Li \cite[Theorem 1.1]{Li-2022} that
\begin{equation} \label{eq:Hausdorff-dimension-preserved}
\dim_H \pi_\beta(A)=\frac{\log 2}{\log\beta}\dim_H A \qquad \forall\,A\subseteq \Sigma_\beta.
\end{equation}

\begin{definition} \label{def:descending-subshifts}
Let $\mathcal{T}\subseteq(0,1)$ be an index set. A collection $\{\mathcal{X}(t): t\in\mathcal{T}\}$ is a {\em strictly descending collection of subshifts} if for any $t,t'\in\mathcal{T}$ with $t<t'$, $\mathcal{X}(t')$ is a proper subshift of $\mathcal{X}(t)$.
\end{definition}

\begin{definition} \label{def:left-progressing}
A (finite or infinite) sequence $\II=(I_1,I_2,\dots)$ of intervals is {\em left-progressing} if for each $k\geq 1$, $\sup I_{k+1}<\inf I_k$.
\end{definition}

\begin{definition}
Let $\mathcal{Y}$ be a subshift of ${A_\beta}^\N$, and $\mathcal{X}$ a subshift of $\mathcal{Y}$. We say $\mathcal{X}$ has {\em full entropy} in $\mathcal{Y}$ if $h(\mathcal{X})=h(\mathcal{Y})$. Likewise, we say $\mathcal{X}$ has {\em full Hausdorff dimension} in $\mathcal{Y}$ if $\dim_H \mathcal{X}=\dim_H \mathcal{Y}$.
\end{definition}


Recall from the Introduction that a {\em $\beta$-Lyndon interval} is an interval $[t_L,t_R]$ determined by
\[
b(t_L,\beta)=\w 0^\f, \qquad b(t_R,\beta)=\w^\f,
\]
where $\w$ is a $\beta$-Lyndon word. (Recall this means $\w$ is Lyndon and $\si^n(\w^\f)\prec\al(\beta)$ for all $n\geq 0$.)
We denote by $\mathcal{T}_R(\beta)$ the set of all right endpoints of $\beta$-Lyndon intervals in $[0,\tau(\beta)]$, where $\tau(\beta)$ was defined in \eqref{eq:critical-value}. 
{Recall that
\[
\Kt_\beta(t):=\{\z\in{A_\beta}^\N: b(t,\beta)\lle \sigma^n(\z)\lle \alpha(\beta)\ \forall n\geq 0\}.
\]}

\begin{theorem} \label{thm:general-transitivity}
For each $\beta>1$ there is a (finite or infinite, possibly empty) left-progressing sequence $\II$ of intervals such that
\begin{enumerate}[{\rm(i)}]
\item For any $t_R\in\mathcal{T}_R(\beta)\backslash \bigcup_{I\in\II}I$, $\Kt_\beta(t_R)$ has a transitive subshift $\mathcal{K}_\beta'(t_R)$ of full entropy and full Hausdorff dimension that contains the sequence $b(t_R,\beta)$;
\item The collection
\[
\{\mathcal{K}_\beta'(t_R): t_R\in\mathcal{T}_R(\beta)\backslash \textstyle{\bigcup_{I\in\II}I}\}
\]
is a strictly descending collection of subshifts;
\item The entropy function $t\mapsto h(\Kt_\beta(t))$ is constant throughout each interval $I\in\II$.
\end{enumerate}
\end{theorem}

Theorem \ref{thm:general-transitivity} describes the most general case; for many values of $\beta$ the situation is simpler. In the following sections, we will prove Theorem \ref{thm:general-transitivity} for different cases of $\beta$. But first, we show how it implies Theorem \ref{thm:main}.

\begin{lemma} \label{lem:E-beta-next-to-beta-Lyndon}
Let $t\in\EE_\beta\cap(0,1)$. Then $t$ is either the right endpoint of a $\beta$-Lyndon interval, or $t$ is the limit from the left of a sequence of right endpoints of $\beta$-Lyndon intervals.
\end{lemma}

\begin{proof}
By Lemma \ref{lem:E-beta-characterizations}, $b(t,\beta)\lle\si^n(b(t,\beta))\prec\al(\beta)$ for all $n\geq 0$. If $b(t,\beta)$ is periodic, say $b(t,\beta)=(b_1\dots b_k)^\f$ where $k$ is the minimal period of $b(t,\beta)$, then $b_1\dots b_k$ is $\beta$-Lyndon, and hence $t$ is the right endpoint of a $\beta$-Lyndon interval.

Assume now that $b(t,\beta)=b_1b_2\dots$ is not periodic. Then by \cite[Lemma 3.5]{Kalle-Kong-Langeveld-Li-18}, there are infinitely many integers $k$ such that $b_1\dots b_k$ is Lyndon. But then $b_1\dots b_k$ is in fact $\beta$-Lyndon, since for each $n\geq 0$ there is an integer $j<k$ such that
\[
\si^n\big((b_1\dots b_k)^\f\big)=b_{j+1}\dots b_k(b_1\dots b_k)^\f\lle \si^j(b(t,\beta))\prec\al(\beta).
\]
Since $(b_1\dots b_k)^\f\to b(t,\beta)$ in the order topology as $k\to\f$ and $(b_1\dots b_k)^\f\prec b(t,\beta)$ for all $k$ by the assumption that $t\in\EE_\beta$, it follows that $t$ is the limit from the left of a sequence of right endpoints of $\beta$-Lyndon intervals.
\end{proof}

\begin{proof}[Proof of Theorem \ref{thm:main}, assuming Theorem \ref{thm:general-transitivity}]
The proof is based roughly on the argument in the proof of \cite[Lemma 3.8]{Baker-Kong-2020}, but requires several new ideas. Observe that we only need to prove the ``$\geq$" inequality, since $\EE_\beta\cap[t,1)\subseteq K_\beta(t)$ by Lemma \ref{lem:E-beta-characterizations}. Furthermore, it suffices to establish the inequality for $t\in \BB_\beta$: For arbitrary $t\in(0,1)$, let $t':=\inf(\BB_\beta\cap[t,1])$; {then $t'\in\BB_\beta$ by definition of $\BB_\beta$, and}
 we then have
\begin{align*}
\dim_H\big(\EE_\beta\cap[t,1])&\geq \dim_H\big(\EE_\beta\cap[t',1])=\dim_H K_\beta(t')\\
&=\dim_H {K_\beta(t'-)}=\dim_H K_\beta(t),
\end{align*}
where the {second} equality follows from the continuity of $t\mapsto \dim_H K_\beta(t)$.

Observe that each $t\in\BB_\beta$ can be approximated from the right by a sequence of right endpoints of $\beta$-Lyndon intervals $[t_L^{(n)},t_R^{(n)}], n\in\N$. (If this were not the case, then some right neighborhood $(t,t+\ep)$ would not contain any points of $\EE_\beta$ by Lemma \ref{lem:E-beta-next-to-beta-Lyndon}, hence $(t,t+\ep)$ would not contain any points of $\BB_\beta$. But every point of $\BB_\beta$ is an accumulation point from the right of $\BB_\beta$ by the continuity of $t\mapsto \dim_H K_\beta(t)$.) We can moreover choose the sequence $(t_R^{(n)})$ so that $\dim_H K_\beta(t_R^{(n)})$ is strictly increasing in $n$. Thus, we have
\[
h\big(\Kt_\beta(t_R^{(n)})\big)<h(\Kt_\beta(t)) \quad\forall\,n \qquad\mbox{and} \qquad h\big(\Kt_\beta(t_R^{(n)})\big)\nearrow h(\Kt_\beta(t)) \quad\mbox{as $n\to\f$}.
\]

Since the sequence $\II$ of intervals in Theorem \ref{thm:general-transitivity} is left-progressing, there is at most one point all of whose neighborhoods intersect infinitely many of the intervals in $\II$; if such a point exists, it lies to the left of all the intervals in $\II$ and we denote it by $t_0$.

We assume first that $t\neq t_0$. {Since $t\in\BB_\beta$ and entropy is constant on each interval of $\II$ by Theorem \ref{thm:general-transitivity} (iii), $t$ does not lie in any interval of $\II$ except possibly as a right endpoint of such an interval. Hence, except for finitely many values of $n$, the points $t_R^{(n)}$ do not lie in any interval of $\II$.}  
By Theorem \ref{thm:general-transitivity} (i) it follows that for all large enough $n$, $\Kt_\beta(t_R^{(n)})$ has a transitive subshift $\mathcal{K}_\beta'(t_R^{(n)})$ of full Hausdorff dimension containing the sequence $b(t_R^{(n)},\beta)$.

Write $\al(\beta)=\al_1\al_2\dots$.
First assume $t\in\BB_\beta$ such that $t=t_R$ for a $\beta$-Lyndon interval $[t_L,t_R]$ generated by the $\beta$-Lyndon word $\w=w_1\dots w_p$, so $b(t_R,\beta)=\w^\f$. Since the greedy expansion $b(t,\beta)$ is right-continuous in $t$, we may assume that $t_R^{(n)}$ is close enough to $t_R$ so that $b(t_R^{(n)},\beta)$ begins with $\w$. Furthermore, since $b(t_R^{(n)},\beta)\succ b(t_R,\beta)=\w^\f$, there is for each $n$ a sufficiently large integer $k_n$ such that
\begin{equation} \label{eq:eventually-above-w}
b(t_R^{(n)},\beta)\succ \w^{k_n}{M_\beta^\f}.
\end{equation}
{(Recall that $A_\beta=\{0,1,\dots,M_\beta\}$)}. Let
\[
\mathcal{F}_n:=\left\{\w^{k_n}\x: \x=x_1x_2\dots\in \mathcal{K}_\beta'\big(t_R^{(n)}\big), x_1\dots x_p=\w,\ \mbox{and}\  \si^j(\x)\prec \al(\beta)\ \forall\,j\geq 0\right\},
\]
and
\[
F_n:=\pi_\beta(\mathcal{F}_n).
\]
{We will show that $F_n\subset \EE_\beta\cap[t_R, 1]$.}

Take $\w^{k_n}\x\in\mathcal{F}_n$. First we show that
\begin{equation} \label{eq:admissible-beta-expansion}
  \si^j(\w^{k_n}\x)\prec\al(\beta) \quad\forall j\ge 0.
\end{equation}
Since $\si^l(\x)\prec\al(\beta)$ for all $l\geq 0$, \eqref{eq:admissible-beta-expansion} holds for $j\geq k_n|\w|$. Observing that $\x$ begins with $\w$, the inequality will follow for all $j<k_n|\w|$ once we prove that
\begin{equation} \label{eq:admissible-1}
\L(\w)=\L(w_1\ldots w_p)\prec \al_1\ldots \al_p.
\end{equation}
Since $\si^n(\w^\f)\prec \al(\beta)$ for all $n\ge 0$, $\L(\w)\lle \al_1\ldots \al_p$. If $\L(\w)=\al_1\ldots \al_p$, then for some $i\leq p$,
\[
(\al_1\ldots \al_p)^\f=\si^i(\w^\f)\prec \al(\beta).
\]
But then $\al_1\ldots \al_p\prec \al_{\ell+1}\ldots \al_{\ell+p}$ for some $\ell\ge 0$, contradicting Lemma \ref{lem:quasi-greedy expansion-alpha-q}. Therefore, (\ref{eq:admissible-1}) holds, and thus we have (\ref{eq:admissible-beta-expansion}).

Since $\x\in \mathcal{K}_\beta'\big(t_R^{(n)}\big)$, we have $\x\lge b(t_R^{(n)},\beta)\succ \w^\f$. Furthermore, by \eqref{eq:admissible-beta-expansion}, $\w^{k_n}\x$ is a greedy $\beta$-expansion. Since $\si^j(\w^\f)\prec \al(\beta)$ for each $j\geq 0$, $\w^\f$ is also a greedy $\beta$-expansion. Thus,
\[
\pi_\beta(\w^{k_n}\x)\geq\pi_\beta (\w^\f)=t_R.
\]
Furthermore, using \eqref{eq:eventually-above-w},
\[
\sigma^j(\x)\lge b(t_R^{(n)},\beta)\succ \w^{k_n}{M_\beta}^\f\succ \w^{k_n}\x \qquad \forall j\geq 0,
\]
and $w_{j+1}\dots w_p\succ w_1\dots w_{p-j}$ for all $1\leq j<p$ since $\w$ is Lyndon. Hence
\[
\sigma^j(\w^{k_n}\x)\lge \w^{k_n}\x \qquad\forall j\geq 0.
\]
We conclude that
\begin{equation} \label{eq:containment}
F_n\subseteq \EE_\beta\cap[t_R,1].
\end{equation}

Next, observe that the condition $\si^j(\x)\prec \al(\beta)\ \forall\,j\geq 0$ in the definition of $\mathcal{F}_n$ removes at most a countable number of sequences from $\mathcal{K}_\beta'\big(t_R^{(n)}\big)$.
Since $\mathcal{K}_\beta'\big(t_R^{(n)}\big)$ is transitive and $\w$ is allowed in $\mathcal{K}_\beta'\big(t_R^{(n)}\big)$, there is for each sequence $\z\in\mathcal{K}_\beta'\big(t_R^{(n)}\big)$ a word $\v=\v(\z)$ such that $\w\v\z\in \mathcal{K}_\beta'\big(t_R^{(n)}\big)$. For each finite word $\v$, denote by $\mathcal{X}_\v$ the set of all sequences $\z\in\mathcal{K}_\beta'\big(t_R^{(n)}\big)$ such that $\v(\z)=\v$. Then
\[
\bigcup_{\v} \mathcal{X}_\v=\mathcal{K}_\beta'\big(t_R^{(n)}\big) \qquad \mbox{and} \qquad \mathcal{F}_n\supseteq \left(\bigcup_\v\w^{k_n+1}\v\mathcal{X}_\v\right)\backslash\mathcal{C},
\]
where $\mathcal{C}$ is a countable set. (Here our stronger definition of {\em transitive} is particularly convenient.) Hence,
\[
\dim_H \mathcal{F}_n\geq \sup_\v \dim_H \mathcal{X}_\v=\dim_H \mathcal{K}_\beta'\big(t_R^{(n)}\big)=\dim_H \Kt_\beta\big(t_R^{(n)}\big),
\]
where the last equality follows since $\mathcal{K}_\beta'\big(t_R^{(n)}\big)$ has full Hausdorff dimension in $\Kt_\beta\big(t_R^{(n)}\big)$.
Since all of the sets involved are subsets of $\Sigma_\beta$, it follows from \eqref{eq:Hausdorff-dimension-preserved} that
\[
\dim_H F_n=\dim_H\pi_\beta(\mathcal{F}_n)\geq \dim_H \pi_\beta\left(\Kt_\beta\big(t_R^{(n)}\big)\right)=\dim_H K_\beta\big(t_R^{(n)}\big).
\]
Thus, together with \eqref{eq:containment}, we obtain
\[
\dim_H\big(\EE_\beta\cap[t_R,1]\big)\geq \dim_H K_\beta\big(t_R^{(n)}\big).
\]
Letting $t_R^{(n)}\searrow t_R$ gives $\dim_H\big(\EE_\beta\cap[t_R,1]\big)\geq \dim_H K_\beta(t_R)$, in view of the continuity of $t\mapsto \dim_H K_\beta(t)$.

Finally, for arbitrary $t\in\BB_\beta$ (including the accumulation point $t=t_0$ if it exists), $t$ can be approximated from the right by a sequence $\big(t_R^{(n)}\big)$ of right endpoints of $\beta$-Lyndon intervals, none of which are equal to $t_0$; so
\[
\dim_H\big(\EE_\beta\cap[t,1]\big)\geq \dim_H\big(\EE_\beta\cap[t_R^{(n)},1]\big)\geq \dim_H K_\beta\big(t_R^{(n)}\big),
\]
and, again by the continuity of the map $t\mapsto \dim_H K_\beta(t)$, it follows that $\dim_H\big(\EE_\beta\cap[t,1]\big)\geq \dim_H K_\beta(t)$.
\end{proof}

In this proof we did not use property (ii) of Theorem \ref{thm:general-transitivity}. However, it will be needed later in the proof of Theorem \ref{thm:size-of-E-minus-B}. The following result illustrates the main idea.

\begin{proposition} \label{prop:E-B-equality}
Let $\beta{>1}$, and suppose that
\begin{enumerate}[{\rm(i)}]
\item For any $t_R\in\mathcal{T}_R(\beta)$, $\Kt_\beta(t_R)$ has a transitive sofic subshift $\mathcal{K}_\beta'(t_R)$ of full entropy;
\item The collection $\{\mathcal{K}_\beta'(t_R): t_R\in\mathcal{T}_R(\beta)\}$ is a strictly descending collection of subshifts; and
\item The $\beta$-Lyndon intervals are dense in $[0,\tau(\beta)]$.
\end{enumerate}
Then
\[
\EE_\beta\cap[0,\tau(\beta))=\BB_\beta\cap[0,\tau(\beta)).
\]
\end{proposition}

\begin{proof}
We show that the plateaus of $\delta_\beta: t\mapsto\dim_H K_\beta(t)$ in $[0,\tau(\beta)]$ are precisely the $\beta$-Lyndon intervals. This gives the desired conclusion, since $\EE_\beta$ intersects the $\beta$-Lyndon intervals only in their right endpoints, and the $\beta$-Lyndon intervals are dense in $[0,\tau(\beta)]$.

It is enough to show that if $[t_L,t_R]$ and $[u_L,u_R]$ are two distinct $\beta$-Lyndon intervals with $t_R<u_R$, then $h\big(\Kt_\beta(t_R)\big)>h\big(\Kt_\beta(u_R)\big)$. By hypotheses (i) and (ii), $\mathcal{K}_\beta'(t_R)$ is sofic and transitive, and $\mathcal{K}_\beta'(u_R)$ is a proper subshift of $\mathcal{K}_\beta'(t_R)$. Hence, by \cite[Corollary 4.4.9]{Lind_Marcus_1995},
\[
h\big(\Kt_\beta(t_R)\big)=h\big(\mathcal{K}_\beta'(t_R)\big)>h\big(\mathcal{K}_\beta'(u_R)\big)=h\big(\Kt_\beta(u_R)\big).
\]
This completes the proof.
\end{proof}

As we will see later, the hypotheses of Proposition \ref{prop:E-B-equality} are satisfied when $\beta\in\{\beta_\ell^\cs, \beta_*^\cs,\beta_r^\cs\}$, for any $\cs\in\La$. Thus, for these values of $\beta$, any points of $\EE_\beta\backslash\BB_\beta$ must lie to the right of $\tau(\beta)$; see Theorem \ref{thm:E-minus-B}.

\section{The case $\beta\in E_L$} \label{sec:beta-in-E}

We begin with the case when $\beta\in E_L$. Our main result in this section is:

\begin{theorem} \label{thm:transitive-in-E_L}
Let $\beta\in E_L$. Then $\Kt_\beta(t_R)$ is transitive for any right endpoint $t_R$ of a $\beta$-Lyndon interval $[t_L,t_R]$ with $t_R<\tau(\beta)=1-(1/\beta)$. In other words, the conclusion of Theorem \ref{thm:general-transitivity} holds with $\II=\emptyset$ and $\mathcal{K}_\beta'(t_R)=\Kt_\beta(t_R)$ for each $t_R\in\mathcal{T}_R$.
\end{theorem}

We will prove Theorem \ref{thm:transitive-in-E_L} after developing some lemmas. {We deal first with $1<\beta<2$ and then show how to derive the general case.}


\begin{lemma} \label{lem:extensibility}
Let $1<\beta<2$, and $0<t<1-(1/\beta)$. Then for every sequence $\z\in \Kt_\beta(t)$ there exist a word $\u$ beginning with $0$ and a word $\v$ beginning with $1$ such that $\u\z, \v\z \in \Kt_\beta(t)$.
\end{lemma}


\begin{proof}
Let $\z\in\Kt_\beta(t)$. Suppose $1\z\not\in\Kt_\beta(t)$. Then $1\z\succ\al(\beta)$, and hence $0\z\succ 0\si(\al(\beta))=b(1-1/\beta,\beta)\succ b(t,\beta)$. Therefore, $0\z\in\Kt_\beta(t)$. So $\z$ can always be extended to the left.

Let $N$ be the number of consecutive $1$'s at the beginning of $\al(\beta)$, and $M$ be the number of consecutive $0$'s at the beginning of $b(t,\beta)$. Then $N,M\geq 1$ since $\beta>1$ and $t<1-1/\beta<1/\beta$. Clearly, no sequence in $\Kt_\beta(t)$  can have more than $N$ consecutive $1$'s or $M$ consecutive $0$'s anywhere. Hence, the lemma follows.
\end{proof}

Recall that $\cL(\XX)$ denotes the language of the subshift $\XX$, and $\mathcal{X}$ is transitive if for any word $\u\in\cL(\mathcal{X})$ and sequence $\z\in\mathcal{X}$, there is a word $\v$ such that $\u\v\z\in\mathcal{X}$. The next definition and lemma will be used to simplify transitivity proofs in this section, by removing the word $\u$ from consideration.

If $b(t,\beta)=\w^\f$ for a Lyndon word $\w$, we also write $\Kt_\beta(\w)$ instead of $\Kt_\beta(t)$. So
\[
\Kt_\beta(\w):=\{\z\in{A_\beta}^\N: \w^\f\lle \sigma^n(\z)\lle \alpha(\beta)\ \forall n\geq 0\}.
\]

\begin{definition} \label{def:property-TL}
Let $\beta{>1}$ and write $\al(\beta)=\al_1\al_2\ldots$. We say a Lyndon word $\w$ has {\em Property TL} for $\beta$ if for each sequence $\z\in\Kt_\beta(\w)$ there is a word $\v=v_1\dots v_n$ such that
\begin{enumerate}[(i)]
\item $\v\prec \al_{j+1}\dots\al_{j+n}$ for all $j\geq 0$; and
\item $\w^\f\lle\si^i(\v\z)\lle \al(\beta)$ for all $i\geq 0$, i.e. $\v\z\in\Kt_\beta(\w)$.
\end{enumerate}
\end{definition}

\begin{lemma} \label{lem:transitivity-simplified}
Let $\beta\in\overline{E}$, let $[t_L,t_R]$ be a $\beta$-Lyndon interval with $t_R<1-(1/\beta)$, and let $\w=w_1\dots w_m$ be the $\beta$-Lyndon word such that $b(t_R,\beta)=\w^\f$. If $\w$ satisfies Property TL, then $\Kt_\beta(t_R)$ is transitive.
\end{lemma}

\begin{proof}
{By Proposition \ref{prop:Sturmian} (i) the sequence $\al(\beta)=(\al_i)$ is balanced.} 
{Let $M$ be the integer such that $\beta\in(M,M+1]$.}
We first show that
\begin{equation} \label{eq:0-alpha-inequality}
{(M-1)\si(\al(\beta))=(M-1)\al_2\al_3\dots \lle \si^j(\al(\beta)) \qquad \forall\,j\geq 0.}
\end{equation}
This is trivial for $j=0$, so assume $j\geq 1$. Suppose the contrary; then there is an index $k\geq 2$ such that ${(M-1)}\al_2\dots\al_{k-1}=\al_{j+1}\dots\al_{j+k-1}$ and $\al_k={M}$, $\al_{j+k}={M-1}$. {(Since $\beta\in \overline{E}\cap(M,M+1]$, $\al(\beta)\in\{M-1,M\}^\N$ by Lemma \ref{lem:two-digits}.)} But then the word $\al_1\dots\al_k$ has two more {$M$'s} than the word $\al_{j+1}\dots\al_{j+k}$ of the same length, contradicting that $\al(\beta)$ is balanced. Therefore, we have \eqref{eq:0-alpha-inequality}.

Let $\u=u_1\dots u_n\in\cL\big(\Kt_\beta(t_R)\big)$ and $\z\in\Kt_\beta(t_R)$ be given. Let $j$ be the smallest index such that $u_{j+1}\dots u_n$ is a prefix of $\w^\f$, say $u_{j+1}\dots u_n=(w_1\dots w_m)^{k-1}w_1\dots w_r$ with $k\in\N$ and $0\leq r<m$. If no such $j$ exists, set $j:=n$, $r:=0$ and $k:=1$. Let
\[
\u':=u_1\dots u_n w_{r+1}\dots w_m=u_1\dots u_j \w^k.
\]
Let $\v$ be the word satisfying (i) and (ii) in Definition \ref{def:property-TL} for the given $\z$. We claim that $\c:=\u'\v\z\in\Kt_\beta(t_R)$.

We check first that $\si^i(\c)\lge \w^\f$ for all $i\geq 0$. For $0\leq i<j$ this follows from the minimality of $j$. For $j\leq i<j+km$ and $m\nmid(i-j)$ it follows since $\w$ is Lyndon, so $w_{l+1}\dots w_m\succ w_1\dots w_{m-l}$ for all $1\leq l<m$. And for all other $i\geq j$ the inequality follows from assumption (ii).

Next, we check that $\si^i(\c)\lle \al(\beta)$ for all $i\geq 0$. For $0\leq i<j$, this follows since
\begin{enumerate}[(a)]
\item $u_{i+1}\dots u_j\lle\al_1\dots\al_{j-i}$;
\item $\w^\f\prec b(1-1/\beta,\beta)={(M-1)}\si(\al(\beta))$, so $\w^k\lle \al_{j-i+1}\dots\al_{j-i+km}$ by \eqref{eq:0-alpha-inequality}; and
\item $\v\prec\al_{j-i+km+1}\dots\al_{j-i+km+n}$ by {assumption}  (i).
\end{enumerate}
For $j\leq i<j+km$, the inequality follows since $\si^{i-j}(\w^k)\lle \al_1\dots\al_{km-(i-j)}$ because $\w$ is $\beta$-Lyndon, and $\v\prec\al_{km-(i-j)+1}\dots\al_{km-(i-j)+n}$ by (i). Finally, for $i\geq j+km$, $\si^i(\c)\lle \al(\beta)$ follows immediately from {assumption} (ii).
\end{proof}

{The next two lemmas are specific to $\beta\in(1,2]$.}

\begin{lemma} \label{lem:upper-bound-equivalence-xi}
Let $\beta$ and $\hat{\beta}$ be bases such that $0\al(\beta)=U_0(\al(\hat{\beta}))$.
Let $\z=(z_1,z_2,\dots)$ and $\hat{\z}=(\hat{z}_1,\hat{z}_2,\dots)$ be sequences in $\{0,1\}^\N$ such that $\z=U_0(\hat{\z})$. Then
\[
\si^n(\z)\lle\alpha(\beta) \quad \forall\,n\geq 0 \qquad \Longleftrightarrow \qquad \si^k(\hat{\z})\lle\alpha(\hat{\beta}) \quad \forall\,k\geq 0.
\]
\end{lemma}

\begin{proof}
Suppose first $\si^k(\hat{\z})\lle\alpha(\hat{\beta})$ for all $k\geq 0$. Clearly it is enough to show that $\si^n(\z)\lle\alpha(\beta)$ whenever $z_{n+1}=1$. This implies $n\geq 1$, because $\z=U_0(\hat{\z})$ begins with $0$. Since furthermore $\z$ does not contain the word $11$, we have $\si^{n-1}(\z)=U_0(\si^k(\hat{\z}))$ for some $k\geq 0$. Since $U_0$ is increasing, this implies
\[
0\si^n(\z)=\si^{n-1}(\z)=U_0(\si^k(\hat{\z}))\lle U_0(\al(\hat{\beta}))=0\al(\beta),
\]
and hence, $\si^n(\z)\lle \al(\beta)$.

Suppose next that $\si^n(\z)\lle\alpha(\beta)$ for all $n\geq 0$. For given $k\geq 0$, we can find $n\geq 0$ such that $\sigma^n(\z)=U_0(\sigma^k(\hat{\z}))$ and moreover, $\si^{n}(\z)$ begins with $0$. So
\[
U_0(\si^k(\hat{\z}))=\si^n(\z)=0\si^{n+1}(\z)\lle 0\alpha(\beta)=U_0(\alpha(\hat{\beta})),
\]
and as a result, $\si^k(\hat{\z})\lle\alpha(\hat{\beta})$.
\end{proof}

\begin{lemma} \label{lem:upper-bound-equivalence-eta}
Let $\beta$ and $\hat{\beta}$ be bases such that $0\al(\beta)=U_1(0\al(\hat{\beta}))$.
Let $\z=(z_1,z_2,\dots)$ and $\hat{\z}=(\hat{z}_1,\hat{z}_2,\dots)$ be sequences in $\{0,1\}^\N$ such that $\z=U_1(\hat{\z})$. Then
\[
\si^n(\z)\lle\alpha(\beta) \quad \forall\,n\geq 0 \qquad \Longleftrightarrow \qquad \si^k(\hat{\z})\lle\alpha(\hat{\beta}) \quad \forall\,k\geq 1.
\]
\end{lemma}


\begin{proof}
{
Suppose first that $\si^k(\hat{\z})\lle\alpha(\hat{\beta})$ for all $k\geq 1$. Since $\mathbf{c}\lle 1\mathbf{c}$ for any sequence $\mathbf{c}$, it suffices to show $\si^n(\z)\lle\alpha(\beta)$ for $n=0$ and for $n\geq 1$ with $z_n=0$. If $n=0$, observe that the inequality is trivial when $z_1=0$, so assume $z_1=1$; then $\hat{z}_1=1$ as well. Hence,
\[
\z=U_1(\hat{\z})=1U_1(\si(\hat{\z}))\lle 1U_1(\al(\hat{\beta}))=\al(\beta).
\]
Next, let $n\geq 1$ and $z_n=0$. Then there is a $k\geq 0$ such that $0\si^n(\z)=U_1(\si^k(\hat{\z}))$. This implies that $\si^k(\hat{\z})$ begins with $0$, so $\si^k(\hat{\z})=0\si^{k+1}(\hat{\z})$. Since $U_1$ is increasing, it follows that 
\[
0\si^n(\z)=U_1(0\si^{k+1}(\hat\z))\lle U_1(0\alpha(\hat{\beta}))=0\al(\beta),
\] 
and hence, $\si^n(\z)\lle\al(\beta)$. 

Conversely, suppose $\si^n(\z)\lle\alpha(\beta)$ for all $n\geq 0$. Write $\hat{\z}=\hat{z}_1\hat{z}_2\dots$. Let $k\geq 1$ be given. Since $\mathbf{c}\lle 1\mathbf{c}$ for any sequence $\mathbf{c}$, the critical case is when either $\hat{z}_k=0$ or else $k=1$ and $\hat{z}_1=1$.
Assume first that $\hat{z}_k=0$.
Then we can find $n\geq 0$ such that $\sigma^n(\z)=U_1(\sigma^{k-1}(\hat{\z}))$, and in particular, $\sigma^n(\z)$ begins with $0$. Thus,
\[
U_1(0\si^k(\z))=U_1(\sigma^{k-1}(\hat{\z}))=\sigma^n(\z)=0\sigma^{n+1}(\z)\lle 0\alpha(\beta)=U_1(0\alpha(\hat{\beta})),
\]
and since $U_1$ is strictly increasing, it follows that $\sigma^k(\hat{\z})\lle\alpha(\hat{\beta})$.

Next, assume $k=1$ and $\hat{z}_1=1$. Then, as in the first part of the proof,
\[
1U_1(\si(\hat{\z}))=\z\lle\al(\beta)=1U_1(\al(\hat{\beta})),
\]
and hence, $\si(\hat{\z})\lle \al(\hat{\beta})$.
}
\end{proof}

\begin{remark} \label{rem:upper-bound-equivalences}
(a) Both Lemmas \ref{lem:upper-bound-equivalence-xi} and \ref{lem:upper-bound-equivalence-eta} continue to hold if we replace the inequalities with strict inequalities, because the maps $U_0$ and $U_1$ are strictly increasing.

(b) In Lemma \ref{lem:upper-bound-equivalence-eta} we had to exclude the value $k=0$. For example, let $\al(\hat{\beta})$ begin with $10$, so $0\al(\beta)=U_1(0\al(\hat{\beta}))$ begins with $01101$ and consists of blocks $01$ and $011$. Let $\hat{\z}=110^\f$, so $\z=U_1(\hat{\z})=11(01)^\f$. Then $\si^n(\z)\lle\al(\beta)$ for all $n\geq 0$, yet $\hat{\z}\succ \al(\hat{\beta})$.
\end{remark}

\begin{proof}[Proof of Theorem \ref{thm:transitive-in-E_L}]
{Let $\beta\in E_L$.}
{Assume first that $\beta\in(1,2]$.}
By Lemma \ref{lem:transitivity-simplified} it is enough verify Property TL for each $\w\in L^*$ with $\w^\f\prec 0\si(\al(\beta))$. We use induction on the length of the generating word $\w$ of $[t_L,t_R]$.

The shortest Lyndon word that can generate a $\beta$-Lyndon interval is $\w=01$. Since $\w$ is $\beta$-Lyndon, we have
$\al(\beta)\succ \L(\w)^\f=\al(\beta_\ell^\w)$.
Furthermore, $\w\in\F$, and $\beta\in E_L$. So by (\ref{eq:bifurction-E-L}) it follows that
\begin{equation*} 
\alpha(\beta)\succ \al(\beta_r^\w)=\L(\w)^+\w^\f=11(01)^\f.
\end{equation*}
In particular, $\al(\beta)$ begins with $11$. By Lemma \ref{lem:not-0-1-repeated}, $\al(\beta)$ does not end in $(01)^\f$, so there is an $M\in\N$ such that $\al(\beta)$ does not contain the word $(01)^M$. {Take $\z\in\Kt_\beta(\w)$. In view of Lemma \ref{lem:extensibility},} assume $\z$ begins with $0$, and set $\v:=(01)^M$. Since $\al(\beta)$ is balanced by Lemma \ref{lem:balanced-expansion}, and $\al(\beta)$ begins with $11$, it does not contain the word $00$. Conditions (i) and (ii) of Definition \ref{def:property-TL} are clearly satisfied. Hence $\w$ satisfies Property TL.

Now let $\w$ be a Lyndon word of length $\geq 3$. By Lemma \ref{lem:alpha-renormalizing}, there is a base $\hat{\beta}\in E_L$ such that
$0\al(\beta)=U_0(\al(\hat{\beta}))$ or $0\al(\beta)=U_1(0\al(\hat{\beta}))$.

\medskip

{\bf Case (A)}. Assume first that
\begin{equation} \label{eq:quasi-greedy-relationship-b2}
0\alpha(\beta)=U_0(\alpha(\hat{\beta})).
\end{equation}
Then $\al(\beta)$ does not contain the word 11, so $\w$ does not contain this word either, since $\w$ is $\beta$-Lyndon. Thus, since $|\w|\geq 3$ and $\w$ is Lyndon, $\w$ must begin with $00$, which implies $\w=U_0(\hat{\w})$ for some word $\hat{\w}$ beginning with $0$ and ending with $1$. Since $U_0$ is strictly increasing, $\hat{\w}$ is again Lyndon. Lemma \ref{lem:upper-bound-equivalence-xi} (with strict, instead of weak inequalities) implies that $\si^k(\hat{\w}^\f)\prec \alpha(\hat{\beta})$ for all $k\geq 0$. Thus, $\hat{\w}^\f=b(\hat{t}_R,\hat{\beta})$ for some $\hat{t}_R\in(0,1)$, and
\begin{equation}  \label{eq:p16-2}
  U_0(b(\hat{t}_R,\hat{\beta}))=U_0(\hat\w^\f)=\w^\f=b(t_R,\beta).
\end{equation}
Furthermore, the assumption $\w^\f\prec 0\si(\alpha(\beta))$ implies $\hat{\w}^\f\prec 0\si(\alpha(\hat{\beta}))$, because \eqref{eq:quasi-greedy-relationship-b2} implies $\al(\beta)=1U_0\big(\si(\al(\hat{\beta}))\big)$, and so
\begin{equation} \label{eq:shifted-substitution}
0\si(\al(\beta))=0U_0\big(\si(\al(\hat{\beta}))\big)=U_0\big(0\si(\al(\hat{\beta}))\big).
\end{equation}
Hence, $\hat{t}_R<1-1/\hat{\beta}$.

Now let $\z\in\Kt_\beta(t_R)$ be given, and assume $\z$ begins with $0$ {by Lemma \ref{lem:extensibility}}. Observe that
\begin{equation} \label{eq:p16-1}
b(t_R,\beta)\lle \si^n(\z)\lle\alpha(\beta)\quad\forall n\geq 0.
\end{equation}
Since $\al(\beta)$ begins with $10$, $\z$ cannot contain the word $11$, and so $\z=U_0(\hat{\z})$ for some sequence $\hat{\z}$. By (\ref{eq:p16-2}), (\ref{eq:p16-1}) and Lemma \ref{lem:upper-bound-equivalence-xi} it follows that
\[
b(\hat t_R, \hat\beta)=\hat\w^\f\lle\si^k(\hat{\z})\lle\alpha(\hat{\beta})\quad\textrm{for all }k\geq 0.
\]
As a result, $\hat{\z}\in \Kt_{\hat{\beta}}(\hat{t}_R)$ with $\hat t_R<1-1/\hat\beta$.

Since $\hat{\w}$ is shorter than $\w$, the induction hypothesis now implies that $\hat{\w}$ has property TL for $\hat{\beta}$, so with $\al(\hat{\beta})=:\hat{\al}_1\hat{\al}_2\dots$, 
 there is a word $\hat{\v}$ such that
\begin{enumerate}[(i)]
\item $\hat{\v}\prec \hat{\al}_{j+1}\dots\hat{\al}_{j+|\hat{\v}|}$ for all $j\geq 0$; and
\item $\hat{\w}^\f\lle\si^i(\hat{\v}\hat{\z})\lle \al(\hat{\beta})$ for all $i\geq 0$.
\end{enumerate}
Set $\v:=U_0(\hat{\v})$. Then $\v$ satisfies condition (i) of Definition \ref{def:property-TL} because $U_0$ is strictly increasing. Furthermore, the lower bound $\si^i(\v\z)\lge\w^\f$ is satisfied because $U_0$ is increasing and $\w$ is Lyndon. Finally, the upper bound $\si^i(\v\z)\lle\al(\beta)$ follows from (ii) above and Lemma \ref{lem:upper-bound-equivalence-xi}. Hence, $\w$ has Property TL for $\beta$.

\medskip
{\bf Case (B).} Assume next that
\begin{equation} \label{eq:quasi-greedy-relationship-a}
0\alpha(\beta)=U_1(0\alpha(\hat{\beta})).
\end{equation}
Then $\al(\beta)\neq (10)^\f$, for otherwise $\al(\hat{\beta})$ would equal $0^\f$, which isn't possible. {Note by (\ref{eq:quasi-greedy-relationship-a}) that $00$ is forbidden in $\al(\beta)$.} So $\al(\beta)$ begins with $11$, and by Lemma \ref{lem:not-0-1-repeated}, $\al(\beta)$ does not end in $(01)^\f$. Hence there is a positive integer $M$ such that $\al(\beta)$ does not contain the word $(01)^M$. Furthermore, $\al(\beta)$ does not contain the word $00$.

If $\w$ begins with $00$, {then for $\z\in\Kt_\beta(\w)$ beginning with $0$, we} take $\v=(01)^M$;   it is easy to see that $\v$ satisfies the conditions of Definition \ref{def:property-TL}. Hence, $\w$ has Property TL for $\beta$.

Assume therefore, for the remainder of the proof, that $\w$ begins with $01$. Then $\w$ does not contain the word $00$ because $\w$ is Lyndon, and hence $\w=U_1(\hat{\w})$ for some word $\hat{\w}$. As in Case (A), we can deduce from Lemma \ref{lem:upper-bound-equivalence-eta} (with strict inequality) that $\hat{\w}\in L^*(\hat{\beta})$, so $\hat{\w}^\f=b(\hat{t}_R,\hat{\beta})$ for some $\hat{t}_R\in(0,1)$. Furthermore, the assumption \eqref{eq:quasi-greedy-relationship-a}
 implies that $U_1\big(0\si(\alpha(\hat{\beta}))\big)=0\si(\alpha(\beta))$, because $0\si(\alpha(\hat{\beta}))$ is simply $0\alpha(\hat{\beta})$ without the first `1', and $U_1(1)=1$. Therefore, since $\w^\f\prec 0\si(\alpha(\beta))$, we also have $\hat{\w}^\f\prec 0\si(\alpha(\hat{\beta}))$, and hence $\hat{t}_R<1-1/\hat{\beta}$.

Now let $\z\in\Kt_\beta(t_R)$ be given, and assume that $\z$ begins with $0$ {by Lemma \ref{lem:extensibility}}. Since $\si^n(\z)\lge {b(t_R, \beta)=}\w^\f$ for all $n\geq 0$ and $\w$ begins with $01$, $\z$ does not contain the word $00$. Hence, $\z=U_1(\hat{\z})$ for some sequence $\hat{\z}$ which again begins with $0$. As in Case (A), we can deduce from Lemma \ref{lem:upper-bound-equivalence-eta} that
\[
b(\hat t_R, \hat\beta)=\hat\w^\f\lle\si^k(\hat{\z})\lle\alpha(\hat{\beta})\quad\textrm{for all }k\geq 0,
\]
where for $k=0$ the last inequality holds since $\hat{\z}$ begins with $0$.
Hence, $\hat{\z}\in \Kt_{\hat{\beta}}(\hat{t}_R)$ with $\hat t_R<1-1/\hat{\beta}$.
The induction hypothesis now comes in to complete the proof in the same way as in Case (A), using Lemma \ref{lem:upper-bound-equivalence-eta} instead of Lemma \ref{lem:upper-bound-equivalence-xi}.

{
This concludes the proof for $\beta\in E_L\cap(1,2]$. Assume next that $\beta\in E_L\cap(M,M+1]$ for some integer $M\geq 2$. Write $\al(\beta)=\al_1\al_2\dots$. Let $[t_L,t_R]$ be a $\beta$-Lyndon interval generated by a $\beta$-Lyndon word $\w$. Note that $\al(\beta)\in\{M-1,M\}^\N$ by Lemma \ref{lem:two-digits}. Furthermore, since $\beta>\beta_r^{(M-1)}$, we have $\al(\beta)\succ \al\big(\beta_r^{(M-1)}\big)=M(M-1)^\f$, so there is an index $n\geq 2$ such that $\al_n=M$. Since $\al_1=M$ and $\al(\beta)$ is balanced, it follows that $\al(\beta)$ does not contain the word $(M-1)^n$. If $\w=M-1$ or $\w$ begins with a digit $d\leq M-2$, it follows immediately that we can take $\v=(M-1)^n$ in Definition \ref{def:property-TL}, showing that $\w$ has property TL.

Otherwise, $\w\in\{M-1,M\}^*$ and $\w\neq M-1$. Now any sequence $\z\in\Kt_\beta(t_R)$ must also lie in $\{M-1,M\}^\N$. Put $\tilde{\w}:=\theta^{-(M-1)}(\w)$, $\tilde{\z}:=\theta^{-(M-1)}(\z)$ and $\tilde{\beta}:=\phi^{-(M-1)}(\beta)$, so that $\al(\tilde{\beta})=\theta^{-(M-1)}(\al(\beta))$. Then $\tilde{\beta}\in E_L\cap(1,2]$, so by the case $\beta\in(1,2]$ proved above, there is a word $\tilde{\v}\in\{0,1\}^*$ witnessing, on behalf of $\tilde{\z}$, that $\tilde{\w}$ satisfies property TL for $\tilde{\beta}$. The word $\v:=\theta^{M-1}(\tilde{\v})$ then witnesses on behalf of $\z$ that $\w$ satisfies property TL for $\beta$.

Thus, in all cases, Lemma \ref{lem:transitivity-simplified} implies that $\Kt_\beta(t_R)$ is transitive.
}
\end{proof}

\subsection{Density of $\beta$-Lyndon intervals} \label{subsec:dense-intervals}

In this subsection we show that the $\beta$-Lyndon intervals are dense in $[0,1-1/\beta]$ for $\beta\in\overline{E}$.

\begin{lemma} \label{lem:Lyndon-implies-beta-Lyndon}
Let {$\beta\in\overline{E}\cap(M,M+1]$ for $M\in\N$}, and let $\w$ be a Lyndon word such that
\begin{enumerate}[{\rm (i)}]
\item $\w0^\f$ is a greedy $\beta$-expansion; and
\item $\w0^\f\prec {(M-1)}\si(\al(\beta))$.
\end{enumerate}
Then $\w$ is $\beta$-Lyndon.
\end{lemma}

\begin{proof}
{We show this first for $M=1$, so $\beta\in\overline{E}\cap(1,2]$.}
We use induction on the length of $\w$. Precisely, we prove: For any Lyndon word $\w$, if $\beta\in\overline{E}$ is such that (i) and (ii) are satisfied, then $\w$ is $\beta$-Lyndon.

If $\w=01$, then (ii) implies $\al(\beta)\lge 110^\f$, so $\si^n(\w^\f)\lle(10)^\f\prec\al(\beta)$ for each $n\geq 0$. Hence $\w$ is $\beta$-Lyndon.

Now let $k\geq 3$ and assume the statement is true for all Lyndon words of length $<k$. {Let $\w=w_1\dots w_k$ be a Lyndon word of length $k$, and let $\beta\in\overline{E}$ be such that (i) and (ii) are satisfied. We first deal with the special case when $\al(\beta)=11(01)^\f$. By assumption (ii), $\w0^\f\prec (01)^\f$. Since $\w$ is Lyndon and $\w\neq 01$, $\w$ must therefore begin with $00$. Then for any $0\leq j<k$, we have
\[
w_{j+1}\dots w_k w_1w_2\lle \al_1\dots\al_{k-j}00\prec\al_1\dots\al_{k-j+2},
\]
and so $\si^j(\w^\f)\prec\al(\beta)$, proving that $\w$ is $\beta$-Lyndon.

Assume from now on that $\al(\beta)\neq 11(01)^\f$. By Lemma \ref{lem:alpha-renormalizing} there is then a base $\hat{\beta}\in \overline{E}$ such that $0\al(\beta)=U_0(\al(\hat{\beta}))$ or $\al(\beta)=1U_1(\al(\hat{\beta}))$.
}

First assume the former, i.e. $0\al(\beta)=U_0(\al(\hat{\beta}))$. Then $\al(\beta)$ does not contain the word $11$, so $\w$ does not contain this word either by assumption (i). Since $\w\in L^*$, this implies $\w=U_0(\hat{\w})$ for some $\hat{\w}\in L^*$. Furthermore, $U_0(\hat{\w}0^\f)=\w0^\f$. Thus, by Lemma \ref{lem:upper-bound-equivalence-xi} {and (i)}, $\hat{\w}0^\f$ is a greedy $\hat{\beta}$-expansion. It also follows from (ii) and \eqref{eq:shifted-substitution} that $\hat{\w}0^\f\prec 0\si(\al(\hat{\beta}))$. Hence, by the induction hypothesis, $\hat{\w}\in L^*(\hat{\beta})$. But then, again by Lemma \ref{lem:upper-bound-equivalence-xi}, $\w$ is $\beta$-Lyndon.

Next, suppose $\al(\beta)=1U_1(\al(\hat{\beta}))$. Then $\al(\beta)$ does not contain the word $00$, so if $\w$ begins with $00$, then $\si^n(\w^\f)\prec\al(\beta)$ for all $n\geq 0$: setting $m:=n\!\!\mod |\w|$, we have {by (i) that}
\begin{align*}
\si^n(\w^\f)&=w_{m+1}\dots w_{|\w|}\w^\f\lle \al_1\dots \al_{|\w|-m}\w^\f\\
&\prec\al_1\dots\al_{|\w|-m}\si^{|\w|-m}(\al(\beta))=\al(\beta).
\end{align*}
So assume $\w$ begins with $01$. Then $\w=U_1(\hat{\w})$ for some $\hat{\w}\in L^*$ because $\w\in L^*$. Now observe that $\al(\beta)=1U_1(\al(\hat{\beta}))$ implies $\si^n(\al(\beta))\succ(01)^\f$ for all $n\geq 0$, so $\w(01)^\f=U_1(\hat{\w}0^\f)$ is also a greedy $\beta$-expansion. Thus, by Lemma \ref{lem:upper-bound-equivalence-eta}, $\hat{\w}0^\f$ is a greedy $\hat{\beta}$-expansion.
The rest of the proof goes as in the first case above, using Lemma \ref{lem:upper-bound-equivalence-eta} instead of Lemma \ref{lem:upper-bound-equivalence-xi}.

{
Next, assume $M\geq 2$. Note $\al(\beta)=:\al_1\al_2\dots\in\{M-1,M\}^\N$ by Lemma \ref{lem:two-digits}. 
If $\w=w_1\dots w_m$ begins with a digit $d\leq M-2$, then for each $j<m$ assumption (i) yields $w_{j+1}\dots w_m\lle \al_1\dots\al_{m-j}$ and so $w_{j+1}\dots w_m w_1\lle \al_1\dots\al_{m-j}(M-2)\prec\al_1\dots\al_{m-j+1}$. Hence, $\w$ is $\beta$-Lyndon. Furthermore, if $\w=M-1$ then $\w$ is clearly $\beta$-Lyndon since $\al(\beta)$ begins with $M$.

This leaves the case when $\w\in\{M-1,M\}^*$ and $\w\neq M-1$. But now the argument for the case $M=1$ above gives the conclusion, by simply replacing the alphabet $\{0,1\}$ with $\{M-1,M\}$.
}
\end{proof}

\begin{proposition} \label{prop:dense-intervals}
For each $\beta\in\overline{E}\cap{(1,\f)}$, the $\beta$-Lyndon intervals are dense in $\big(0,1-\frac{1}{\beta}\big)$.
\end{proposition}

\begin{proof}
Fix $t\in\big(0,1-\frac{1}{\beta}\big)$ and write $b(t,\beta)=t_1 t_2\dots$. We will show that $t$ either lies in a $\beta$-Lyndon interval or else $t\in{\EE_\beta}$, in which case $t$ can be approximated arbitrarily closely by endpoints of $\beta$-Lyndon intervals by Lemma \ref{lem:E-beta-next-to-beta-Lyndon}.

{Since $t<1-\frac{1}{\beta}$, $t_1t_2\ldots \prec (M-1)\si(\al(\beta))$.} Suppose $t\not\in{\EE_\beta}$. Then $\si^k((t_i))\prec (t_i)$ for some $k$; choose $k$ as small as possible with this property. Then $\w:=t_1\dots t_k$ is either Lyndon or a power of a Lyndon word. Assume the former; in the latter case we replace $k$ with $k'<k$ so that $t_1\dots t_{k'}$ is Lyndon and $k'\mid k$. By Lemma \ref{lem:Lyndon-implies-beta-Lyndon}, $\w$ is $\beta$-Lyndon. Since furthermore, $\si^k((t_i))\prec (t_1\dots t_k)^\f$, it follows that $t$ lies in the $\beta$-Lyndon interval generated by $\w$.
\end{proof}

The conclusion of Proposition \ref{prop:dense-intervals} may fail if $\beta\not\in\overline{E}$; in fact the $\beta$-Lyndon intervals may not even be dense in $[0,\tau(\beta)]$, as the following example shows. 

\begin{example} \label{ex:non-dense-intervals}
Let $\beta$ be any base such that $\alpha(\beta)$ begins with $111\, 01\, 011\,001$. Set $\s=011$. Then $\beta$ lies in the basic interval $[\beta_\ell^{\s},\beta_*^\s]$, where $\alpha(\beta_\ell^{\s})=\L(\s)^\f=(110)^\f$ and $\alpha(\beta_*^\s)=\L(\s)^+\s^-\L(\s)^\f=111\,010\,(110)^\f$. Hence, by \cite[Theorem 2]{Allaart-Kong-2021},
\[
\tau(\beta)=\pi_\beta\big(\s^-\L(\s)^\f\big)=\pi_\beta\big(010(110)^\f\big).
\]
Let $\w=01010111$ and $\tilde{\w}=01011$. Note that $\w$ and $\tilde{\w}$ are both $\beta$-Lyndon words, and $\w^\f\prec \tilde{\w}^\f\prec 010(110)^\f$. So the (symbolic) $\beta$-Lyndon intervals $[\w0^\f,\w^\f]$ and $[\tilde{\w}0^\f,\tilde{\w}^\f]$ are disjoint and their projections under $\pi_\beta$ both lie inside $[0,\tau(\beta)]$. But it is not difficult to see that there is no other $\beta$-Lyndon interval in between these intervals, because there is no $\beta$-Lyndon word that extends $\w$. Hence, the $\beta$-Lyndon intervals are not dense in $[0,\tau(\beta)]$. We will have more to say about this example in Example \ref{ex:non-transitivity-window}.
\end{example}



{
\begin{lemma} \label{lem:empty}
Let $\beta=\beta_\ell^\s$ for $\s\in F_e$. Then $\mathcal{K}_\beta(1-1/\beta)=\emptyset$.
\end{lemma}

\begin{proof}
For $\s\in F^*$, this was shown in the proof of \cite[Proposition 5.2]{Kalle-Kong-Langeveld-Li-18}. If $F_e\ni\s=\theta^k(\s')$ for some $k\geq 1$ and $\s'\in F^*$, the result follows in the same way. Finally, if $\s=k$ with $k\in\N$, then $\beta=\beta_\ell^\s=k+1$, so $\al(\beta)=k^\f$ and $b(1-1/\beta,\beta)=k0^\f$, so
\[
\mathcal{K}_\beta(1-1/\beta)=\{\z\in A_\beta^\f: k0^\f\lle \si^n(\z)\prec k^\f\ \forall n\geq 0\}=\emptyset.
\]
In all cases, the lemma follows.
\end{proof}
}

\begin{corollary} \label{cor:E-B-left-endpoints}
If $\beta=\beta_\ell^\s$ is the left endpoint of a Farey interval $[\beta_\ell^\s,\beta_r^\s]$ {with $\s\in F_e$}, then $\BB_\beta=\EE_\beta$.
\end{corollary}

\begin{proof}
Note that $\beta=\beta_\ell^\s\in E_L$, so $\Kt_\beta(t_R)$ is transitive for every $\beta$-Lyndon interval $[t_L,t_R]$ by Theorem \ref{thm:transitive-in-E_L}. Furthermore, the $\beta$-Lyndon intervals are dense in $[0,1-1/\beta]$ by Proposition \ref{prop:dense-intervals}. Since $\al(\beta)=\L(\s)^\f$, {for $\w^\f=b(t_R, \beta)$} we have
\[
\Kt_\beta(t_R)=\set{\z\in{A_\beta}^\N: \w^\f\lle \si^n(\z)\lle \L(\s)^\f~\forall n\ge 0},
\]
so $\Kt_\beta(t_R)$ is a SFT, hence sofic. Proposition \ref{prop:E-B-equality} now implies that there are no points of $\EE_\beta\backslash\BB_\beta$ in $[0,1-1/\beta)$. On the other hand, $\EE_\beta\cap[1-1/\beta,1]\subseteq \mathcal{K}_\beta(1-1/\beta)=\emptyset$ {by Lemma \ref{lem:empty}}. Hence, $\EE_\beta=\BB_\beta$.
\end{proof}

\section{Right endpoints of first-order basic intervals} \label{sec:right-endpoints}

Our main result in this section is the following.

\begin{theorem} \label{thm:basic-interval-right-endpoint}
Let $\beta_*=\beta_*^\s$ be the right endpoint of a basic interval generated by $\s\in{F_e}$. Then for any $\beta_*$-Lyndon interval $[t_L,t_R]$ in $[0,\tau(\beta_*)]$, $\Kt_{\beta_*}(t_R)$ is a transitive sofic subshift.
In particular, the conclusion of Theorem \ref{thm:general-transitivity} holds for $\beta=\beta_*$.
\end{theorem}

We prove the theorem after establishing a useful lemma, which uses the following fact.

\begin{lemma} \label{lem:Farey-sandwich}
For $\s\in {F_e}$, define the set
\[
\mathcal{X}(\s):=\{\z\in {\N_0}^\N: \s 0^\f\lle \sigma^n(\z)\lle \L(\s)^\f\ \forall n\geq 0\}.
\]
Then $\mathcal{X}(\s)=\{\sigma^j(\s^\f): 0\leq j<|\s|\}$.
\end{lemma}

{
\begin{proof}
For $\s\in F^*$, this was proved in \cite[Proposition 4.4]{Kalle-Kong-Langeveld-Li-18}. If $\s=k\in\N$, we see immediately that $\mathcal{X}(\s)=\{\s^\f\}$. For all other $\s\in F_e$, the result follows from the case $\s\in F^*$ by translation.
\end{proof}
}

\begin{lemma} \label{lem:begin-with-s}
Let $\beta\in(\beta_\ell,\beta_*]$, where $[\beta_\ell,\beta_*]$ is a basic interval generated by { an extended} Farey word $\s=s_1\dots s_m{\in F_e}$, and let $[t_L,t_R]$ be a $\beta$-Lyndon interval in $(0,\tau(\beta))$. Let $\z\in \Kt_{\beta}(t_R)$.
Then $\z$ can be extended to the left to a sequence $\z'\in \Kt_{\beta}(t_R)$ beginning with $\s$, and also to a sequence $\z''\in \Kt_{\beta}(t_R)$ beginning with $\L(\s)$.
\end{lemma}

\begin{proof}
{Let $M$ be the integer such that $\beta\in(M,M+1]$, so $A_\beta=\{0,1,\dots,M\}$ and $\s\in\{M-1,M\}^*$. Assume first that $\s\neq M-1$.
We claim that for $\z\in \Kt_\beta(t_R)$, at least one of the sequences $M\z$ and $(M-1)\z$ also belongs to $\Kt_\beta(t_R)$. Suppose $M\z\not\in\Kt_\beta(t_R)$. Then $M\z\succ\al(\beta)$, and hence 
$$(M-1)\z\succ (M-1)\si(\al(\beta))=b(1-1/\beta,\beta)\succ b(t_R,\beta),$$ 
since $t_R<\tau(\beta)\leq 1-1/\beta$. Therefore, $(M-1)\z\in\Kt_\beta(t_R)$.}

Set $\z^{(0)}:=\z$, and, for $k=0,1,2,\dots$, proceed inductively as follows. If $\z^{(k)}$ has been constructed and does not begin with the prefix $\s$, set $\z^{(k+1)}={M}\z^{(k)}$ if ${M}\z^{(k)}$ is legal in $\Kt_\beta(t_R)$, or $\z^{(k+1)}={(M-1)}\z^{(k)}$ otherwise. We end the procedure if we reach a sequence $\z^{(k)}$ beginning with $\s$.

Suppose the procedure never ends. Let $k\geq 1$, and write $\z^{(k)}=z_1^{(k)}z_2^{(k)}\dots$. 
 Since $\beta\in(\beta_\ell,\beta_*]$, $\al(\beta)$ begins with $\L(\s)^+$. So, if $z_1^{(k)}={M-1}$, this means ${M}\si(\z^{(k)})\succ \al(\beta)$ and hence
\[
{M}z_2^{(k)}\dots z_m^{(k)}\lge \L(\s)^+,
\]
which implies
\[
z_1^{(k)}z_2^{(k)}\dots z_m^{(k)}={(M-1)}\si(1z_2^{(k)}\dots z_m^{(k)})\lge {(M-1)}\si(\L(\s)^+)=\s,
\]
where the last equality follows from Lemma \ref{lem:Farey-property} (i). Since $\z^{(k)}$ does not begin with $\s$, this means that $z_1^{(k)}z_2^{(k)}\dots z_m^{(k)}\succ \s$. Clearly, this inequality holds also when $z_1^{(k)}={M}$.

On the other hand, for $k>m$ we claim that $\z^{(k)}$ cannot begin with $\L(\s)^+$. For, suppose it did. Then, since $\al(\beta)\lle \al(\beta_*)=\L(\s)^+\s^-\L(\s)^\f$ and $\z^{(k)}$ is legal in $\Kt_\beta(t_R)$, we would have
\[
z_1^{(k-m)}\dots z_m^{(k-m)}=z_{m+1}^{(k)}\dots z_{2m}^{(k)}\lle \s^-\prec\s,
\]
contradicting what we just showed. Hence, we also have $z_1^{(k)}\dots z_m^{(k)}\lle \L(\s)$.

We have now constructed a sequence, extending infinitely to the left, all of whose subwords of length $m$ are strictly greater than $\s$ but smaller than or equal to $\L(\s)$. But such a sequence cannot exist, in view of Lemma \ref{lem:Farey-sandwich}. This contradiction shows that we eventually end up with an extension $\z'$ of $\z$ beginning with $\s$. But then $\s\z'\in\Kt_\beta(t_R)$ as well, and since $\s\z'$ begins with $\s^2$ and $\s^2$ contains the word $\L(\s)$, we can extend $\z$ to a sequence $\z''\in \Kt_{\beta}(t_R)$ beginning with $\L(\s)$.

{ If $\s=M-1$, the argument is simpler: Since $\al(\beta)$ begins with $M$ and $b(t_R,\beta)\prec b(\tau(\beta),\beta)=\s^-\L(\s)^\f$, $b(t_R,\beta)$ begins with a digit $d\leq M-2$. Hence, $\s\z=(M-1)\z\in\Kt_\beta(t_R)$ whenever $\z\in\Kt_\beta(t_R)$.
}
\end{proof}

\begin{proof}[Proof of Theorem \ref{thm:basic-interval-right-endpoint}]
{Let $M$ be the integer such that $\beta_*\in(M,M+1]$, so $\s\in\{M-1,M\}^*$. We split the proof in two cases.

\medskip
{\em Case 1.} Assume first that $|\s|\geq 2$. Then $\s=\theta^k(\tilde{\s})$ for some $k\geq 0$ and $\tilde{\s}\in F^*$.
Let $\tilde{\s}=\tilde{\s}^{(1)}\tilde{\s}^{(2)}$ be the {\em standard factorization} of $\s$ into Farey words $\tilde{\s}^{(1)}$ and $\tilde{\s}^{(2)}$. It is well known that such a factorization exists and is unique (cf.~\cite{Carminati-Isola-Tiozzo-2018}). Put $\s^{(i)}:=\theta^k(\tilde{\s}^{(i)})$ for $i=1,2$. Then $\s=\s^{(1)}\s^{(2)}$ and we call this the standard factorization of $\s$.}
We can write
\[
\alpha(\beta_*)=\L(\s)^+\s^-\L(\s)^\f=\L(\s)^+\s^{(1)}\s^\f,
\]
where the second equality can be deduced from \cite[Lemma 2.6]{Carminati-Isola-Tiozzo-2018} and the palindrome property of Farey words (see Lemma \ref{lem:Farey-property} (ii)).

Recall from \cite[Theorem 2]{Allaart-Kong-2021} that
\[
\tau(\beta_*)={\pi_{\beta_*}\big( \s^-\L(\s)^\f\big)}.
\]
Let $[t_L,t_R]$ be a $\beta_*$-Lyndon interval and assume $t_R<\tau(\beta_*)$; then $b(t_R,\beta_*)\prec \s^-\L(\s)^\f=\s^{(1)}\s^\f$. So, letting $\w$ be the word such that $b(t_R,\beta_*)=\w^\f$, we can choose an integer $N$ so large that
\begin{equation} \label{eq:w-bound-1}
\w^\f\prec \s^{(1)}\s^N 0^\f,
\end{equation}
and
\begin{equation} \label{eq:w-bound-2}
\sigma^n(\w^\f)\prec \L(\s)^+\s^{(1)}\s^N 0^\infty \quad \forall n\geq 0.
\end{equation}
(The latter condition holds for large enough $N$ since $\w\in L^*(\beta_*)$ and $\w^\f$ is periodic.)

Now let $\u=u_1\dots u_r\in\cL(\Kt_{\beta_*}(t_R))$ and $\z\in \Kt_{\beta_*}(t_R)$. By Lemma \ref{lem:begin-with-s} we may assume that $\z$ begins with the word $\s$. We also extend $\u$ to the right to a word $\u'$ as follows: Let $k$ be the largest integer such that $u_{r-k+1}\dots u_r=\alpha_1\dots\alpha_k$, where $\alpha_1\alpha_2\dots:=\alpha(\beta_*){=\L(\s)^+\s^{(1)}\s^\f}$. If no such $k$ exists, set $k=0$. Then set
\[
\u':=u_1\dots u_{r-k}\L(\s)^+\s^{(1)}\s^{N_1},
\]
where $N_1\geq N$ is chosen large enough so that $\u'$ actually extends $\u$. It is not difficult to see using \eqref{eq:w-bound-1} and \eqref{eq:w-bound-2} that $\u'\in\cL(\Kt_{\beta_*}(t_R))$.

Now if in fact $\z\lle \s^\f$, we have immediately that $\u'\z\in\Kt_{\beta_*}(t_R)$, in view of the inequalities
\[
s_{i+1}\ldots s_ms_1\ldots s_i\prec\L(\s)^+\quad \textrm{and}\quad s_{i+1}\ldots s_m\succ s_1\ldots s_{m-i}\lge w_1\ldots w_{m-i}\quad \forall\, 0\le i<m.
\]
Otherwise, there is an integer $l\geq 1$ and a word $\v\succ\s$ such that $\z$ begins with $\s^l\v$. Now observe that $\s^{(1)}\z\succ \s^{(1)}\s^\f\succ \w^\f$, and also $\s^{(1)}\z\prec \s^\f$, since $\s$ being Lyndon implies $\s=\s^{(1)}\s^{(2)}\prec \s^{(2)}\s^{(1)}$ and hence $\s^{(1)}\s\prec \s\s^{(1)}$. Using the definition of $\u'$ and the assumptions \eqref{eq:w-bound-1} and \eqref{eq:w-bound-2} it now follows readily that
\[
\w^\f \lle \sigma^n(\u'\s^{(1)}\z) \lle \alpha(\beta)\qquad\forall n\geq 0,
\]
and therefore, $\u'\s^{(1)}\z\in\Kt_{\beta_*}(t_R)$. As a result, $\Kt_{\beta_*}(t_R)$ is transitive.

{
\medskip
{\em Case 2.} Next, assume $|\s|=1$, so $\s=M-1$. There is no sensible definition of standard factorization of $\s$, but put $\s^{(1)}:=M-2$. Then $\s^-\L(\s)^\f=(M-2)(M-1)^\f=\s^{(1)}\s^\f$ as in Case 1, and the rest of the proof proceeds in the same way as above. (Note in particular that the inequality $\s^{(1)}\z\prec \s^\f$ is now trivial.)
}
\end{proof}

\section{The case of finitely renormalizable $\beta$} \label{sec:relative-exceptional}

In this section we consider values of $\beta$ in the set $\overline{E^\cs}$ for $\cs\in\La$. Recall $E^\cs$ is defined in \eqref{eq:relative-E}.
Recall further that $\mathcal{T}_R(\beta)$ is the set of all right endpoints of $\beta$-Lyndon intervals in $[0,\tau(\beta)]$.

\begin{theorem} \label{thm:beta-in-relative-exceptional-set}
Let $\beta\in \overline{E^\cs}$ for some $\cs\in\La$. Then:
\begin{enumerate}[{\rm(i)}]
\item For every $t_R\in\mathcal{T}_R(\beta)$, $\Kt_\beta(t_R)$ has a transitive subshift $\mathcal{K}_\beta'(t_R)$ of full entropy and full Hausdorff dimension that contains the sequence $b(t_R,\beta)$. Moreover, these transitive subshifts can be chosen so that $\{\mathcal{K}_\beta'(t_R): t_R\in\mathcal{T}_R(\beta)\}$ is a strictly descending collection of subshifts, and if $\al(\beta)$ is eventually periodic, then $\mathcal{K}_\beta'(t_R)$ is sofic for each $t_R\in\mathcal{T}_R(\beta)$.
\item The $\beta$-Lyndon intervals are dense in $[0,\tau(\beta)]$.
\end{enumerate}
\end{theorem}

Note that (i) implies Theorem \ref{thm:general-transitivity} for $\beta\in\overline{E^\cs}$ {with $\mathcal I=\emptyset$}.
We first establish a series of lemmas.

\begin{lemma} \label{lem:four-blocks}
Let $\r\in\La$, and suppose a sequence $\z\in{\N_0}^\N$ begins with $\r^-$ or $\L(\r)^+$ and satisfies
\begin{equation} \label{eq:x-y-sandwich}
\Phi_\r(0^\f)\lle \si^n(\z)\lle \Phi_\r(1^\f)\qquad\forall\,n\geq 0.
\end{equation}
Then  $\z=\Phi_\r(\hat{\z})$ for some sequence $\hat{\z}\in\set{0,1}^\N$.
\end{lemma}
 
\begin{proof}
From Definition \ref{def:substitution}, we see that $\Phi_\r(0^\f)=\r^-\L(\r)^\f$ and $\Phi_\r(1^\f)=\L(\r)^+\r^\f$. So by \eqref{eq:x-y-sandwich},
\[
\r^-\L(\r)^\f \lle \sigma^n(\z)\lle \L(\r)^+\r^\f \qquad \forall\, n\geq 0,
\]
and hence any block $\r^-\L(\r)^l$ (with $l\geq 0$) in $\z$ must be followed by $\L(\r)$ or $\L(\r)^+$, whereas any block $\L(\r)^+\r^m$ (with $m\geq 0$) in $\z$ must be followed by $\r$ or $\r^-$. This shows $\z\in X(\r)$, and so $\z=\Phi_\r(\hat{\z})$ for some sequence $\hat{\z}$.
\end{proof}
 
For any $\cs\in\La$, we define the set
\begin{equation} \label{eq:Gamma}
\Ga(\cs):=\{\z\in{\{0,1,\dots,M\}}^\N: \cs^\f\lle \sigma^n(\z)\lle \L(\cs)^\f\ \ \forall n\geq 0\},
\end{equation}
{where $M$ is the integer such that $J^\cs\subset(M,M+1]$.}
The next result, which extends Lemma \ref{lem:Farey-sandwich} to arbitrary words in $\La$, was proved in \cite[Proposition 4.1]{Allaart-Kong-2021} {for the case $M=1$. It also follows, for $M=1$,} from \cite[Lemma 2.12]{Glendinning-Sidorov-2015} via the relationship \eqref{eq:Omega-and-Sigma}, or can alternatively be deduced, with some effort, from \cite[Theorem 2.5]{Komornik-Steiner-Zou-2022}.
{We extend this lemma here to arbitrary $M$ and provide a shorter proof.}

\begin{lemma} \label{lem:countable}
Let $\cs\in\La$. Then the set $\Ga(\cs)$ is countable.
\end{lemma}

\begin{proof}
{For the purpose of this proof, let
\[
\La_k^*:=\{\s_1\bullet\dots\bullet\s_k: \s_i\in F^*\ \mbox{for $i=1,\dots,k$}\}, \qquad k\in\N.
\]
Clearly, $\La_k^*\subseteq \La_k$.}
We proceed by induction on the degree $k$ of $\cs$.
The case $k=1$ follows from Lemma \ref{lem:Farey-sandwich}. Let $k\geq 2$, and suppose the statement holds for all $\cs\in\La_{k-1}$. {In particular, it holds for all $\cs\in\La_{k-1}^*$.} Take $\cs\in\La_k$, and write $\cs=\r\bullet\s$ where $\r\in {F_e}$ and $\s\in{\La_{k-1}^*}$. Let $\z\in\Ga(\cs)$. Since $\r\bullet\s$ begins with $\r^-$ and $\L(\r\bullet\s)$ begins with $\L(\r)^+$, it follows that
\begin{align*}
\r^-0^\f&\prec \Phi_\r(\s^\f)=(\r\bullet\s)^\f\lle \sigma^n(\z)\\
&\lle \L(\r\bullet\s)^\f=\Phi_\r(\L(\s)^\f)\prec \L(\r)^+1^\f\qquad\forall n\geq 0,
\end{align*}
so either $\z\in\Ga(\r)$, or else $\z$ eventually contains $\r^-$ or $\L(\r)^+$, which by Lemma \ref{lem:four-blocks} implies $\sigma^k(\z)=\Phi_\r(\hat{\z})$ for some $k\geq 0$ and sequence $\hat{\z}$. But $\z\in\Ga(\r\bullet\s)$ implies that {$\Phi_\r(\hat\z)=\si^k(\z)\in\Ga(\r\bullet\s)$, and thus} $\hat{\z}\in\Ga(\s)$ by Lemma \ref{lem:substitution-properties} (i). Therefore,
\[
\Ga(\cs)=\Ga(\r\bullet\s)\subseteq\Ga(\r)\cup\bigcup_{\d\in{\N_0}^*}\d\;\Phi_\r(\Ga(\s)).
\]
Since both $\Ga(\r)$ and $\Ga(\s)$ are countable by the induction hypothesis, we conclude that $\Ga(\cs)$ is countable as well.
\end{proof}

The next lemma, {which is specific to $\beta\in(1,2]$,} describes the interplay between the substitution map $\Phi_\r$ and the maps $U_0$ and $U_1$.

{

\begin{lemma}  \label{lem:substitution-commute-c}
  Let $\r\in\F$. Suppose $\alpha(\beta)=\Phi_\r(\y)$ for some sequence $\y\in\{0,1\}^\N$. 
  \begin{enumerate}[{\rm(i)}]
    \item If $\r=U_0(\tilde{\r})$ for some $\tilde{\r}\in\F$, then there exists $\tilde\beta$ such that 
    \[
    \alpha(\tilde{\beta})=\Phi_{\tilde{\r}}(\y)\quad \textrm{and}\quad 0\alpha(\beta)=U_0(\alpha(\tilde{\beta})).
    \]
    \item If $\r=U_1(\tilde{\r})$ for some $\tilde{\r}\in\F$,  then there exists $\tilde\beta$ such that 
    \[
    \alpha(\tilde{\beta})=\Phi_{\tilde{\r}}(\y)\quad \textrm{and}\quad  \alpha(\beta)=1U_1(\alpha(\tilde{\beta})).
    \]
  \end{enumerate}
\end{lemma}

\begin{proof}
Since the proofs of (i) and (ii) are similar, we only prove (i).
  Suppose   $\r=U_0(\tilde{\r})$. Since $\alpha(\beta)=\Phi_\r(\y)$ begins with 1 and $\Phi_\r(0)=\r^-$ begins with 0, it follows that $\y$ begins with 1. Write \[
\mathbf y=1^{k_1+1}0^{l_1+1}1^{k_2+1}0^{l_2+1}\ldots,
\]
where $k_i, l_i$ are nonnegative integers, possibly one of them taking the value $+\infty$. Then, from   the definition \eqref{eq:block-map} and Lemma \ref{lem:eta-and-xi} (i) we obtain
\begin{align*}
0\alpha(\beta)=0\Phi_\r(\y)&=0\L(\r)^+\r^{k_1}\r^-\L(\r)^{l_1}\L(\r)^+\r^{k_2}\r^-\L(\r)^{l_2}\dots\\
&=U_0(\L(\tilde\r)^+\tilde\r^{k_1}\tilde\r^-\L(\tilde\r)^{l_1}\L(\tilde\r)^+\tilde\r^{k_2}\tilde\r^-\L(\tilde\r)^{l_2}\dots)=U_0(\Phi_{\tilde\r}(\y)).
\end{align*}
So in the following it suffices to prove $\Phi_{\tilde\r}(\y)=\al(\tilde\beta)$ for some $\tilde\beta$. 

Note by Lemma \ref{lem:quasi-greedy expansion-alpha-q} that $\si^n(\al(\beta))\lle \al(\beta)$ for all $n\ge 0$. Since $\al(\beta)=\Phi_\r(\y)$, by Lemma \ref{lem:substitution-properties} (iii) it follows that 
\[
\si^n(\y)\lle \y\quad\forall n\ge 0.
\]
Using Lemma \ref{lem:substitution-properties} (iii) once more, we conclude that 
\[
\si^n(\Phi_{\tilde \r}(\y))\lle \Phi_{\tilde\r}(\y)\quad \forall n\ge 0.
\]
By Lemma \ref{lem:quasi-greedy expansion-alpha-q} this implies $\Phi_{\tilde\r}(\y)=\al(\tilde\beta)$ for some $\tilde\beta$, since $\Phi_{\tilde\r}(\y)$ does not end in $0^\f$. This completes the proof.
\end{proof}
}

\begin{lemma} \label{lem:general-connecting-w}
Let $\r\in{F_e}$ and let $\beta{>1}$ such that $\alpha(\beta)=\Phi_\r(\y)$ for some sequence $\y\succ 10^\f$ {not ending in $0^\f$}.
Let $\w$ be a $\beta$-Lyndon word such that $\w^\f\prec \r^-\L(\r)^\f$. Then for any sequence $\z\in\Kt_\beta(\w)$ there is a word $\v$ such that $\w\v\z\in \Kt_\beta(\w)$.
\end{lemma}

\begin{proof}
{ We consider three cases, depending on the type of $\r$.

{\em Case 1.} $\r\in F^*$.}
{Recall that 
$
\Kt_\beta(\w)=\set{\z\in{A_\beta}^\N: \w^\f\lle \si^n(\z)\lle \al(\beta)~\forall n\ge 0}.
$}
We use induction on the length of the Farey word $\r$. For the purpose of this proof, we will consider $\w=0$ a $\beta$-Lyndon word, and the degenerate interval $[t_L,t_R]=\{0\}$ a $\beta$-Lyndon interval. Note that for this case, the statement of the lemma is trivial.

\medskip
{\em Step 1.} First we take $\r=01$. Without loss of generality we may assume that $\z$ begins with $0$. Since $\w^\f\prec \r^-\L(\r)^\f=00(10)^\f$, there is an integer $M$ such that
\begin{equation} \label{eq:w-upper-bound}
\w^\f\prec 00(10)^M 0^\f.
\end{equation}
{
Since $\y$ does not end in $0^\f$, $\alpha(\beta)=\Phi_\r(\y)$ does not end in $\r^-\L(\r)^\f=00(10)^\f$. Hence any block $00=\r^-$ in $\al(\beta)$ must be followed by $\L(\r)^n\L(\r)^+=(10)^n11$ for some $n\geq 0$. Write $\w=w_1\dots w_l$. Let $J:=\{0\leq j<l: w_{j+1}\dots w_l=\al_1\dots\al_{l-j}\}$. 
For each $j\in J$ we can choose $N_j\geq M$ so that $\si^{l-j}(\al(\beta))\succ 00(10)^{N_j}1^\f$. Now set $N:=\max_{j\in J} N_j$, and $\v:=00(10)^N 1$. Then for each $0\leq j<l$, we have 
\[w_{j+1}\dots w_l00(10)^N1\z\lle \al_1\dots\al_{l-j}00(10)^N1\z\prec\al(\beta).\]
 This implies easily that $\si^n(\w\v\z)\lle\al(\beta)$ for every $n\geq 0$. Furthermore, the Lyndon property of $\w$, $N\geq M$, and \eqref{eq:w-upper-bound} imply that $\si^n(\w\v\z)\lge\w^\f$ for every $n\geq 0$. Hence, $\w\v\z\in \Kt_\beta(\w)$.
}

This completes the basis for the induction.

\medskip
{\em Step 2.} Now let $k\geq 3$ and assume the statement of the lemma holds for all Farey words of length $<k$. Let $\r$ be a Farey word of length $k$. So {by Lemma \ref{lem:characterization-Farey-words}} either $\r=U_0(\tilde{\r})$ or $\r=U_1(\tilde{\r})$ for some Farey word $\tilde{\r}$ of length $<k$.

\medskip
{\bf Case (A):} Suppose $\r=U_0(\tilde{\r})$.  {By Lemma \ref{lem:substitution-commute-c} let $\tilde{\beta}$ be the base such that $0\alpha(\beta)=U_0(\alpha(\tilde{\beta}))$ and $\alpha(\tilde{\beta})=\Phi_{\tilde{\r}}(\y)$. Then} $\al(\beta)$ begins with $10$. Thus, $\w$ cannot contain the word $11$, and so $\w=U_0(\tilde{\w})$ for some $\tilde{\beta}$-Lyndon word $\tilde{\w}$. Furthermore, by similar reasoning as in the proof of Lemma \ref{lem:substitution-commute-c}, we obtain
\begin{equation*}
U_0(\tilde{\w}^\f)=\w^\f\prec \r^-\L(\r)^\f=U_0\big(\tilde{\r}^-\L(\tilde{\r})^\f\big).
\end{equation*}
Thus, $\tilde{\w}^\f\prec \tilde{\r}^-\L(\tilde{\r})^\f$.

Assume without loss of generality that $\z\in\Kt_\beta(\w)$ begins with $0$. It then follows in the same way as before that $\z=U_0(\tilde{\z})$ for some sequence $\tilde{\z}$, and $\tilde{\z}\in \Kt_{\tilde\beta}(\tilde{\w})$. 
Observe that {$\tilde{\r}$, $\tilde{\beta}$ and $\tilde{\w}$}  satisfy the hypotheses of the lemma. Therefore, by the induction hypothesis, there is a word $\tilde{\v}$ such that $\tilde{\w}\tilde{\v}\tilde{\z}\in \Kt_{\tilde\beta}(\tilde{\w})$. Set $\v:=U_0(\tilde{\v})$; then $\w\v\z=U_0(\tilde{\w}\tilde{\v}\tilde{\z})\in \Kt_\beta(\w)$,
as desired.

\medskip
{\bf Case (B):} Suppose $\r=U_1(\tilde{\r})$. {By Lemma \ref{lem:substitution-commute-c} let $\tilde{\beta}$ be the base such that $0\alpha(\beta)=U_1(0\alpha(\tilde{\beta}))$ and $\alpha(\tilde{\beta})=\Phi_{\tilde{\r}}(\y)$}.

Assume without loss of generality that $\z\in\Kt_\beta(\w)$ begins with $0$. Note that $\al(\beta)$ does not contain the word $00$, and $\al(\beta)$ does not end in $(01)^\f$, because if it did, then $\al(\tilde{\beta})$ would end in $0^\f$, which is impossible.  If $\w$ begins with $00$, then we can find a positive integer $M$ such that $(01)^M$ does not occur in $\al(\beta)$, and then $\w(01)^M\z\in\Kt_\beta(\w)$.

So assume $\w$ begins with $01$. Since $\w$ is Lyndon, we can then write $\w=U_1(\tilde{\w})$, where $\tilde{\w}$ is again Lyndon. (Possibly, $\tilde{\w}=0$.) As before, we find that $\tilde{\w}$ is in fact $\tilde{\beta}$-Lyndon.
Since $\w\prec \r^-\L(\r)^\f$, we can deduce in a similar manner as in Case (A) that $\tilde{\w}\prec \tilde{\r}^-\L(\tilde{\r})^\f$.

Since $\z$ begins with $0$, the same reasoning as before yields that $\z=U_1(\tilde{\z})$ for some sequence $\tilde{\z}$. As before, we have $\tilde{\z}\in\Kt_{\tilde\beta}(\tilde{\w})$. By the induction hypothesis, there is a word $\tilde{\v}$ such that $\tilde{\w}\tilde{\v}\tilde{\z}\in \Kt_{\tilde\beta}(\tilde{\w})$. Set $\v:=U_1(\tilde{\v})$; then $\w\v\z=U_1(\tilde{\w}\tilde{\v}\tilde{\z})\in \Kt_\beta(\w)$.

\medskip
{
{\em Case 2.} $\r=\theta^k(\tilde{\r})$ for some $k\in\N$ and $\tilde{\r}\in F^*$. For ease of presentation we take $k=1$; the reader will have no difficulty with the general case. Note $\r\in\{1,2\}^*$ and $\r$ ends in the digit $2$. Then $\al(\beta)=\Phi_\r(\y)\in\{1,2\}^\N$, and $\al(\beta)$ contains infinitely many $2$'s. So we can find a positive integer $l$ such that the word $1^l$ does not occur in $\al(\beta)$. If $\w$ begins with $0$ or $\w=1$, then, regardless of $\z\in\Kt_\beta(\w)$, we simply take $\v=1^l$. 

On the other hand, suppose $\w$ begins with $1$ but $\w\neq 1$, so $\w=\theta(\tilde{\w})$ for some Lyndon word $\tilde{\w}\in\{0,1\}^*$. If $\z\in\Kt_\beta(\w)$, then $\z\in\{1,2\}^\N$, so $\z=\theta(\tilde{\z})$ for some $\tilde{\z}\in\{0,1\}^\N$. Similarly, $\al(\beta)=\theta(\al(\tilde{\beta}))$ for some $\tilde{\beta}\in(1,2)$. Applying Case 1 to $\tilde{\r}$ and $\tilde{\beta}$ we find a word $\tilde{\v}$ such that $\tilde{\w}\tilde{\v}\tilde{\z}\in \Kt_{\tilde\beta}(\tilde{\w})$. Set $\v:=\theta(\tilde{\v})$; then $\w\v\z\in \Kt_\beta(\w)$.

\medskip
{\em Case 3.} $\r=M-1$ for some $M\geq 2$. Again, for simplicity we take $M=2$, so $\r=1$. By assumption, 
$\w^\f\prec \r^-\L(\r)^\f=01^\f$, so $\w$ begins with $01^m0$ for some $m\geq 0$. Given $\z\in\Kt_\beta(\w)$, we can by the first paragraph of the proof of Lemma \ref{lem:begin-with-s} extend $\z$ to the left to a sequence $\z'\in\Kt_\beta(\w)$ beginning with either $2$ or $1^{m+1}$. Put $\v=01^{m+1}0$; we claim that $\w\v\z'\in\Kt_\beta(\w)$. To see this, suppose $\w$ has a suffix that is a prefix of $\al(\beta)$; say $\al_1\dots\al_p$. Then the $\beta$-Lyndon property of $\w$ gives
\begin{equation} \label{eq:alpha-greater}
\al(\beta)\succ\al_1\dots\al_p\w^\f\succ \al_1\dots\al_p 01^m0^\f.
\end{equation}
Since $\al(\beta)=\Phi_1(\y)$, $\al(\beta)$ cannot contain a word $01^j0$ for any $j\geq 0$. Hence \eqref{eq:alpha-greater} implies $\al(\beta)\lge \al_1\dots\al_p 01^\f\succ \al_1\dots\al_p \v\z'$, as desired.
}
\end{proof}

\begin{lemma} \label{lem:renormalizable-inequality}
Suppose $\al(\beta)=\Phi_\r(\y)$ for some Farey word $\r$ and sequence $\y$. Then $b(\tau(\beta),\beta)\lle \r0^\f$.
\end{lemma}

\begin{proof}
Since $\al(\beta)$ begins with $\L(\r)^+$, $\r0^\f$ is a greedy $\beta$-expansion. Let $t\geq \pi_\beta(\r0^\f)$, so that $b(t,\beta)\lge \r0^\f$. We will show that $\Kt_\beta(t)$ is countable. Note that
\begin{align*}
\Kt_\beta(t)&=\{\z\in{A_\beta}^\N: b(t,\beta)\lle \si^n(\z)\lle \al(\beta)\ \forall n\geq 0\}\\
&\subseteq \{\z\in{A_\beta}^\N: \r0^\f\lle \si^n(\z)\lle \al(\beta)\ \forall n\geq 0\}\\
&=\{\z\in{A_\beta}^\N: \r^\f\lle \si^n(\z)\lle \al(\beta)\ \forall n\geq 0\}.
\end{align*}
By Lemma \ref{lem:countable}, there are only countably many sequences $\z$ satisfying $\r^\f\lle \si^n(\z)\lle \L(\r)^\f$ for all $n\geq 0$. The remaining sequences in $\Kt_\beta(t)$ must contain the word $\L(\r)^+$, and this can only be followed by $\r^\f$ because $\al(\beta)=\Phi_\r(\y)$ implies $\al(\beta)\lle\L(\r)^+\r^\f$. Thus, $\Kt_\beta(t)$ is countable, and we conclude that $b(\tau(\beta),\beta)\lle \r0^\f$.
\end{proof}

\begin{proposition} \label{prop:general-transitivity}
Let $\beta$ be renormalizable, so $\al(\beta)=\Phi_\r(\al(\hat{\beta}))$ for some Farey word $\r$ and base $\hat{\beta}\in(1,2]$. Let $t_R\in\mathcal{T}_R(\beta)$. Then $\Kt_\beta(t_R)$ is transitive if and only if $b(t_R,\beta)\prec \r^-\L(\r)^\f$.
\end{proposition}

\begin{proof}
Let $\w$ be the $\beta$-Lyndon word such that $b(t_R,\beta)=\w^\f$.  
Set $\hat{\y}:=\al(\hat{\beta})$, so $\alpha(\beta)=\Phi_\r(\hat{\y})$ and $\hat{\y}\succ 10^\f$. {Furthermore, $\hat{\y}$ does not end in $0^\f$.}

Assume first that $\w^\f\prec \r^-\L(\r)^\f$. Then we can find positive integers $N$ and $k$ such that
\begin{equation} \label{eq:k-N-inequality}
\w^k {M_\beta}^\f\prec \r^-\L(\r)^N 0^\f.
\end{equation}

Let $\u=u_1\dots u_l\in\cL(\Kt_\beta(t_R))$ and $\z\in \Kt_\beta(t_R)$. By Lemma \ref{lem:general-connecting-w} there is a word $\v$ such that $\w\v\z\in \Kt_\beta(t_R)$. It is easy to see that then $\w^k\v\z\in \Kt_\beta(t_R)$ also. Now let $j$ be the smallest positive integer such that $u_{j+1}\dots u_l$ is a prefix of $\w^\f$; or set $j=l$ if no such integer exists.
By choosing $k$ even larger if necessary, we may assume $k|\w|>l-j$. Consider the sequence
\[
\c:=u_1\dots u_j \w^k\v\z.
\]
Then $\c$ begins with $\u$ and ends in $\z$, and $\u$ and $\z$ are separated at least by the word $\v$. We now check that $\c\in\Kt_\beta(t_R)$.

Since $\alpha(\beta)=\Phi_\r(\hat{\y})$, by Lemma \ref{lem:substitution-properties} (iii) it follows that $\sigma^n(\alpha(\beta))\lge \Phi_\r(0^\f)=\r^-\L(\r)^\f$ for all $n\geq 0$. Hence, for $n<j$, \eqref{eq:k-N-inequality} gives
\[
\sigma^n(\c)=u_{n+1}\dots u_j \w^k\v\z\prec u_{n+1}\dots u_j \r^-\L(\r)^N 0^\f\prec\alpha(\beta),
\]
since $u_{n+1}\dots u_j\lle\alpha_1\dots\alpha_{j-n}$, and $\r^-\L(\r)^N 0^\f\prec \alpha_{j-n+1}\alpha_{j-n+2}\dots$.
For $n\geq j$, $\sigma^n(\c)\prec\alpha(\beta)$ follows since $\w^k\v\z\in\Kt_\beta(t_R)$. The other requisite inequality, $\sigma^n(\c)\lge \w^\f$, follows from the choice of $j$. Hence, $\c\in\Kt_\beta(t_R)$. This proves that $\Kt_\beta(t_R)$ is transitive.

Next, assume $\w^\f\lge \r^-\L(\r)^\f$. Note that this actually means $\w^\f\succ \r^-\L(\r)^\f$ since $\w^\f$ is periodic. 
Since {$\w^\f\prec b(\tau(\beta),\beta)\lle \r0^\f$ by Lemma \ref{lem:renormalizable-inequality}, it follows that $\w^\f$ must} begin with $\r^-$. {Since $\w$ is Lyndon, $\si^n(\w^\f)\lge \w^\f\succ \r^-\L(\r)^\f=\Phi_\r(0^\f)$ for all $n\geq 0$. And since $\w$ is in fact $\beta$-Lyndon, $\si^n(\w^\f)\prec\al(\beta)=\Phi_\r(\hat{\y})\lle \Phi_\r(1^\f)$ for all $n\geq 0$. Hence,} by Lemma \ref{lem:four-blocks} it follows that $\w^\f=\Phi_\r(\hat{\mathbf{x}})$ for some sequence $\hat{\mathbf{x}}$.


Again by Lemma \ref{lem:four-blocks}, any sequence $\z\in\Kt_\beta(t_R)$ that begins with $\r^-$ must be of the form $\Phi_\r(\hat{\z})$ for some sequence $\hat{\z}$. Note that $\r^-\in\cL\big(\Kt_\beta(t_R)\big)$ and $\L(\r)^\f\in\Kt_\beta(t_R)$. But, since each block $\L(\r)$ can only be preceded by another block $\L(\r)$ or by $\r^-$, any sequence of the form $\r^-\v\L(\r)^\f$ in $\Kt_\beta(t_R)$ would have to end in $\r^-\L(\r)^\f$, which is impossible since $\r^-\L(\r)^\f\prec\w^\f$. This shows $\Kt_\beta(t_R)$ is not transitive.
\end{proof}

\begin{lemma} \label{lem:sofic-subshift}
Let $\beta{>1}$ be a base such that $\alpha(\beta)$ is eventually periodic, and let $[t_L,t_R]$ be a $\beta$-Lyndon interval. Then $\Kt_\beta(t_R)$ is a sofic subshift.
\end{lemma}

\begin{proof}
It is well known (see \cite{Blanchard-1989}) that the $\beta$-shift $\Sigma_\beta$ is sofic if and only if $\alpha(\beta)$ is eventually periodic. Let $\w$ be the $\beta$-Lyndon word such that $b(t_R,\beta)=\w^\f$, and define the set
\[
\mathcal{X}_\w:=\{\z\in{A_\beta}^\N: \sigma^n(\z)\lge \w^\f\ \forall n\geq 0\}.
\]
Then $\mathcal{X}_\w$ is a SFT and hence sofic. Since the intersection of two sofic shifts is sofic (see \cite[Proposition 3.4.10]{Lind_Marcus_1995}), it follows that $\Kt_\beta(t_R)=\Sigma_\beta\cap \mathcal{X}_\w$ is sofic.
\end{proof}

\begin{proposition} \label{prop:transitive-subshifts}
The statement of Theorem \ref{thm:beta-in-relative-exceptional-set} (i) holds for all $\beta\in E_L^\cs$, where $\cs\in\La$ and
$E_L^\cs:=E^\cs\cup\bigcup_{\r\in\F}\beta_\ell^{\cs\bullet\r}$.
\end{proposition}

\begin{proof}
Write $\cs=\r_1\bullet\dots\bullet \r_k$, where $\r_i\in\F$ for each $i$. Set $\cs_i:=\r_1\bullet\dots\bullet\r_i$ for $i=1,\dots,k$, and note that we can write $\al(\beta)=\Phi_{\cs_i}(\al(\hat{\beta}_i))$ for certain bases $\hat{\beta}_1,\dots,\hat{\beta}_k$, where in particular $\hat{\beta}_k\in E_L{\cap(1,2]}$. Let $\w$ be a $\beta$-Lyndon word and $t_R$ the number such that $b(t_R,\beta)=\w^\f$. {Assume $t_R<\tau(\beta)$.}

If $\w^\f\prec {\r_1^-}\L(\r_1)^\f={\cs_1^-}\L(\cs_1)^\f$, then $\Kt_\beta(t_R)$ is transitive by Proposition \ref{prop:general-transitivity}. Otherwise, there is a {largest} integer $i\geq 1$ such that $\w^\f\succ {\cs_i^-}\L(\cs_i)^\f$. Plainly $i\leq k$. Recall from the proof of Proposition \ref{prop:general-transitivity} that this implies $\w^\f=\Phi_{\cs_i}(\mathbf{x})$ for some sequence $\mathbf{x}$. Since $\w^\f$ is periodic and $\w$ {is Lyndon}, it follows that the sequence $\mathbf{x}$ is also periodic, so we can write $\mathbf{x}=\hat{\w}^\f$ for some word $\hat{\w}$, and it is easy to see that $\hat{\w}$ is $\hat{\beta}_i$-Lyndon, using that $\Phi_{\cs_i}$ is increasing. Let $\hat{t}_R$ be the point such that $b(\hat{t}_R,\hat{\beta}_i)=\hat{\w}^\f$. Now we consider two cases:
\begin{itemize}
\item If $i<k$, then $\w^\f\prec \cs_{i+1}^-\L(\cs_{i+1})^\f$, and so $\hat{\w}^\f\prec \Phi_{\cs_i}^{-1}\big(\cs_{i+1}^-\L(\cs_{i+1})^\f\big)=\r_{i+1}^-\L(\r_{i+1})^\f$. Observe that $\al(\hat{\beta}_i)=\Phi_{\r_{i+1}}(\al(\hat{\beta}_{i+1}))$. So by Proposition \ref{prop:general-transitivity}, $\Kt_{\hat\beta_i}(\hat{t}_R)$ is transitive.
\item If $i=k$, then $\w^\f\prec {b(\tau(\beta),\beta)}=\Phi_\cs(0\hat{\al}_1\hat{\al}_2\dots)$, where $(\hat{\al}_i):=\al(\hat{\beta}_k)$. This implies $\hat{\w}^\f\prec 0\hat{\al}_1\hat{\al}_2\dots$, in other words, $\hat{t}_R<1-(1/\hat{\beta}_k)$. Observe furthermore that $\hat{\beta}_k\in E_L$. Hence, $\Kt_{\hat\beta_k}(\hat{t}_R)$ is transitive by Theorem \ref{thm:transitive-in-E_L}.
\end{itemize}
We conclude in both cases that $\Kt_{\hat\beta_i}(\hat{t}_R)$ is transitive. From here on we simply write $\hat{\beta}:=\hat{\beta}_i$. We also set $\bR:=\cs_i$.

Now we set
\begin{equation} \label{eq:the-transitive-subshift}
\mathcal{K}_\beta'(t_R):=\{\sigma^n(\Phi_{\bR}(\hat{\z})): \hat{\z}\in \Kt_{\hat\beta}(\hat{t}_R),\ n\geq 0\},
\end{equation}
and claim that $\mathcal{K}_\beta'(t_R)$ has the desired properties.

First of all, it is clear that $\mathcal{K}_\beta'(t_R)$ is invariant under $\sigma$, and it is closed because {$\Kt_{\hat\beta}(\hat{t}_R)$} is closed and $\Phi_{\bR}^{-1}$ is continuous on the range $X(\bR)$ of $\Phi_{\bR}$. So $\mathcal{K}_\beta'(t_R)$ is a subshift of ${A_\beta}^\N$.

Next, since each sequence $\hat{\z}\in \Kt_{\hat\beta}(\hat{t}_R)$ satisfies
\[
\hat{\w}^\f\lle\sigma^j(\hat{\z})\lle \alpha(\hat{\beta}) \qquad \forall j\geq 0,
\]
Lemma \ref{lem:substitution-properties} (iii) implies that
\[
\w^\f=\Phi_{\bR}(\hat{\w})^\f\lle \sigma^j(\Phi_{\bR}(\hat{\z}))\lle \Phi_{\bR}(\alpha(\hat{\beta}))=\alpha(\beta) \qquad \forall j\geq 0.
\]
Hence, $\mathcal{K}_\beta'(t_R)$ is in fact a subshift of $\Kt_\beta(t_R)$. Furthermore, \eqref{eq:the-transitive-subshift} implies that $\mathcal{K}_\beta'(t_R)$ is a factor of $\Kt_{\hat\beta}(\hat{t}_R)$. If $\al(\beta)$ is eventually periodic, then so is $\al(\hat{\beta})=\Phi_\bR^{-1}(\al(\beta))$; hence $\Kt_{\hat\beta}(\hat{t}_R)$ is sofic by Lemma \ref{lem:sofic-subshift}. But a factor of a sofic subshift is sofic (see \cite[Corollary 3.2.2]{Lind_Marcus_1995}), and so $\mathcal{K}_\beta'(t_R)$ is sofic.

Note that $b(t_R,\beta)\in \mathcal{K}_\beta'(t_R)$ because $b(\hat{t}_R,\hat{\beta})\in \mathcal{K}_{\hat\beta}'(\hat{t}_R)$ and $b(t_R,\beta)=\Phi_{\bR}(b(\hat{t}_R,\hat{\beta}))$.

We next verify that $\mathcal{K}_\beta'(t_R)$ is transitive. Let $\u\in\cL(\mathcal{K}_\beta'(t_R))$ and $\z\in\mathcal{K}_\beta'(t_R)$. By the definition of $\mathcal{K}_\beta'(t_R)$ we can extend $\u$ to the left and right to obtain a word $\u'\in\cL(\mathcal{K}_\beta'(t_R))$ containing $\u$ that consists of blocks from $\{\bR^-,\bR,\L(\bR),\L(\bR)^+\}$. Similarly we can extend $\z$ to the left to a sequence $\z'$ consisting of such blocks. By putting more blocks in front of $\u'$ if necessary, we may assume that $\u'$ begins with the block $\bR^-$. Also by adding more blocks to the right of $\u'$, we can obtain a word $\u''\in\cL(\mathcal{K}_\beta'(t_R))$ ending in $\L(\bR)^+$. Similarly, by putting more blocks in front of $\z'$ we obtain a sequence $\z''\in \mathcal{K}_\beta'(t_R)$ beginning with $\bR^-\L(\bR)^{l}$, where $l$ is chosen as large as possible so that $\z''$ is still in $\mathcal{K}_\beta'(t_R)$. (Since $\w^\f\succ \bR^-\L(\bR)^\f$, there is a largest such $l$.)

Now we will connect the word $\u''$ to the sequence $\z''$ {in $\mathcal{K}'_\beta(t_R)$}. Observe that $\u''=\Phi_\bR(\hat{\u})$ for some word $\hat{\u}$, and $\z''=\Phi_\bR(\hat{\z})$ for some sequence $\hat{\z}$. Since $\Phi_\bR$ is strictly increasing, $\hat{\u}\in \cL\big(\Kt_{\hat\beta}(\hat{t}_R)\big)$ and $\hat{\z}\in {\Kt_{\hat\beta}(\hat{t}_R)}$. Moreover, $\hat{\u}$ begins with $0$ and ends in $1$. Similarly, $\hat{\z}$ begins with $0^{l+1}$ and extending $\hat{\z}$ by another 0 would create a sequence that is not in $\Kt_{\hat\beta}(\hat{t}_R)$.

Since $\Kt_{\hat\beta}(\hat{t}_R)$ is transitive, there is a (possibly empty) word $\hat{\v}$ such that $\hat{\u}\hat{\v}\hat{\z}\in \Kt_{\hat\beta}(\hat{t}_R)$. By the above remark, $\hat{\v}$ must end with $1$ {if it is not empty}. We can write $\hat{\v}=1^k \hat{\v}'$, where $k\geq 0$ and $\hat{\v}'$ is either empty or begins with $0$. Now the sequence $(\hat{\u}1^k,\hat{\v}',\hat{\z})$ is connectible, and hence (see Lemma \ref{lem:connectible})
\[
\Phi_\bR\big((\hat{\u}1^k)\hat{\v}'\hat{\z}\big)=\Phi_\bR(\hat{\u}1^k)\Phi_\bR(\hat{\v}')\Phi_\bR(\hat{\z}).
\]
Set $\v:=\bR^k\Phi_\bR(\hat{\v}')$. Then, {since $\u''$ ends in $\L(\bR)^+$,} $\u''\v\z''=\Phi_\bR(\hat{\u}1^k\hat{\v}'\hat{\z})=\Phi_\bR(\hat{\u}\hat{\v}\hat{\z})\in \mathcal{K}_\beta'(t_R)$. This shows that $\mathcal{K}_\beta'(t_R)$ is transitive.

Next, we verify that $\mathcal{K}_\beta'(t_R)$ has full entropy and full Hausdorff dimension in $\Kt_\beta(t_R)$. If $\z=z_1z_2z_3\dots\in\Kt_\beta(t_R)$, then
\[
\bR^-0^\f\prec \Phi_\bR(\hat{\w}^\f)=\w^\f\lle\sigma^n(\z)\lle \alpha(\beta)=\Phi_\bR(\al(\hat{\beta}))\prec \L(\bR)^+1^\f \qquad\forall n\geq 0.
\]
So either $\z\in\Ga(\bR)$, where $\Ga(\bR)$ is the set defined in Lemma \ref{lem:countable}, or else $\z$ contains the word $\bR^-$ or $\L(\bR)^+$ somewhere, in which case there is by Lemma \ref{lem:four-blocks} an $n\in\N$ such that $\sigma^n(\z)=\Phi_\bR(\hat{\z})$ for some sequence $\hat{\z}\in\Kt_{\hat\beta}(\hat{t}_R)$, and the prefix $z_1\dots z_n$ is a word in $\cL(\Ga(\bR))$. 
This implies
\[
h\big(\Kt_\beta(t_R)\big)=\max\left\{h(\Ga(\bR)),h\big(\mathcal{K}_{\beta}'(t_R)\big)\right\}.
\]
But $\Ga(\bR)$ is countable by Lemma \ref{lem:countable}. Hence, we obtain
\begin{equation*}
h\big(\Kt_\beta(t_R)\big)=h\big(\mathcal{K}_{\beta}'(t_R)\big).
\end{equation*}
As a result, $\mathcal{K}_\beta'(t_R)$ has full entropy. By the same reasoning, we also have that
\[
\dim_H \mathcal{K}_\beta'(t_R)=\dim_H \Kt_\beta(t_R).
\]

Finally, it is easy to see from \eqref{eq:the-transitive-subshift} that $\{\mathcal{K}_\beta'(t_R): t_R\in\mathcal{T}_R(\beta)\}$ is a strictly descending collection of subshifts because $\{\Kt_{\hat\beta}(\hat{t}_R): \hat{t}_R\in\mathcal{T}_R(\hat{\beta})\}$ is one.
\end{proof}

\begin{proposition} \label{prop:transitive-subshifts-special-case}
The statement of Theorem \ref{thm:beta-in-relative-exceptional-set} (i) holds for $\beta=\beta_*^\cs$, for all $\cs\in\La$.
\end{proposition}

\begin{proof}
The proof is the same as the previous one, with a minor modification. Here we may assume that $\cs=\r_1\bullet\dots\bullet\r_k$ with $k\geq 2$, {$\r_1\in F_e$ and $\r_2,\dots,\r_k\in\F$, since for $\cs\in F_e$} we already know $\Kt_\beta(t_R)$ is transitive for all $t_R\in\mathcal{T}_R(\beta)$ by Theorem \ref{thm:basic-interval-right-endpoint}. Set $\cs_i:=\r_1\bullet\dots\bullet\r_i$ for $i=1,\dots,k$, and write $\al(\beta)=\Phi_{\cs_i}(\al(\hat{\beta}_i))$ for $i=1,\dots,k-1$, and certain bases $\hat{\beta}_1,\dots,\hat{\beta}_{k-1}$. Observe, however, that $\al(\beta)=\L(\cs_k)^+\cs_k^-\L(\cs_k)^\f=\Phi_{\cs_k}(10^\f)$, so there is no corresponding base $\hat{\beta}_k$.

Let $\w$ be a $\beta$-Lyndon word and $t_R$ the corresponding point such that $b(t_R,\beta)=\w^\f$. If $\w^\f\prec \r_1\L(\r_1)^\f=\cs_1\L(\cs_1)^\f$, then $\Kt_\beta(t_R)$ is transitive by Proposition \ref{prop:general-transitivity}. Otherwise, there is a unique integer $i\leq k-1$ such that
\begin{equation} \label{eq:S-sandwich-1}
\cs_i^-\L(\cs_i)^\f\prec \w^\f\prec \cs_{i+1}^-\L(\cs_{i+1})^\f,
\end{equation}
because $\tau(\beta)=\pi_\beta\big(\cs^-\L(\cs)^\f\big)$. Let $\hat{t}_R$ be the point such that $b(\hat{t}_R,\hat{\beta}_i)=\hat{\w}^\f$. 
\begin{itemize}
\item If $i<k-1$, then $\hat{\w}^\f\prec \r_{i+1}^-\L(\r_{i+1})^\f$ by \eqref{eq:S-sandwich-1}. Observe that $\al(\hat{\beta}_i)=\Phi_{\r_{i+1}}(\al(\hat{\beta}_{i+1}))$. So by Proposition \ref{prop:general-transitivity}, $\Kt_{\hat\beta_i}(\hat{t}_R)$ is transitive.
\item If $i=k-1$, then $\hat{\w}^\f\prec \r_k^-\L(\r_k)^\f$. Also ${\hat{\beta}_{k-1}}=\beta_*^{\r_k}$. Hence, $\Kt_{\hat\beta_{k-1}}(\hat{t}_R)$ is transitive by Theorem \ref{thm:basic-interval-right-endpoint}.
\end{itemize}
In both cases, $\Kt_{\hat\beta_i}(\hat{t}_R)$ is transitive. The rest of the proof is now the same as before.
\end{proof}

\subsection{Density of $\beta$-Lyndon intervals}

In order to complete the proof of Theorem \ref{thm:beta-in-relative-exceptional-set}, we need the following extension of Proposition \ref{prop:dense-intervals}.

\begin{proposition} \label{prop:S-dense-intervals}
Let $\cs\in\La$, and $\beta\in \overline{E^\cs}$. Then the $\beta$-Lyndon intervals are dense in $[0,\tau(\beta)]$.
\end{proposition}

The proof uses the following lemma.

\begin{lemma} \label{lem:mapping-Lyndon-intervals}
Let $\cs\in\La$, and let bases $\beta$ and $\hat{\beta}$ be related by $\alpha(\beta)=\Phi_\cs(\alpha(\hat{\beta}))$. Define the map
\begin{equation} \label{eq:Theta_S}
\Theta_{\cs,\beta}: [0,1)\to [0,1); \qquad \hat{t}\mapsto \pi_\beta\circ\Phi_\cs\circ b(\hat{t},\hat{\beta}).
\end{equation}
If $[\hat{t}_L,\hat{t}_R]$ is a $\hat{\beta}$-Lyndon interval, then $[\Theta_{\cs,\beta}(\hat{t}_L-),\Theta_{\cs,\beta}(\hat{t}_R)]$ is a $\beta$-Lyndon interval. 
\end{lemma}

\begin{proof}
Let $\hat{\w}$ be the $\hat{\beta}$-Lyndon word such that $b(\hat{t}_L,\hat{\beta})=\hat{\w}0^\f$. Write $\alpha(\hat{\beta})=\alpha_1\alpha_2\alpha_3\dots$. Consider the sequence $(\hat{\z}_n)$ defined by
\[
\hat{\z}_n:=\hat{\w}^-\alpha_1\alpha_2\dots\alpha_n 0^\f, \qquad n\in\N,
\]
and set $t_n:=\pi_{\hat\beta}(\hat{\z}_n)$. It is clear that $t_n\nearrow \pi_{\hat\beta}(\hat{\w}^-\alpha(\hat{\beta}))=\pi_{\hat\beta}(\hat{\w}0^\f)=\hat{t}_L$, and that $\hat{\z}_n$ is the greedy $\hat{\beta}$-expansion of $t_n$. So, using that $\Theta_{\cs,\beta}$ is strictly increasing and the maps $\pi_\beta$ and $\Phi_\cs$ are continuous, we obtain
\begin{align*}
\Theta_{\cs,\beta}(\hat{t}_L-)&=\lim_{t\nearrow \hat{t}_L} \Theta_{\cs,\beta}(t)=\lim_{n\to\infty} \pi_\beta\circ\Phi_\cs(\hat{\z}_n)\\
&=\pi_\beta\circ\Phi_\cs(\hat{\w}^-\alpha(\hat{\beta}))=\pi_\beta\big(\Phi_\cs(\hat{\w}^-)\Phi_\cs(\alpha(\hat{\beta}))\big)\\
&=\pi_\beta\big(\Phi_\cs(\hat{\w})^-\alpha(\beta)\big)=\pi_\beta\big(\Phi_\cs(\hat{\w})0^\f\big).
\end{align*}
Moreover, $\Theta_{\cs,\beta}(\hat{t}_R)=\pi_\beta\big(\Phi_\cs(\hat{\w}^\f)\big)$. Thus, $[\Theta_{\cs,\beta}(\hat{t}_L-),\Theta_{\cs,\beta}(\hat{t}_R)]$ is the $\beta$-Lyndon interval $[\pi_\beta(\w 0^\f),\pi_\beta(\w^\f)]$, where $\w:=\Phi_\cs(\hat{\w})$. 
\end{proof}

\begin{proof}[Proof of Proposition \ref{prop:S-dense-intervals}]
We set $t_*:=\pi_\beta\big(\cs^-\L(\cs)^\f\big)$, and divide the proof in two parts.

\medskip
{\em Step 1.} The interval $[t_*,\tau(\beta)]$. If $\beta=\beta_*^\cs$, then $\tau(\beta)=t_*$ so this interval degenerates to a point and there is nothing to prove. So assume $\beta\in \overline{E^\cs}\backslash \{\beta_*^\cs\}$. Then $\alpha(\beta)=\Phi_\cs(\alpha(\hat{\beta}))$ for some $\hat{\beta}\in \overline{E}{\cap(1,2]}$. Using the map $\Theta_{\cs,\beta}$ from Lemma \ref{lem:mapping-Lyndon-intervals}, we observe that $t_*=\Theta_{\cs,\beta}(0)$ and $\tau(\beta)=\Theta_{\cs,\beta}(\tau(\hat{\beta}))$ (see { \cite[Theorem 3.4]{Allaart-Kong-2024}}). Proposition \ref{prop:dense-intervals} tells us that the $\hat{\beta}$-Lyndon intervals are dense in $[0,\tau(\hat{\beta})]$. Since the Lyndon intervals have the property that any two of them are either disjoint or one contains the other (see \cite[Proposition 4.1]{Kalle-Kong-Langeveld-Li-18}), it follows that every Lyndon interval in $[0,\tau(\hat{\beta})]$ is contained in a $\hat{\beta}$-Lyndon interval.

Although the function $\Theta_{\cs,\beta}$ is not continuous (it jumps where the greedy expansion $t\mapsto b(t,\hat{\beta})$ does), we claim that the intervals $[\Theta_{\cs,\beta}(\hat{t}_L-),\Theta_{\cs,\beta}(\hat{t}_R)]$, as $[\hat{t}_L,\hat{t}_R]$ ranges over the $\hat{\beta}$-Lyndon intervals, cover all the jumps of $\Theta_{\cs,\beta}$. Note that $\Theta_{\cs,\beta}$ jumps only at points of the form $\pi_{\hat\beta}(\w 0^\f)$, where $\w$ is a word ending in 1 such that $\w 0^\f$ is a greedy $\hat{\beta}$-expansion. Let $\r$ be the longest prefix of $\w$ that is Lyndon. Then the Lyndon interval $[\pi_{\hat\beta}(\r0^\f),\pi_{\hat\beta}(\r^\f)]$ contains $\pi_{\hat\beta}(\w0^\f)$. Since this Lyndon interval is contained in some $\hat{\beta}$-Lyndon interval $[\hat{t}_L,\hat{t}_R]$, and $[\Theta_{\cs,\beta}(\hat{t}_L-),\Theta_{\cs,\beta}(\hat{t}_R)]$ contains the ``jump" $[\Theta_{\cs,\beta}(\hat{t}_L-),\Theta_{\cs,\beta}(\hat{t}_L)]$ by {the monotonicity of $\Theta_{\cs,\beta}$}, the claim follows. But then the intervals $[\Theta_{\cs,\beta}(\hat{t}_L-),\Theta_{\cs,\beta}(\hat{t}_R)]$, which are $\beta$-Lyndon by Lemma \ref{lem:mapping-Lyndon-intervals}, are dense in $[t_*,\tau(\beta)]$.

\medskip
{\em Step 2.} The interval $[0,t_*)$. 
Write $\alpha(\beta)=\alpha_1\alpha_2\alpha_3\dots$, and $\cs^-\L(\cs)^\f=\gamma_1\gamma_2\gamma_3\dots$. Let $t\in (0,t_*)$, so $(t_i):=b(t,\beta)\prec (\gamma_i)$. We will show that
\[
\si^n((t_1\dots t_k)^\f)\prec \al(\beta) \qquad \forall\,n\geq 0, \quad \forall\,k\in\N.
\]
The density of the $\beta$-Lyndon intervals then follows in the same way as in the proof of Proposition \ref{prop:dense-intervals}.

Fix $k\in\N$ and $0\leq n<k$. Since $(t_i)$ is a greedy $\beta$-expansion, $t_{n+1}\dots t_k\lle\alpha_1\dots\alpha_{k-n}$. We will now show that $(t_1\dots t_k)^\f\prec \si^{k-n}(\al(\beta))$.
Since $\beta\in \overline{E^\cs}$, it follows that $\alpha(\beta)\in X(\cs)$ and so $\sigma^l(\alpha(\beta))\lge \cs^-\L(\cs)^\f=\ga_1\ga_2\ldots$ for all $l\geq 0$. Thus, it suffices to show that
\begin{equation*} 
(t_1\dots t_k)^\f\prec \cs^-\L(\cs)^\f.
\end{equation*}
{This will follow from the assumption $(t_i)\prec \cs^-\L(\cs)^\f$ once we show that 
\begin{equation} \label{eq:smaller-than-shifts}
\si^j(\cs^-\L(\cs)^\f)\succ \cs^-\L(\cs)^\f \qquad\mbox{for all $j\geq 1$}. 
\end{equation}
Let $m:=|\cs|$, and write $\cs=s_1\dots s_m$. If $j\geq m$, then there is $0\leq i<m$ such that
\[
\si^j(\cs^-\L(\cs)^\f)=(s_{i+1}\dots s_m s_1\dots s_i)^\f\succ s_{i+1}\dots s_m^- M_\beta^\f\lge \cs^-\L(\cs)^\f
\]
by the Lyndon property of $\cs$. Suppose $1\leq j<m$, and note this implies $m\geq 2$. Since $\cs$ is Lyndon, we have 
\[
\gamma_{j+1}\dots\gamma_m=s_{j+1}\dots s_m^-\lge s_1\dots s_{m-j},
\]
and
\[
\gamma_{m+1}\gamma_{m+2}\dots=\L(\cs)^\f\lge (s_{m-j+1}\dots s_m s_1\dots s_{m-j})^\f\succ s_{m-j+1}\dots s_m^-\L(\cs)^\f.
\]
Thus, we again obtain \eqref{eq:smaller-than-shifts}.
}
\end{proof}



\begin{proof}[Proof of Theorem \ref{thm:beta-in-relative-exceptional-set}]
Statement (i) follows Propositions \ref{prop:transitive-subshifts} and \ref{prop:transitive-subshifts-special-case}, and (ii) follows from Proposition \ref{prop:S-dense-intervals}. 
\end{proof}

\section{Comparison of the bifurcation sets} \label{sec:comparison}

Our aim in this section is to prove the following extension of Corollary \ref{cor:E-B-left-endpoints}.

\begin{theorem} \label{thm:E-minus-B}
Let $\cs=\r_1\bullet \r_2\bullet\dots\bullet\r_k\in\Lambda$, where {$\r_1\in F_e$ and $\r_i\in F^*$ for $i=2,\dots,k$}. Suppose $\beta\in\{\beta_\ell^\cs,\beta_*^\cs,\beta_r^\cs\}$.
\begin{enumerate}[{\rm(i)}]
\item If $\beta=\beta_\ell^\cs$, then $\EE_\beta\backslash\BB_\beta$ consists exactly of the $k-1$ points whose greedy $\beta$-expansions, in decreasing order, are
\begin{equation} \label{eq:periodic-isolated-points}
\r_1^\f, \qquad (\r_1\bullet \r_2)^\f, \qquad \dots, \qquad (\r_1\bullet \r_2\bullet\dots\bullet \r_{k-1})^\f.
\end{equation}
\item If $\beta=\beta_r^\cs$, then $\EE_\beta\backslash\BB_\beta$ consists exactly of the $k$ points whose greedy $\beta$-expansions, in decreasing order, are
\[
\r_1^\f, \qquad (\r_1\bullet \r_2)^\f, \qquad \dots, \qquad (\r_1\bullet \r_2\bullet\dots\bullet \r_k)^\f.
\]
{
\item If $\beta=\beta_*^\cs$, then $\EE_\beta\backslash\BB_\beta$ consists of exactly $k+1$ points, namely $\tau(\beta)=\pi_\beta(\cs^-\L(\cs)^\f)$ and the $k$ points whose greedy $\beta$-expansions are as in {\rm (ii)} above.
}
\end{enumerate}
\end{theorem}

The proof uses the following lemma. Recall the definition of $\Ga(\r)$ from \eqref{eq:Gamma}. {We also define the following subset of $\La$:
\begin{equation} \label{eq:Lambda-star}
\La^*:=\{\cs=\s_1\bullet\s_2\bullet\dots\bullet \s_k\in\La: \s_1\in F^*\}.
\end{equation}
}

\begin{lemma} \label{lem:equivalent-upper-bound}
Let { $\r\in\La$ and $\s\in\La^*$}, and suppose $\z\in{\N_0}^\N$ satisfies
\begin{equation} \label{eq:asymmetric-sandwich}
\r^\f\lle \sigma^n(\z)\lle {\L(\r\bullet\s)^\f}\qquad\forall n\geq 0.
\end{equation}
Then $\z\in\Ga(\r)$.
\end{lemma}

\begin{proof}
{Since $\s$ begins with $0$ and ends with $1$, $\L(\s)$ begins with $1$ and ends with $0$. So $\L(\r\bullet\s)=\Phi_\r(\L(\s))$} begins with $\L(\r)^+\r^k\r^-$ for some finite $k$. Hence, $\sigma^n(\z)\prec \L(\r)^+{\r}^\f$ for all $n\geq 0$, and if $\z$ ever contains $\L(\r)^+$, then the next block of length $(k+1)|\r|$ must be $\lle \r^k\r^-$ and also $\lge \r^{k+1}$ by the lower bound in \eqref{eq:asymmetric-sandwich}. This is impossible, so $\z\in\Ga(\r)$.
\end{proof}

\begin{proof}[Proof of Theorem \ref{thm:E-minus-B}]
We begin with (i). For $k=1$ this is simply Corollary \ref{cor:E-B-left-endpoints}, so assume $k\geq 2$. Let $\cs=\r_1\bullet \r_2\bullet\cdots\bullet \r_k\in\Lambda$, where {$\r_1\in F_e$ and $\r_i\in F^*$ for $i=2,\dots,k$}. Let $\beta=\beta_\ell^\cs$, {so $\al(\beta)=\L(\cs)^\f$}. From Theorem \ref{thm:beta-in-relative-exceptional-set} and Proposition \ref{prop:E-B-equality} it follows that $\EE_\beta\cap[0,\tau(\beta))=\BB_\beta\cap[0,\tau(\beta))$, so it suffices to show that $\EE_\beta$ has precisely $k-1$ points in the interval $[\tau(\beta),1)$, whose greedy $\beta$-expansions are given by \eqref{eq:periodic-isolated-points}.

Observe from {\cite[Theorem 3.4]{Allaart-Kong-2024}} that
$\tau(\beta)=\pi_\beta\big(\cs^-\L(\cs)^\f\big)$.
Let
\[
t_j:=\pi_\beta\big((\r_1\bullet\dots\bullet\r_j)^\f\big), \qquad j=1,2,\dots,k-1.
\]
Note that $(\r_1\bullet\dots\bullet\r_j)^\f$ is indeed the greedy $\beta$-expansion of $t_j$, since
${\L(\r_1\bullet\dots\bullet\r_j)^\f}\prec {\L(\r_1\bullet\dots\bullet\r_k)^\f=}\al(\beta)$. Noting also that $\r_1\bullet\dots\bullet\r_j$ is Lyndon, it follows that $T^n(t_j)\geq t_j$ for every $n\geq 0$, so that $t_j\in\EE_\beta^+=\EE_\beta$ by Lemma \ref{lem:E-beta-characterizations}.
Furthermore, $(\r_1\bullet\cdots\bullet\r_j)^\f\succ \r_1\bullet\cdots\bullet\r_k^-{M_\beta}^\f\succ\cs^-\L(\cs)^\f$. Thus, $t_j\in(\tau(\beta), 1)$.

It remains to show that there are no further points of $\EE_\beta$ in $(\tau(\beta),1)$. 

Recall that $k\geq 2$. We first claim that $K_\beta(t)=\emptyset$ for all $t>t_1$, so there are no points of $\EE_\beta$ to the right of $t_1$. This follows since
\begin{equation*}
\mathcal{K}_\beta(t_1)=\{\z: \r_1^\f\lle \sigma^n(\z)\prec \L(\cs)^\f\ \ \forall n\geq 0\}=\Ga(\r_1)
\end{equation*}
by Lemma \ref{lem:equivalent-upper-bound}, so $\mathcal{K}_\beta(t_1)$ consists only of $\r_1^\f$ and its shifts by Lemma \ref{lem:Farey-sandwich}, and none of these lie in $\mathcal{K}_\beta(t)$ for $t>t_1$.

Next, we show that there are no points of $\EE_\beta$ in $(\tau(\beta),t_{k-1})$; in other words, $K_\beta(t)$ is constant on $(\tau(\beta),t_{k-1}]$. Note that
\[
\tau(\beta)=\pi_\beta\big(\cs^-\L(\cs)^\f\big)=\pi_\beta\big(\cs 0^\f\big)=\pi_\beta\big((\r_1\bullet\dots\bullet \r_k)0^\f\big)
\]
and $t_{k-1}=\pi_\beta\big((\r_1\bullet\dots \bullet \r_{k-1})^\f\big)$. Let $t\in(\tau(\beta),t_{k-1})$ and suppose $\z\in\mathcal{K}_\beta(t)$. We claim that $\z\in\mathcal{K}_\beta(t_{k-1})$. Observe that {$b(t,\beta)\succ \cs 0^\f$. Then $z\in\mathcal K_\beta(t)$ implies that }
\begin{equation} \label{eq:S-sandwich}
\cs0^\f\lle \sigma^n(\z)\prec \L(\cs)^\f \quad\forall n\geq 0.
\end{equation}
Since $\cs$ begins with $\r_1\bullet \dots\bullet\r_{k-1}^-$, we have $\sigma^n(\z)\lge (\r_1\bullet \dots\bullet\r_{k-1}^-)0^\f$ for all $n\geq 0$. Suppose $\z$ contains the word $\r_1\bullet \dots\bullet\r_{k-1}^-$; without loss of generality we may assume $\z$ begins with it. Then by Lemma \ref{lem:four-blocks},
\[
\z=\Phi_{\r_1\bullet\dots\bullet \r_{k-1}}(\hat{\z})
\]
for some sequence $\hat{\z}$ satisfying
\[
\r_k0^\f\lle \sigma^n(\hat{\z})\prec \L(\r_k)^\f \quad\forall n\geq 0,
\]
in view of \eqref{eq:S-sandwich}. However, by Lemma \ref{lem:Farey-sandwich} such a sequence $\hat{\z}$ does not exist. Therefore, $\sigma^n(\z)\lge (\r_1\bullet \dots\bullet\r_{k-1})^\f$ for all $n\geq 0$, which implies $\z\in \mathcal{K}_\beta(t_{k-1})$.

Note that we are now done with the case $k=2$, so fix $k\geq 3$ from now on.
Fix $1\leq j\leq k-2$, and let $t\in(t_{j+1},t_j)$. Suppose by way of contradiction that there exists a sequence $\z\in \mathcal{K}_\beta(t)\backslash \mathcal{K}_\beta(t_j)$. Then
\begin{equation} \label{eq:z-tail-sandwich}
\cs^\f\lle (\r_1\bullet\dots\bullet\r_{j+1})^\f=b(t_{j+1},\beta)\prec b(t,\beta)\lle \sigma^n(\z)\prec\alpha(\beta)=\L(\cs)^\f \qquad \forall n\geq 0,
\end{equation}
and there is an integer $n_0\geq 0$ such that $\sigma^{n_0}(\z)\prec b(t_j,\beta)=(\r_1\bullet\dots\bullet \r_j)^\f$.
Since $\r_1\bullet\dots\bullet\r_{j+1}$ begins with $\r_1\bullet\dots\bullet\r_j^-$, there is  therefore a further integer $m_0\geq n_0$ such that $\sigma^{m_0}(\z)$ begins with $\r_1\bullet\dots\bullet\r_j^-$. Set $\r:=\r_1\bullet\dots\bullet\r_j$ and $\s:=\r_{j+1}\bullet\dots\bullet\r_k$. Applying Lemma \ref{lem:four-blocks} yields that
\[
\z=\Phi_{\r_1\bullet\dots\bullet \r_j}(\hat{\z})
\]
for some sequence $\hat{\z}$, which by \eqref{eq:z-tail-sandwich} satisfies the inequalities
\begin{equation} \label{eq:wider-z-hat-sandwich}
\r_{j+1}^\f\prec\sigma^n(\hat{\z})\prec \L(\r_{j+1}\bullet\dots\bullet\r_k)^\f\qquad \forall n\geq 0.
\end{equation}
But then
\begin{equation} \label{eq:z-hat-sandwich}
\r_{j+1}^\f\prec\sigma^n(\hat{\z})\lle \L(\r_{j+1})^\f\qquad \forall n\geq 0
\end{equation}
by the obvious modification of Lemma \ref{lem:equivalent-upper-bound}. However, no sequence $\hat{\z}$ satisfying \eqref{eq:z-hat-sandwich} exists by Lemma \ref{lem:Farey-sandwich}. Hence, $K_\beta(t)=K_\beta(t_j)$.

{
This completes the proof of (i). For (ii) and (iii) we need to modify the argument slightly. Take first $\beta=\beta_r^\cs$. Then $\tau(\beta)=\pi_\beta\big(\cs 0^\f\big)$, and $\al(\beta)=\L(\cs)^+\cs^\f$. Now $\cs$ itself is $\beta$-Lyndon, and so $t_k:=\pi_\beta(\cs^\f)=\pi_\beta\big((\r_1\bullet\dots\bullet\r_k)^\f\big)$ defines a point $t_k>\tau(\beta)$. Note that $[\tau(\beta),t_k]$ is the $\beta$-Lyndon interval generated by $\cs$, so there are no points of $\EE_\beta$ in $[\tau(\beta),t_k)$.

If $t\in(t_{j+1},t_j)$ for $j=1,\dots,k-1$ and $\z\in \mathcal{K}_\beta(t)\backslash \mathcal{K}_\beta(t_j)$, then instead of  \eqref{eq:z-tail-sandwich}, we initially have
\[
\cs^\f\lle (\r_1\bullet\dots\bullet\r_{j+1})^\f=b(t_{j+1},\beta)\prec b(t,\beta)\lle \sigma^n(\z)\prec\alpha(\beta)=\L(\cs)^+\cs^\f \qquad \forall n\geq 0.
\]
But this reduces to \eqref{eq:z-tail-sandwich}, since $\z$ cannot contain the block $\L(\cs)^+$: If it did, what follows would have to be both $\prec\cs^\f$ and $\succ\cs^\f$, which is absurd. The rest proceeds as before.

For (iii), take $\beta=\beta_*^\cs$. Then $\tau(\beta)=\pi_\beta\big(\cs^- \L(\cs)^\f\big)$ and $\al(\beta)=\L(\cs)^+\cs^-\L(\cs)^\f$. Recall from \eqref{eq:smaller-than-shifts} that the sequence $(c_i):=\cs^- \L(\cs)^\f$ satisfies $\si^n((c_i))\lge (c_i)$ for all $n\geq 0$; hence $\tau(\beta)\in \EE_\beta\backslash\BB_\beta$. As in (ii), there is also an extra bifurcation point $t_k:=\pi_\beta(\cs^\f)=\pi_\beta\big((\r_1\bullet\dots\bullet\r_k)^\f\big)\in \EE_\beta\backslash\BB_\beta$, because here, too, $\cs$ is $\beta$-Lyndon. Instead of \eqref{eq:S-sandwich}, we now have, initially, for $t\in(\tau(\beta),t_k)$ and $\z\in\mathcal{K}_\beta(t)$,
\[
\cs^-\L(\cs)^\f\prec b(t,\beta)\lle \sigma^n(\z)\prec \L(\cs)^+\cs^-\L(\cs)^\f \quad\forall n\geq 0.
\]
But then $\z$ clearly cannot contain the block $\L(\cs)^+$, so we must have $\si^n(\z)\lle \L(\cs)^\f$ for every $n$; and then $\z$ cannot contain the block $\cs^-$ either, as is easy to see. Hence, $\si^n(\z)\lge\cs^\f$ for each $n$, which implies $\z\in\mathcal{K}_\beta(t_k)$.

In the same way, \eqref{eq:z-tail-sandwich} remains valid for $t\in(t_{j+1},t_j)$ with $j=1,\dots,k-1$, and $\z\in \mathcal{K}_\beta(t)\backslash \mathcal{K}_\beta(t_j)$. The rest of the argument goes as before.
}
\end{proof}

\section{The case of infinitely renormalizable $\beta$} \label{sec:E-infinity}

Recall the set $E_\f$ from \eqref{eq:E_f}. In this section we prove the following.


\begin{theorem} \label{thm:infinitely-Farey}
Let $\beta\in E_\f$. Then
\begin{enumerate}[{\rm(i)}]
\item For every $t_R\in\mathcal{T}_R$, $\Kt_\beta(t_R)$ has a transitive subshift $\mathcal{K}_\beta'(t_R)$ of full entropy and full Hausdorff dimension that contains the sequence $b(t_R,\beta)$. Moreover, these transitive subshifts can be chosen so that $\{\mathcal{K}_\beta'(t_R): t_R\in\mathcal{T}_R(\beta)\}$ is a strictly descending collection of subshifts.
\item The $\beta$-Lyndon intervals are dense in $[0,\tau(\beta)]$.
\end{enumerate}
\end{theorem}

Part (i) gives Theorem \ref{thm:general-transitivity} for $\beta\in E_\f$.

\begin{proof}
(i) Since $\beta\in E_\f$, there {exist an extended Farey word $\s_1$ and an infinite sequence $\{\s_2,\s_3,\dots\}$} of Farey words, uniquely determined by $\beta$, such that $\beta\in J^{\s_1\bullet\dots\bullet\s_n}$ for each $n\in\N$. Write $\cs_i:=\s_1\bullet\dots\bullet\s_i$ for $i\in\N$, and note that for each $i$ there is a base $\hat{\beta}_i$ such that $\al(\beta)=\Phi_{\cs_i}\big(\al(\hat{\beta}_i)\big)$. Recall from \cite[Proposition 6.3]{Allaart-Kong-2021} that
\begin{equation} \label{eq:tau-beta-E-infinity}
\tau(\beta)=\lim_{n\to\f}\pi_\beta(\s_1\bullet\cdots\bullet\s_n 0^\f).
\end{equation}
Hence, given a point $t_R\in\mathcal{T}_R$, we either have $b(t_R,\beta)\prec \s_1^-\L(\s_1)^\f$, in which case $\Kt_\beta(t_R)$ is transitive by Proposition \ref{prop:general-transitivity}, or else there is a unique $i\in\N$ such that
\[
\cs_i^-\L(\cs_i)^\f\prec b(t_R,\beta) \prec \cs_{i+1}^-\L(\cs_{i+1})^\f.
\]
We conclude as in the proof of Proposition \ref{prop:transitive-subshifts} that $\Kt_{\hat\beta_i}(\hat{t}_R)$ is transitive. The rest of the proof is the same as the proof of Proposition \ref{prop:transitive-subshifts}.

(ii) {Set $\cs_n:=\s_1\bullet\dots\bullet\s_n$ and $\tau_n:=\pi_\beta\big(\cs_n^-\L(\cs_n)^\f\big)$, for $n\in\N$.
For each $n$, the $\beta$-Lyndon intervals are dense in $[0,\tau_n]$ by applying the argument in Step 2 of the proof of Proposition \ref{prop:S-dense-intervals} to $\cs_n$.} Since $\tau_n\nearrow \tau(\beta)$ {by \eqref{eq:tau-beta-E-infinity}}, this implies the $\beta$-Lyndon intervals are dense in $[0,\tau(\beta)]$.
\end{proof}

A simple modification of the proof of Theorem \ref{thm:E-minus-B} yields the following:

\begin{proposition} \label{prop:E-inf-bifurcation}
Let $\beta\in E_\f$ with corresponding sequence $\s_1,\s_2,\dots$ of Farey words. Then $\tau(\beta)$ is given by \eqref{eq:tau-beta-E-infinity},
and $\EE_\beta$ has countably infinitely many points in the interval $(\tau(\beta),1)$, namely the points whose greedy $\beta$-expansions, in decreasing order, are
\[
\s_1^\f, \qquad (\s_1\bullet \s_2)^\f, \qquad (\s_1\bullet \s_2\bullet \s_3)^\f,\qquad \dots.
\]
\end{proposition}

Unfortunately, the $\beta$-shift $\Sigma_\beta$ is never sofic for $\beta\in E_\f$, as it is easy to check that $\al(\beta)$ cannot be eventually periodic. As a result, the subshifts $\Kt_\beta(t_R)$ are not all sofic, and therefore our argument from the previous section that $\EE_\beta\cap[0,\tau(\beta)]=\BB_\beta\cap[0,\tau(\beta)]$ is not valid. We suspect, but have been unable to prove, that this equation is nonetheless true, and hence $\EE_\beta\backslash\BB_\beta$ is countably infinite for $\beta\in E_\f$.

\section{Interiors of basic intervals: Construction of non-transitivity windows} \label{sec:basic-interiors}

We now begin to deal with the most complex case, which is when $\beta$ lies in the interior of a basic interval. Before developing the general theory, we give a concrete example which illustrates the main ideas.

\begin{example} \label{ex:non-transitivity-window}
Let $\beta$ be as in Example \ref{ex:non-dense-intervals} such that $\alpha(\beta)$ begins with $1110101100$, and set $\w=01010111$. For $t=t_R$ given by $b(t_R,\beta)=\w^\f$, the subshift $\Kt_\beta(t)$ is not transitive: Note that, in fact,
\[
\Kt_\beta(t_R)=\{\z\in\{0,1\}^\N: (01010111)^\f\lle \si^n(\z)\lle (11101010)^\f\ \forall\,n\geq 0\},
\]
so the only sequence in $\Kt_\beta(t)$ beginning with $111$ is $(11101010)^\f$, and the legal word $111$ cannot be connected to, for instance, the legal word $011011$. On the other hand, disallowing the word $111$ removes only countably many sequences from $\Kt_\beta(t_R)$, so the subshift
\[
\mathcal{K}_\beta'(t):=\{\z\in\Kt_\beta(t): \z\ \mbox{does not contain $111$}\}
\]
has full entropy in $\Kt_\beta(t)$. Note that $\mathcal{K}_\beta'(t)$ consists precisely of those sequences of the form
\[
(110)^{k_1}10(110)^{k_2}10(110)^{k_3}10\dots, \qquad k_i\geq 1\ \forall i,
\]
together with their orbits under $\si$, and is clearly transitive. Now let $\w^*=01011$ and let $t^*$ be given by $b(t^*,\beta)=(\w^*)^\f$. It is easy to see that $\Kt_\beta(t^*)=\mathcal{K}_\beta'(t)$. Hence,
\[
h\big(\Kt_\beta(t)\big)=h\big(\mathcal{K}_\beta'(t)\big)=h\big(\Kt_\beta(t^*)\big).
\]
Observe that $\w^*$ is the lexicographically smallest $\beta$-Lyndon word greater than $\w$. Note also that $\Kt_\beta(t_R)$ does {\em not} have a transitive subshift of full entropy containing the sequence $\w^\f$.

Developing the example a bit further, we set $\w_k:=0101(011)^k0111$ for $k=0,1,2,\dots$, and let $t_k$ be the point given by $b(t_k,\beta)=\w_k^\f$. As above it can be shown that $\Kt_\beta(t_k)$ is not transitive, though here there are more sequences in $\Kt_\beta(t_k)$ beginning with $111$. With a bit more effort, it can be shown that $h\big(\Kt_\beta(t_k)\big)=h\big(\Kt_\beta(t^*)\big)$. Observe that $t_k\searrow t_\f$, where $t_\f$ is given by $b(t_\f,\beta)=0101(011)^\f$. By the continuity of $t\mapsto h(\Kt_\beta(t))$, it follows that $h\big(\Kt_\beta(t_\f)\big)=h\big(\Kt_\beta(t^*)\big)$.

We call the interval $[t_\f,t^*)$ a {\em non-transitivity window}. The assertions about entropy in the last paragraph will follow from Proposition \ref{prop:entropy-constant-in-window} below.
\end{example}

We saw in Section \ref{sec:beta-in-E} that if $\beta\in E_L$, then $\Kt_\beta(t_R)$ is transitive for any $\beta$-Lyndon interval $[t_L,t_R]\subseteq [0,\tau(\beta)]$. On the other hand, the above example shows that this need no longer be the case when $\beta$ lies in the interior of a basic interval. We will now examine the question of transitivity for the interior of basic intervals in greater detail. We will show that, depending on $\beta$, there is a collection of intervals which we call {\em non-transitivity windows}, such that $\Kt_\beta(t_R)$ is transitive if and only if $t_R$ does not lie in any of these intervals. The number of non-transitivity windows can be zero, positive and finite, or infinite, depending on $\beta$.

We first prove a simple but useful lemma.

\begin{lemma} \label{lem:no-beta-Lyndon-extension}
If a word $\v$ is Lyndon but not $\beta$-Lyndon, then no word extending $\v$ is $\beta$-Lyndon.
\end{lemma}

\begin{proof}
Suppose $\v=v_1\dots v_k$ is Lyndon but not $\beta$-Lyndon, so there is an integer $j<k$ such that
\begin{equation} \label{eq:big-cyclic-permutation}
{v_{j+1}\dots v_k (v_1\dots v_k)^\f\lge \al(\beta).}
\end{equation}
Let $\w=v_1\dots v_k v_{k+1}\dots v_l$ be an extension of $\v$. Assume $\w$ is Lyndon, as otherwise there is nothing to prove. Write $l=qk+r$, where $q\in\N$ and $0\leq r<k$. Since $\w$ is Lyndon, $v_{ik+1}\dots v_{(i+1)k}\lge v_1\dots v_k$ for all $1\leq i<q$, with strict inequality for $i=q-1$ in case $r=0$; and $v_{qk+1}\dots v_l=v_{qk+1}\dots v_{qk+r}\succ v_1\dots v_r$ if $r>0$. Thus, with $j$ as in \eqref{eq:big-cyclic-permutation},
\begin{equation*} 
v_{j+1}\dots v_l\succ v_{j+1}\dots v_k(v_1\dots v_k)^{q-1}v_1\dots v_r,
\end{equation*}
{and this implies
\[
v_{j+1}\dots v_l (v_1\dots v_l)^\f\succ v_{j+1}\dots v_l 0^\f\succ v_{j+1}\dots v_k (v_1\dots v_k)^\f\lge \al(\beta),
\]
using \eqref{eq:big-cyclic-permutation}. Hence, $\w$ is not $\beta$-Lyndon.}
\end{proof}

In the sequel, it is convenient to extend the lexicographical order to compare words of different lengths: If $\v=v_1\dots v_m$ and $\w=w_1\dots w_n$ with $m<n$, we write $\v\prec\w$ (and say $\v$ is {\em smaller} than $\w$) if $v_1\dots v_m\lle w_1\dots w_m$. Otherwise, we write $\v\succ\w$ (and say $\v$ is {\em larger} than $\w$). Note in particular, that $\v\prec\w$ if $\v$ is a proper prefix of $\w$.

\begin{lemma} \label{lem:repeating-preserves-order}
If $\v$ and $\w$ are Lyndon words and $\v\prec\w$, then $\v^\f\prec \w^\f$.
\end{lemma}

\begin{proof}
This is non-trivial only in the case when one of the words extends the other, say $\v=v_1\dots v_m$ and $\w=w_1\dots w_n=v_1\dots v_m w_{m+1}\dots w_n$, where $n>m$. Write $n=qm+r$, where $q$ and $r$ are integers with $q\geq 1$ and $1\leq r\leq m$. Then, using the Lyndon property of $\w$,
\[
w_{pm+1}\dots w_{(p+1)m}\lge w_1\dots w_m=v_1\dots v_m \qquad \mbox{for $1\leq p<q$},
\]
and $w_{qm+1}\dots w_n\succ w_1\dots w_r=v_1\dots v_r$, from which the assertion follows.
\end{proof}

Now we fix a basic interval $[\beta_\ell,\beta_*]=[\beta_\ell^{\cs},\beta_*^{\cs}]$ generated by a word $\cs\in\La$, and fix $\beta\in(\beta_\ell,\beta_*)$. We shall repeatedly use the following fact.

\begin{lemma} \label{lem:existence-of-v-star}
Let $\v$ be a Lyndon word such that $\v\lle \cs$. Then there exists a (lexicographically) smallest $\beta$-Lyndon word $\v^*$ satisfying $\v^*\lge \v$. Furthermore, $\v^*\lle \cs$, and if $\v$ is not $\beta$-Lyndon, then $|\v^*|<|\v|$.
\end{lemma}

\begin{proof}
If $\v$ is already $\beta$-Lyndon, then we simply set $\v^*=\v$ and there is nothing to prove. So assume $\v$ is not $\beta$-Lyndon.

Note first that, since $\al(\beta)$ begins with $\L(\cs)^+$, $\cs$ is $\beta$-Lyndon. Hence there is at least one $\beta$-Lyndon word greater than or equal to $\v$. 

We next show that, if $\w$ is a $\beta$-Lyndon word with $\w\succ\v$ and $|\w|\geq |\v|$, then some proper prefix of $\w$ of length less than $|\v|$ is already $\beta$-Lyndon. Write $\v=v_1\dots v_m$ and $\w=w_1\dots w_n$. By Lemma \ref{lem:no-beta-Lyndon-extension}, $\w$ cannot be an extension of $\v$, and hence, $w_1\dots w_m\succ v_1\dots v_m$. Let $l$ be the smallest index such that $w_l>v_l$; then {$\w':=w_1\dots w_l$ is Lyndon because $\v$ is Lyndon. Furthermore, $\w'$ is $\beta$-Lyndon by Lemma \ref{lem:no-beta-Lyndon-extension}. This also implies $l<m$: If $l=m$ then, since $\v$ is not $\beta$-lyndon, $\w'$ would certainly not be $\beta$-Lyndon.}

Since there are only finitely many $\beta$-Lyndon words $\w'$ satisfying $\w'\succ\v$ and $|\w'|<|\v|$, a smallest exists among them. This is our $\v^*$. Since $\v\lle\cs$, $\cs$ is $\beta$-Lyndon and $\v^*$ is the smallest $\beta$-Lyndon word greater than or equal to $\v$, it follows that $\v^*\lle \cs$.
\end{proof}

We will also use the following notation: For a word $\u=u_1\dots u_n$ or sequence $\u=u_1u_2\dots$, we write $\u_{i:j}:=u_i\dots u_j$.

Recall from \cite[Theorem 2]{Allaart-Kong-2021} {and \cite[Theorem 3.4]{Allaart-Kong-2024}} that
\[
\tau(\beta)=\pi_\beta\big(\cs^-\L(\cs)^\f\big).
\]
Assume first that $\al(\beta)=\al_1\al_2\dots$ is not periodic. We will inductively build an ordered collection $\II$ of intervals in $[0,\tau(\beta)]$ as follows.

To begin, set $j_1:=|\cs|$. Then {$\al_1\ldots \al_{j_1}=\L(\cs)^+$, and} $\si^{j_1}(\al(\beta))\prec \si^{j_1}(\al(\beta_*))=\cs^-\L(\cs)^\f$. If $\si^{j_1+n}(\al(\beta))\succ \si^{j_1}(\al(\beta))$ for all $n\geq 1$, then we set $\II:=\emptyset$. Otherwise, let
\[
l_1:=\min\{l\geq 1: \si^{j_1+l}(\al(\beta))\lle \si^{j_1}(\al(\beta))\}.
\]
Then we set $\v_1:=\al_{j_1+1}\dots\al_{j_1+l_1}$, and note that $\v_1$ is Lyndon (though not necessarily $\beta$-Lyndon). To see this, take $1\leq i<l_1$ and observe by the definition of $l_1$ that
\[
\si^{j_1+l_1}(\al(\beta)) \lle \si^{j_1}(\al(\beta)) \prec \si^{j_1+l_1-i}(\al(\beta)).
\]
Thus, by the minimality of $l_1$,
\begin{align*}
\al_{j_1+1}\dots \al_{j_1+l_1-i} \si^{j_1+l_1-i}(\al(\beta)) &=\si^{j_1}(\al(\beta)) \prec \si^{j_1+i}(\al(\beta))\\
&= \al_{j_1+i+1}\dots\al_{j_1+l_1}\si^{j_1+l_1}(\al(\beta)) \\
&\prec \al_{j_1+i+1}\dots\al_{j_1+l_1} \si^{j_1+l_1-i}(\al(\beta)),
\end{align*}
and this implies $\al_{j_1+1}\dots \al_{j_1+l_1-i}\prec\al_{j_1+i+1}\dots\al_{{j_1+l_1}}$.

We define an interval $I_1=[\underline{t}_1,\overline{t}_1)$, where $\underline{t}_1$ and $\overline{t}_1$ are given implicitly by
\[
b(\underline{t}_1,\beta)=\v_1^-(\al_1\dots\al_{j_1}^-)^\f=\v_1^-\L(\cs)^\f, \qquad b(\overline{t}_1,\beta)=(\v_1^*)^\f,
\]
where $\v_1^*$ is the (lexicographically) smallest $\beta$-Lyndon word greater than or equal to $\v_1$. Since
\[
\v_1=\al_{j_1+1}\dots\al_{j_1+l_1}\lle (\cs^-\L(\cs)^\f)_{1:l_1}\prec \cs,
\]
$\v_1^*$ exists by Lemma \ref{lem:existence-of-v-star}. Furthermore, $\v_1^*=\v_1$ if $\v_1$ is already $\beta$-Lyndon, and $|\v_1^*|<|\v_1|$ otherwise.
We put the interval $I_1$ in the collection $\II$.

Next, we set $j_2:=j_1+l_1$ and continue the procedure in the same way. That is, if $\si^{j_2+n}(\al(\beta))\succ \si^{j_2}(\al(\beta))$ for all $n\geq 1$, then we set $\II=(I_1)$ and stop. Otherwise, let
\[
l_2:=\min\{l\geq 1: \si^{j_2+l}(\al(\beta))\lle \si^{j_2}(\al(\beta))\},
\]
and set $\v_2:=\al_{j_2+1}\dots\al_{j_2+l_2}$. Observe that $\v_2$ is Lyndon by the same argument that we used above for $\v_1$. We define an interval $I_2=[\underline{t}_2,\overline{t}_2)$, where $\underline{t}_2$ and $\overline{t}_2$ are given implicitly by
\[
b(\underline{t}_2,\beta)=\v_2^-(\al_1\dots\al_{j_2}^-)^\f=\v_2^-({\L(\cs)^+}\v_1^-)^\f, \qquad b(\overline{t}_2,\beta)=(\v_2^*)^\f,
\]
where $\v_2^*$ is the (lexicographically) smallest $\beta$-Lyndon word greater than or equal to $\v_2$. Again, $\v_2^*$ exists by Lemma \ref{lem:existence-of-v-star}, since
\[
\v_2=\al_{j_1+l_1+1}\dots\al_{j_1+l_1+l_2}\lle \al_{j_1+1}\dots\al_{j_1+l_2}\lle (\cs^-\L(\cs)^\f)_{1:l_2}\prec \cs,
\]
where in the first inequality we used the definition of $l_1$.
We add $I_2$ to the collection $\II$.

This way we continue: Suppose integers $j_1,\dots,j_k$ and $l_1,\dots,l_k$, words $\v_1,\dots,\v_k$ and intervals $I_1=[\underline{t}_1,\overline{t}_1),\dots,I_k=[\underline{t}_k,\overline{t}_k)$ have been constructed, and set $j_{k+1}:=j_k+l_k$. If $\si^{j_{k+1}+n}(\al(\beta))\succ \si^{j_{k+1}}(\al(\beta))$ for all $n\geq 1$, then we set $\II=(I_1,\dots,I_k)$ and stop. Otherwise, let
\[
l_{k+1}:=\min\{l\geq 1: \si^{j_{k+1}+l}(\al(\beta))\lle \si^{j_{k+1}}(\al(\beta))\},
\]
and set $\v_{k+1}:=\al_{j_{k+1}+1}\dots\al_{j_{k+1}+l_{k+1}}$. We define an interval $I_{k+1}=[\underline{t}_{k+1},\overline{t}_{k+1})$, where $\underline{t}_{k+1}$ and $\overline{t}_{k+1}$ are given implicitly by
\begin{align*}
b(\underline{t}_{k+1},\beta)&=\v_{k+1}^-(\al_1\dots\al_{j_{k+1}}^-)^\f=\v_{k+1}^-\big({\L(\cs)^+}\v_1\dots\v_{k-1}\v_k^-\big)^\f,\\
b(\overline{t}_{k+1},\beta)&=(\v_{k+1}^*)^\f,
\end{align*}
where $\v_{k+1}^*$ is the smallest $\beta$-Lyndon word greater than or equal to $\v_{k+1}$.
However, in the special case when $\v_{k+1}$ is $\beta$-Lyndon and $\si^{j_{k+1}}(\al(\beta))=\v_{k+1}^\f$, we take the interval $I_{k+1}$ to be closed, i.e. $I_{k+1}=[\underline{t}_{k+1},\overline{t}_{k+1}]$. We add $I_{k+1}$ to the collection $\II$.

This procedure either eventually stops, with a finite collection $\II=(I_1,\dots,I_{k_0})$, or it continues ad infinitum, producing an infinite sequence $\II=(I_1,I_2,\dots)$ of intervals. Note that if $I_k$ is closed for some $k$, then $\v_l=\v_k$ for all $l>k$ and hence $I_l\subseteq I_k$ for all $l>k$.


In case $\alpha(\beta)$ is periodic, i.e. $\alpha(\beta)=(\al_1\dots\al_m)^\f$ for some {minimal $m\in\N$, then $\al_m{<M_\beta}$ and} we replace $(\al_i)$ in the above procedure with the greedy expansion of $1$, that is, the sequence $(\al_i'):=\al_1\dots{\al_m^+}0^\f$. The inductive step needs to be slightly modified: If $\si^{j_{k+1}}((\al_i'))=0^\f$, we leave $l_{k+1}$ undefined and stop the construction with the finite set $\II=(I_1,\dots,I_k)$. We now verify that in this case, too, $\v_k\prec\cs$ for all $k$ and so $\v_k^*$ is well defined by Lemma \ref{lem:existence-of-v-star}. 

{\em Case 1.} $k=1$. If the period of $\al(\beta)$ is greater than $j_1+l_1$, then $\al(\beta)$ begins with $\L(\cs)^+\v_1$ and the argument is the same as before. Otherwise, we have $\al(\beta)=(\L(\cs)^+\v_1^-)^\f$. Since $\beta<\beta_*^\cs$, this gives
\[
\cs^-\L(\cs)^\f\succ \si^{j_1}(\al(\beta))=(\v_1^-\L(\cs)^+)^\f.
\]
This implies $\v_1^-\lle (\cs^-\L(\cs)^\f)_{1:l_1}$. But if this holds with equality, then we obtain
\[
\si^{l_1}(\cs^-\L(\cs)^\f)\succ (\L(\cs)^+\v_1^-)^\f,
\]
which is impossible since $\cs^-\L(\cs)^\f$ does not contain the word $\L(\cs)^+$. Hence, we have
\begin{equation} \label{eq:v1-small}
\v_1\lle (\cs^-\L(\cs)^\f)_{1:l_1},
\end{equation}
and this implies $\v_1\prec\cs$.

{\em Case 2.} $k\geq 2$. Here the situation is more straightforward: by the construction of $\v_k$, we have
\begin{equation} \label{eq:v_k-small}
\v_k=\al_{j_k+1}'\dots\al_{j_k+l_k}'\lle \al_{j_1+1}'\dots\al_{j_1+l_k}'=\al_{j_1+1}\dots\al_{j_1+l_k}\lle (\cs^-\L(\cs)^\f)_{1:l_k}\prec \cs.
\end{equation}

The following special case of Theorem \ref{thm:general-transitivity} is our main result in this section.

\begin{theorem} \label{thm:basic-interval-transitivity}
Let $\beta\in(\beta_\ell,\beta_*)$ for some basic interval $[\beta_\ell,\beta_*]=[\beta_\ell^\cs,\beta_*^\cs]$, where $\cs\in\La_1={F_e}$, and construct the collection $\II$ of intervals as above. Let $[t_L,t_R]\subseteq [0,\tau(\beta)]$ be a $\beta$-Lyndon interval. Then $\Kt_\beta(t_R)$ is transitive if and only if $t_R\not\in \bigcup_{I\in\II}I$. Furthermore, the entropy of $\Kt_\beta(t)$ is constant on each interval $I\in\II$.
\end{theorem}

The situation for $\cs\in\La_k$ with $k\geq 2$ is a bit more involved, and we deal with it later, in Section \ref{sec:higher-order-basic}. However, it is convenient to develop the necessary machinery in greater generality here.

Before proving the theorem, which shows the significance of the words $\v_k$ and intervals $I_k$, we first give some examples.

\begin{example} \label{ex:v_k}
Let $\cs=011$, so $\cs\in\F$ and the basic interval $[\beta_\ell^{\cs},\beta_*^{\cs}]$ is given by
\[
\al(\beta_\ell^{\cs})=\L(\cs)^\f=(110)^\f, \qquad \al(\beta_*^{\cs})=\L(\cs)^+\cs^-\L(\cs)^\f=111\,010\,(110)^\f.
\]
\begin{enumerate}[(a)]
\item Suppose $\al(\beta)=111\,01\,00110111\,(001)^\f$. Then $\v_1=01$, $\v_2=00110111$, and $\v_k=001$ for all $k\geq 3$. Since each $\v_k$ is $\beta$-Lyndon, we have $\v_k^*=\v_k$ for each $k$. Now our construction above gives rise to the intervals $I_1=[\underline{t}_1,\overline{t}_1), I_2=[\underline{t}_2,\overline{t}_2), I_3=[\underline{t}_2,\overline{t}_3],\dots$ given by
\begin{alignat*}{2}
b(\underline{t}_1,\beta)&=\v_1^-\L(\cs)^\f=00\,(110)^\f, & \quad b(\overline{t}_1,\beta)&=\v_1^\f=(01)^\f,\\
b(\underline{t}_1,\beta)&=\v_2^-\big(\L(\cs)^+\v_1^-\big)^\f=00110110\,(11100)^\f, & \quad b(\overline{t}_2,\beta)&=\v_2^\f=(00110111)^\f,\\
b(\underline{t}_3,\beta)&=\v_3^-\big(\L(\cs)^+\v_1\v_2^-\big)^\f=000\,(111\,01\,00110110)^\f, & \quad b(\overline{t}_3,\beta)&=\v_3^\f=(001)^\f,
\end{alignat*}
and for $k\geq 4$,
\begin{align*}
b(\underline{t}_k,\beta)&=\v_k^-\big(\L(\cs)^+\v_1\dots\v_{k-1}^-\big)^\f=000\,\big(111\,01\,00110111\,(001)^{k-4}\,000\big)^\f,\\
b(\overline{t}_k,\beta)&=\v_k^\f=(001)^\f.
\end{align*}
{(By our construction, we take $I_k=[\underline{t}_k,\overline{t}_k]$ to be closed for $k\geq 3$.)}
Observe that $I_2\subseteq I_1$, $I_3$ lies to the left of $I_2$, and $I_{k+1}\subseteq I_k$ for all $k\geq 3$. 
\item Suppose $\al(\beta)=111\,01\,001(01)^\f$. Then $\v_1=01$, and the procedure stops after one step, giving rise to just one interval $I_1=[\underline{t}_1,\overline{t}_1)$, which is the same as in (a).
\item More extremely, let $\al(\beta)=111\,001(01)^\f$. Then $\si^{j_1+n}(\al(\beta))\succ \si^{j_1}(\al(\beta))$ for all $n\geq 1$ (where $j_1=|\cs|=3$), and hence we obtain the empty collection $\II=\emptyset$.
\item Let $\al(\beta)=111\,01\,001\,0001\dots$. Then $\v_k=0^k1$ for $k\in\N$, and we get an infinite collection $\mathcal{I}=(I_1,I_2,\dots)$ of pairwise disjoint intervals, with left endpoints given by
\begin{gather*}
b(\underline{t}_1,\beta)=00\,(110)^\f, \qquad b(\underline{t}_2,\beta)=000\,(11100)^\f,\\
b(\underline{t}_k,\beta)=0^{k+1}\,\big(111\,01\,001\,\dots\,0^{k-2}1\,0^k\big)^\f \qquad (k\geq 3),
\end{gather*}
and right endpoints given by
\[
b(\overline{t}_k,\beta)=(0^k1)^\f, \qquad k=1,2,\dots.
\]
Note that in this example, the intervals $I_k$ accumulate at $0$.
\item Suppose $\al(\beta)=111\,00111\,(001)^\f$. Then $\v_1=00111$ and $\v_k=001$ for all $k\geq 2$. Note that here $\v_1$ is not $\beta$-Lyndon since $\si^2(\v_1^\f)=(11100)^\f\succ\al(\beta)$. It is easy to check that $\v_1^*=01$. Thus, we get a collection of intervals $I_1=[\underline{t}_1,\overline{t}_1), I_2=[\underline{t}_2,\overline{t}_2],\dots$ given by
\begin{alignat*}{2}
b(\underline{t}_1,\beta)&=00110\,(110)^\f, & \qquad b(\overline{t}_1,\beta)&=(01)^\f,\\
b(\underline{t}_2,\beta)&=000\,(111\,00110)^\f, & \qquad b(\overline{t}_2,\beta)&=(001)^\f,\\
b(\underline{t}_k,\beta)&=000\,\big(111\,00111\,(001)^{k-3}\,000\big)^\f, & \qquad b(\overline{t}_k,\beta)&=(001)^\f \qquad (k\geq 3).
\end{alignat*}
{(Here we take $I_k$ to be closed for all $k\geq 2$.)}

\item Let $\al(\beta)=(111\,0101100)^\f$. Since $\al(\beta)$ is periodic, we apply the construction to the greedy expansion $\al_1'\al_2'\dots=111\,01011\,01\,0^\f$. We find $\v_1=01011$ and $\v_2=01$. The intervals $I_1=[\underline{t}_1,\overline{t}_1)$ and $I_2=[\underline{t}_2,\overline{t}_2)$ are given by
\begin{alignat*}{2}
b(\underline{t}_1,\beta)&=01010\,(110)^\f, & \qquad b(\overline{t}_1,\beta)&=(01011)^\f,\\
b(\underline{t}_2,\beta)&=00\,(111\,01010)^\f, & \qquad b(\overline{t}_2,\beta)&=(01)^\f.
\end{alignat*}
\end{enumerate}
\end{example}

{
\begin{example} \label{ex:v_k-larger-alphabet}
Let $\cs=2\in F_e$, so the basic interval $[\beta_\ell^{\cs},\beta_*^{\cs}]\subseteq(3,4]$ is given by
\[
\al(\beta_\ell^{\cs})=\L(\cs)^\f=2^\f, \qquad \al(\beta_*^{\cs})=\L(\cs)^+\cs^-\L(\cs)^\f=312^\f.
\]
\begin{enumerate}[(a)]
\item Suppose $\al(\beta)=3(12)^\f$. Then $\v_k=12$ for all $k\in\N$.
\item Suppose $\al(\beta)=30123(12)^\f$. Then $\mathcal{I}=\emptyset$; compare with Example \ref{ex:v_k} (c).
\item Suppose $\al(\beta)=312\,0312\,02\,002\,0002,\dots$. Then $\v_1=0312$ and $\v_k=0^{k-1}2$ for all $k\geq 2$. Note that $\v_1$ is not $\beta$-Lyndon since $\si(\v_1^\f)=(3120)^\f\succ\al(\beta)$. We have $\v_1^*=1$.
\item Suppose $\al(\beta)=(3122102)^\f$. Then $(\al_i')=31221030^\f$, and we find $\v_1=122$, $\v_2=1$ and $\v_3=03$.
\end{enumerate}
We leave it to the interested reader to write down the corresponding intervals $(I_k)$ for these examples. Observe that, when $\al(\beta)$ uses more than two different digits, some of the words $\v_k$ or $\v_k^*$ may consist of a single digit. This can only happen when $\beta>2$. 
\end{example}
}

The proof of Theorem \ref{thm:basic-interval-transitivity} is rather involved.
We first prove several lemmas about the words $\v_k$. { The main goal of these lemmas is to derive an explicit construction for the words $\v_k^*$, which we do in Lemma \ref{lem:v_k-star}. This construction is then used in} Proposition \ref{prop:non-transitivity-intervals} to describe the relative placements of the intervals $I_k$. 

The first lemma deals specifically with the case of periodic $\al(\beta)$; see \cite[Lemma 3.8]{Kalle-Kong-Langeveld-Li-18}.

\begin{lemma} \label{lem:periodic-permutations}
Let $\beta{>1}$, and suppose $\al(\beta)=(\al_1\dots\al_m)^\f$ where $m$ is the minimal period of $\al(\beta)$. Then
\begin{equation*} 
\al_{j+1}\dots\al_m \prec \al_1\dots\al_{m-j} \qquad \forall\,1\leq j<m.
\end{equation*}
\end{lemma}

{Recall that $\v_k=\al_{j_k+1}\ldots \al_{j_k+l_k}$.} Below we write $\v_k=v_{k,1}\dots v_{k,l_k}$. 

\begin{lemma} \label{lem:always-below-alpha}
For any $k$ and any two integers $0\leq i_1<i_2\leq l_k$,
\begin{equation} \label{eq:always-below-alpha}
v_{k,i_1+1}\dots v_{k,i_2}\lle \al_1\dots\al_{i_2-i_1}.
\end{equation}
\end{lemma}

\begin{proof}
When $\al(\beta)$ is aperiodic the inequality follows immediately from Lemma \ref{lem:quasi-greedy expansion-alpha-q}, since $\v_k$ is a subword of $\al(\beta)$. But in the periodic case some extra care is needed. In this case there is a positive integer $k_0$ such that $\II=(I_1,\dots,I_{k_0})$, and $\al(\beta)=(\L(\cs)^+\v_1\dots\v_{k_0}^-)^\f$. The only case that needs extra attention is when $k=k_0$ and $i_2=l_{k_0}$. Set $i:=i_1$. It seems theoretically possible that $v_{k_0,i+1}\dots v_{k_0,l_{k_0}}=\al_1\dots\al_{l_{k_0}-i}^+$, violating \eqref{eq:always-below-alpha}. However, if that were the case we would have
\[
\al_1\dots\al_m=\al_1\dots\al_{j_{k_0}}\v_{k_0}^-=\al_1\dots\al_{j_{k_0}}v_{k_0,1}\dots v_{k_0,i}\al_1\dots \al_{l_{k_0}-i},
\]
where $m$ is the minimal period of $\al(\beta)$, so applying Lemma \ref{lem:periodic-permutations} with $j=l_{k_0}-i$ gives
\[
\si^{j_{k_0}+i}(\al(\beta))=\al_1\dots\al_{l_{k_0}-i}(\al_1\dots\al_m)^\f\succ \al(\beta).
\]
This contradiction shows that \eqref{eq:always-below-alpha} holds for this case as well.
\end{proof}

\begin{lemma} \label{lem:prefix-beta-Lyndon}
Let $k\geq 2$. If $\v_k$ is a proper prefix of $\v_{k-1}$, then $\v_k$ is $\beta$-Lyndon.
\end{lemma}

\begin{proof}
By the hypothesis, $l_k<l_{k-1}$ and we can write $\v_k=v_{k-1,1}\dots v_{k-1,l_k}$. Fix $0\leq i<l_k$. By \eqref{eq:always-below-alpha}, $v_{k-1,i+1}\dots v_{k-1,l_k}\lle \al_1\dots\al_{l_k-i}$. Write $l_{k-1}=ql_k+r$, where $q\in\N$ and $1\leq r\leq l_k$. Since $\v_{k-1}$ is Lyndon, we have
\[
\v_k=v_{k-1,1}\dots v_{k-1,l_k}\lle v_{k-1,pl_k+1}\dots v_{k-1,(p+1)l_k}, \qquad p=0,1,\dots,q-1,
\]
and
\[
v_{k-1,1}\dots v_{k-1,r}\prec v_{k-1,ql_k+1}\dots v_{k-1,l_{k-1}}.
\]
Thus, by Lemma \ref{lem:always-below-alpha} it follows that
\[
v_{k-1,i+1}\dots v_{k-1,l_k} \v_k^{q-1} v_{k-1,1}\dots v_{k-1,r} \prec v_{k-1,i+1}\dots v_{k-1,l_{k-1}} \lle \al_1\dots a_{l_{k-1}-i},
\]
and therefore, $\si^i(\v_k^\f)\prec\al(\beta)$. Since $i$ was arbitrary, this shows $\v_k$ is $\beta$-Lyndon.
\end{proof}

\begin{lemma} \label{lem:v_k-decreasing}
Let $k\geq 2$, and suppose $\v_{k}\neq \v_{k-1}$. Then
\begin{enumerate}[{\rm(i)}]
\item $\v_{k-1}$ is not a prefix of $\v_{k}$; and
\item $\v_k \prec \v_{k-1}$.
\end{enumerate}
\end{lemma}

\begin{proof}
Since $\si^{j_k}(\al(\beta))=\si^{j_{k-1}+l_{k-1}}(\al(\beta))\lle \si^{j_{k-1}}(\al(\beta))$, it follows that $\al_{j_{k-1}+1}\dots\al_{j_{k-1}+m}$ cannot be Lyndon for any $m>l_{k-1}$. Suppose $\v_k$ is a proper extension of $\v_{k-1}$. An easy exercise (which we leave to the interested reader) shows that the concatenation $\v_{k-1}\v_k$ is then also Lyndon. But $\v_{k-1}\v_k=\al_{j_{k-1}+1}\dots\al_{j_{k-1}+l_{k-1}+l_k}$, yielding a contradiction. This proves (i). Statement (ii) now follows since, if $\v_k\succ \v_{k-1}$, there is a smallest $m<l_{k-1}$ such that $v_{k,1}\dots v_{k,m}\succ v_{k-1,1}\dots v_{k-1,m}$. But then $\al_{j_{k-1}+1}\dots\al_{j_{k-1}+l_{k-1}+m}=\v_{k-1}v_{k,1}\dots v_{k,m}$ is Lyndon, again leading to a contradiction. (In the case when $\al(\beta)$ is periodic, one replaces $(\al_i)$ with $(\al_i')$ throughout.)
\end{proof}


\begin{lemma} \label{lem:header-suffix}
For each $k\in\N$, the word $\al_1\dots\al_{j_k}$ can occur in $\v_k$ at most once, and only at the end. In other words, the truncated word $v_{k,1}\dots v_{k,l_k-1}$ does not contain the word $\al_1\dots\al_{j_k}$.
\end{lemma}

\begin{proof}
Suppose that $\v_k=v_{k,1}\dots v_{k,p}\al_1\dots\al_{j_k}u_1\dots u_q$ (say), where $p\geq 0$ and $q\geq 1$. Then, on the one hand, $u_1\dots u_q\succ v_{k,1}\dots v_{k,q}$ since $\v_k$ is Lyndon. On the other hand, applying \eqref{eq:always-below-alpha} with $i_1=p$ and $i_2=l_k$ we obtain $u_1\dots u_q\lle \al_{j_k+1}\dots\al_{j_k+q}=v_{k,1}\dots v_{k,q}$. This contradiction completes the proof. 
\end{proof}

\begin{lemma} \label{lem:v_k-upper-bound}
\begin{enumerate}[{\rm(i)}]
\item For each $k\in\N$, we have
\begin{equation} \label{eq:v_k-upper-bound}
\v_k\lle\al_{m+1}\dots \al_{m+l_k} \qquad\forall\, 0\leq m<j_k.
\end{equation}
\item Suppose either $k=1$, or $k\geq 2$ and $\v_k\neq \v_{k-1}$. If $\v_k$ is not $\beta$-Lyndon, then
\begin{equation} \label{eq:strict-v_k-upper-bound}
\v_k\prec\al_{m+1}\dots \al_{m+l_k} \qquad\forall\, 0\leq m<j_k.
\end{equation}
\end{enumerate}
\end{lemma}

\begin{proof}
We prove the two statements by induction on $k$. Take first $k=1$. Recall the inequality \eqref{eq:v1-small}.
It follows that, for any $m<j_1$,
\begin{align}
\begin{split}
\v_1&\lle (\cs 0^\f)_{1:l_1}\lle (\al_{m+1}\dots \al_{j_1}^-\al_1\dots\al_m 0^\f)_{1:l_1}\\
&\lle (\al_{m+1}\dots \al_{j_1}0^\f)_{1:l_1}\lle \al_{m+1}\dots \al_{m+l_1},
\end{split}
\label{eq:v_1-upper-bound}
\end{align}
where the second inequality holds because $\L(\cs)=\al_1\dots\al_{j_1}^-$, and so $\al_{m+1}\dots \al_{j_1}^-\al_1\dots\al_m$ is a cyclic permutation of $\cs$. Observe that the third inequality in \eqref{eq:v_1-upper-bound} is strict when $l_1\geq j_1-m$.

Now suppose that $\v_1$ is not $\beta$-Lyndon, and $l_1<j_1-m$. Observe that $\cs$ is $\beta$-Lyndon since $\si^n(\cs^\f)\lle \L(\cs)^\f \prec\al(\beta)$ for each $n\geq 0$. Thus, by Lemma \ref{lem:no-beta-Lyndon-extension}, $\v_1$ cannot be a prefix of $\cs$. But $|\v_1|=l_1<j_1=|\cs|$, and hence $\v_1\prec s_1\dots s_{l_1}\lle \al_{m+1}\dots\al_{m+l_1}$ as in \eqref{eq:v_1-upper-bound}. This completes the basis for the induction.

Next, let $k\geq 2$ be given, and suppose that $\v_{k-1}\lle \al_{m+1}\dots\al_{m+l_{k-1}}$ for all $m<j_{k-1}$. Take $m<j_k$, and consider two cases:

\medskip
{\em Case 1.} $m\leq j_{k-1}$. If $l_k\leq l_{k-1}$, then Lemma \ref{lem:v_k-decreasing} (ii) and the induction hypothesis imply that $\v_k\lle v_{k-1,1}\dots v_{k-1,l_k}\lle \al_{m+1}\dots \al_{m+l_k}$, giving \eqref{eq:v_k-upper-bound}. Furthermore, the first inequality is strict in case $\v_k\neq \v_{k-1}$ and $\v_k$ is not $\beta$-Lyndon, in view of Lemma \ref{lem:prefix-beta-Lyndon}, so in this case we have \eqref{eq:strict-v_k-upper-bound}. If, on the other hand, $l_k>l_{k-1}$, then $v_{k,1}\dots v_{k,l_{k-1}}\prec \v_{k-1}\lle \al_{m+1}\dots\al_{m+l_{k-1}}$, again by the induction hypothesis, where the first inequality follows from Lemma \ref{lem:v_k-decreasing}.

\medskip
{\em Case 2.} $m>j_{k-1}$. Then $j_k-m<j_k-j_{k-1}=l_{k-1}$, and we use the Lyndon property of $\v_{k-1}$ to argue as follows: If $j_k-m\leq l_k$, then by Lemma \ref{lem:v_k-decreasing} (ii) we have
\[
v_{k,1}\dots v_{k,j_k-m}\lle v_{k-1,1}\dots v_{k-1,j_k-m} \prec v_{k-1,m-j_{k-1}+1}\dots v_{k-1,l_{k-1}}=\al_{m+1}\dots\al_{j_k},
\]
which implies \eqref{eq:strict-v_k-upper-bound}; and if $j_k-m>l_k$, then {$l_k<l_{k-1}$ and the definition of $\v_k$ implies that 
\begin{equation} \label{eq:v_k-comparison-case-2}
\v_k\lle \al_{j_{k-1}+1}\ldots \al_{j_{k-1}+l_k}\lle \al_{m+1}\ldots \al_{m+l_k},
\end{equation}
where the last inequality holds since $\v_{k-1}=\al_{j_{k-1}+1}\ldots \al_{j_{k-1}+l_{k-1}}$ is Lyndon. (In case $\al(\beta)$ is periodic, we still obtain \eqref{eq:v_k-comparison-case-2} by reasoning as in \eqref{eq:v_k-small}.) If $\v_k$ is not $\beta$-Lyndon, then $\v_k$ is not a proper prefix of $\v_{k-1}$ by Lemma \ref{lem:prefix-beta-Lyndon}, so the first inequality in \eqref{eq:v_k-comparison-case-2} becomes strict.}
\end{proof}

\begin{lemma} \label{lem:v_k-beta-Lyndon}
Let $k\in\N$, and suppose either $k=1$, or $k\geq 2$ and $\v_k\neq\v_{k-1}$. Then $\v_k$ is $\beta$-Lyndon if and only if $\v_k$ does not end in $\al_1\dots\al_{j_k}$.
\end{lemma}

\begin{proof}
If $\v_k=\u \al_1\dots\al_{j_k}$ for some word $\u$, then
\[
\si^{|\u|}(\v_k^\f)=(\al_1\dots\al_{j_k}\u)^\f=(\al_1\dots\al_{j_k+|\u|})^\f\lge\al(\beta),
\]
and hence $\v_k$ is not $\beta$-Lyndon.

Vice versa, suppose $\v_k$ is not $\beta$-Lyndon. Then there is an index $i\leq l_k$ such that
\begin{equation} \label{eq:bigger-cyclic-shift}
(v_{k,i+1}\dots v_{k,l_k}v_{k,1}\dots v_{k,i})^\f = v_{k,i+1}\dots v_{k,l_k}\v_k^\f\lge \al(\beta).
\end{equation}
In view of \eqref{eq:always-below-alpha}, this implies
\begin{equation} \label{eq:v_k-overlap-with-alpha}
v_{k,i+1}\dots v_{k,l_k}=\al_1\dots\al_{l_k-i}.
\end{equation}
We claim that $l_k-i=j_k$, and so $\v_k$ ends in $\al_1\dots\al_{j_k}$.
First, if $l_k-i>j_k$, then \eqref{eq:v_k-overlap-with-alpha} gives
\[
v_{k,i+j_k+1}\dots v_{k,l_k}=\al_{j_k+1}\dots\al_{l_k-i}=v_{k,1}\dots v_{k,l_k-i-j_k},
\]
contradicting that $\v_k=v_{k,1}v_{k,2}\ldots v_{k,l_k}$ is Lyndon. On the other hand, if $l_k-i<j_k$, then $\v_k\prec \al_{l_k-i+1}\dots\al_{2l_k-i}$ by Lemma \ref{lem:v_k-upper-bound} (ii), and together with \eqref{eq:v_k-overlap-with-alpha}, this contradicts \eqref{eq:bigger-cyclic-shift}.

Thus, $l_k-i=j_k$, and the proof is complete.
\end{proof}

We are now ready to characterize $\v_k^*$, the lexicographically smallest $\beta$-Lyndon word such that $\v_k^*\lge \v_k$.

\begin{lemma} \label{lem:v_k-star}
Let $k\geq 1$, and suppose $\v_k$ is not $\beta$-Lyndon. If $k\geq 2$ and $\v_k=\v_{k-1}$, then $\v_k^*=\v_{k-1}^*$. Otherwise, there exist a nonnegative integer $r$ and a word $\u$ not ending in $\al_1\dots\al_{j_k}^-$ such that
\begin{equation} \label{eq:v_k-expression}
\v_k=\u(\al_1\dots\al_{j_k}^-)^r\al_1\dots\al_{j_k},
\end{equation}
and we have
\begin{equation} \label{eq:v_k-star}
\v_k^*=\begin{cases}
\v_{k-1}^* & \mbox{if $k\geq 2$ and $\u=\v_{k-1}^-$},\\
\u^+ & \mbox{otherwise}.
\end{cases}
\end{equation}
\end{lemma}

\begin{proof}
{Obviously, $\v_k^*=\v_{k-1}^*$ if $\v_k=\v_{k-1}$. In the following we}
assume that either $k=1$, or $k\geq 2$ and $\v_k\neq \v_{k-1}$. {Since $\v_k$ is not $\beta$-Lyndon, it ends in $\al_1\dots\al_{j_k}$ by Lemma \ref{lem:v_k-beta-Lyndon}}, and this implies the existence of a word $\u$ and integer $r\geq 0$ satisfying \eqref{eq:v_k-expression}. Before determining $\v_k^*$, we first show that $\u$ cannot end in a prefix of $\al(\beta)$. For, suppose $\u$ ends in $\al_1\dots\al_m$ for some $m\geq 1$. We claim that this contradicts Lemma \ref{lem:header-suffix}. Note by \eqref{eq:v_k-expression} and \eqref{eq:always-below-alpha} that
\[
(\al_1\dots\al_{j_k}^-)^r\al_1\dots\al_{j_k}\lle \al_{m+1}\dots \al_{m+(r+1)j_k}\lle \al_1\ldots \al_{(r+1)j_k}.
\]
This implies that $\al_{m+1}\dots\al_{m+j_k}=\al_1\dots\al_{j_k}$ or $\al_1\dots\al_{j_k}^-$. By \eqref{eq:v_k-expression} it follows that $\v_k$ contains the word $\al_1\dots\al_m\al_{m+1}\dots\al_{m+j_k-1}$, followed by one more digit. Since $m\geq 1$, this means that $\v_k$ contains the word $\al_1\dots\al_{j_k}$ prematurely, contradicting Lemma \ref{lem:header-suffix}.

We next prove the expression \eqref{eq:v_k-star}. First, let $k=1$. We need to show that $\v_1^*=\u^+$. Note that $\v_1^*\succ \v_1$ and $|\v_1^*|<|\v_1|$ by Lemma \ref{lem:existence-of-v-star}. 

First we claim that $\u$ cannot be a prefix of $\v_1^*$. {This involves two cases.

{\em Case 1.} $r=0$. Then $\v_1=\u\al_1\dots\al_{j_1}$ by \eqref{eq:v_k-expression}. Note that $\v_1^*$ cannot equal $\u$ because $(\v_1^*)^\f\succ\v_1^\f$ and $\v_1$ is Lyndon, which means $\v_1^*$ cannot be a prefix of $\v_1$. So if $\u$ is a prefix of $\v_1^*$, then since $|\v_1^*|<|\v_1|$ there is some index $i<j_1$ such that $\v_1^*=\u\al_1\dots\al_i^+$. But then $\v_1^*$ is not $\beta$-Lyndon, a contradiction.

{\em Case 2.} $r\geq 1$. Here} we need to rule out the possibility that $\v_1^*$ might equal $\w:=\u(\al_1\dots\al_{j_1}^-)^p\al_1\dots\al_{j_1}$ for some $p<r$. But notice that
\begin{equation} \label{eq:another-big-tail}
\si^{|\u|+pj_1}(\w^\f)=\al_1\dots\al_{j_1}\w^\f\succ \al_1\dots\al_{j_1}\v_1^\f=(\al_1\dots\al_{l_1})^\f\lge \al(\beta),
\end{equation}
where the second equality holds since $\v_1$ ends in $\al_1\dots\al_{j_1}$, so $l_1>j_1$ and 
$$\v_1=\al_{j_1+1}\dots\al_{j_1+l_1}=\al_{j_1+1}\dots\al_{l_1}\al_1\dots\al_{j_1}.$$
By \eqref{eq:another-big-tail}, $\w$ is not $\beta$-Lyndon, and so $\u$ cannot be a prefix of $\v_1^*$. This proves the claim.

It now suffices to show that $\u^+$ is $\beta$-Lyndon. Note $\u^+$ is Lyndon because $\u$ is a prefix of the Lyndon word $\v_1$.
Recall that $j_1=|\cs|$. Write $\cs=s_1\dots s_{j_1}$ and $\u=u_1\dots u_p$. If $p\leq j_1$ and $\u=s_1\dots s_p^-$, then $\u^+=s_1\dots s_p$ is $\beta$-Lyndon by Lemma \ref{lem:no-beta-Lyndon-extension}, because $\cs$ is $\beta$-Lyndon. So assume $\u$ is not of this form. We showed above that $\u$ cannot end in a prefix of $\al(\beta)$, so if $\u^+$ is not $\beta$-Lyndon, then $\u$ must end in $\al_1\dots\al_m^-$ for some $m\geq 1$, i.e. $\u=u_1\dots u_{p-m}\al_1\dots\al_m^-$. Observe that $m\leq j_1$ by Lemma \ref{lem:header-suffix}, and $m\neq j_1$ because we assumed that $\u$ does not end in $\al_1\dots\al_{j_1}^-$. Thus, $m<j_1$.

We claim next that $\u$ is not a prefix of $\cs^-$. Suppose it is, so $\u=s_1\dots s_p$ where $p<j_1$. Then
\[
\v_1^-=s_1\dots s_p(\al_1\dots\al_{j_1}^-)^{r+1}=s_1\dots s_p \L(\cs)^{r+1},
\]
so
\begin{equation*}
\al(\beta)\lge (\L(\cs)^+\v_1^-)^\f=\L(\cs)^+s_1\dots s_p \L(\cs)^{r+1}(\L(\cs)^+\v_1^-)^\f \succ \L(\cs)^+\cs^-\L(\cs)^\f,
\end{equation*}
where the last inequality follows from the definition of $\L(\cs)$. This contradicts that $\beta\in [\beta_\ell,\beta_*]$.

Now there are two cases left to consider.
\begin{enumerate}[(a)]
\item $\cs^-$ is a proper prefix of $\u$. Recall that $m<j_1$. Since
\[
u_1\dots u_{j_1-m}=s_1\dots s_{j_1-m}\lle \L(\cs)_{m+1:j_1}\prec \L(\cs)^+_{m+1:j_1}=\al_{m+1}\dots\al_{j_1},
\]
it follows that $\al_1\dots\al_m u_1\dots u_{j_1-m}\prec \al_1\dots\al_{j_1}$, which yields $\si^{p-m}((\u^+)^\f)\prec\al(\beta)$.
\item There is an index $i<\min\{p,j_1\}$ such that $u_1\dots u_{i-1}=s_1\dots s_{i-1}$ and $u_i<s_i$. Then again, $u_1\dots u_i\prec s_1\dots s_i\lle \al_{m+1}\dots \al_{m+i}$, with the same conclusion as in case (a).
\end{enumerate}
Therefore, $\u^+$ is $\beta$-Lyndon, and $\v_1^*=\u^+$.

\medskip

Next, let $k\geq 2$ and assume $\v_k\neq \v_{k-1}$. By the same argument as above, $\u$ cannot be a prefix of $\v_k^*$, so $\v_k^*=(\u^+)^*$. We first make a few observations:
\begin{enumerate}[1.]
\item $\u$ cannot be a prefix of $\v_{k-1}$. Suppose it is. Recall that $\v_k\lle \v_{k-1}$. Since $|\v_{k-1}|<j_k$ and $\v_k$ begins with $\u\al_1\dots\al_{j_k-1}$, the constraint \eqref{eq:always-below-alpha} implies that $\v_{k-1}=\u\al_1\dots\al_m$ for some $m<j_k$. But then $\v_{k-1}$ is a prefix of $\v_k$, contradicting Lemma \ref{lem:v_k-decreasing} (i).
\item $\u$ does not extend $\v_{k-1}$, because $\v_k$ does not extend $\v_{k-1}$.
\end{enumerate}
This leaves the following cases, in which we set $p:=|\u|$.
\begin{enumerate}[(a)]
\item If $\u=\v_{k-1}^-$, it follows immediately that $(\u^+)^*=\v_{k-1}^*$, {yielding the first case in \eqref{eq:v_k-star}.}
\item If $\v_{k-1}$ properly extends $\u^+$ and $\v_{k-1}$ is $\beta$-Lyndon, then $\u^+$ is $\beta$-Lyndon as well by Lemma \ref{lem:no-beta-Lyndon-extension}.
\item If $\v_{k-1}$ properly extends $\u^+$ and $\v_{k-1}$ is not $\beta$-Lyndon, then for any $m$ such that $\u$ ends in $\al_1\dots\al_m^-$ we must have $m<j_{k-1}$, since otherwise $\v_{k-1}$ contains the word $\al_1\dots\al_{j_{k-1}}$ prematurely, contradicting Lemma \ref{lem:header-suffix}. Furthermore, $\v_{k-1}\prec \al_{m+1}\dots\al_{m+l_{k-1}}$ by Lemma \ref{lem:v_k-upper-bound} (ii). Since $\u^+$ and $\v_{k-1}$ are both Lyndon, this implies $(\u^+)^\f\prec \v_{k-1}0^\f\prec\si^m(\al(\beta))$, and hence $\si^{p-m}((\u^+)^\f)\prec\al(\beta)$. Thus, $\u^+$ is $\beta$-Lyndon.
\item {In the remaining cases, there is an $i\leq\min\{|\u|,|\v_{k-1}|\}$ such that $(\u^+)_{1:i}\prec v_{k-1,1}\dots v_{k-1,i}$. Suppose $\u=u_1\dots u_{p-m}\al_1\dots\al_m^-$ for some $m\geq 1$. As in the case $k=1$ above, we can argue that $m<j_k$. Now we have three subcases:

(d1) $m\leq j_{k-1}$. Then $(\u^+)_{1:i}\prec v_{k-1,1}\dots v_{k-1,i}\lle \al_{m+1}\dots\al_{m+i}$ by Lemma \ref{lem:v_k-upper-bound} (i).

(d2) $j_{k-1}<m\leq j_k-i$. Then $m-j_{k-1}+i\leq j_k-j_{k-1}=l_{k-1}$, and using that $\v_{k-1}$ is Lyndon,
\[
(\u^+)_{1:i}\prec v_{k-1,1}\dots v_{k-1,i}\lle v_{k-1,m-j_{k-1}+1}\dots v_{k-1,m-j_{k-1}+i}=\al_{m+1}\dots\al_{m+i}.
\]

(d3) $j_k-i<m<j_k$. Then $j_k-m<i$, and since $l_{k-1}+j_{k-1}=j_k$, we have
\[
(\u^+)_{1:j_k-m}\lle v_{k-1,1}\dots v_{k-1,j_k-m}\prec v_{k-1,m-j_{k-1}+1}\dots v_{l_{k-1}}=\al_{m+1}\dots\al_{j_k},
\]
again by the Lyndon property of $\v_{k-1}$.

In each case we have a prefix of $\u^+$ that is strictly smaller than the prefix of $\sigma^m(\al(\beta))$ of the same length. We conclude that} $\si^{p-m}((\u^+)^\f)\prec\al(\beta)$. Thus, $\u^+$ is $\beta$-Lyndon.
\end{enumerate}
From these four cases and $\v_k^*=(\u^+)^*$, \eqref{eq:v_k-star} follows.
\end{proof}

\begin{proposition} \label{prop:non-transitivity-intervals}
Let $k,\ell\in\N$ with $k<\ell$, and suppose $I_k=[\underline{t}_k,\overline{t}_k)$ and $I_\ell=[\underline{t}_\ell,\overline{t}_\ell)$ belong to $\II$. Then:
\begin{enumerate}[{\rm(i)}]
\item Either $\overline{t}_{\ell}\leq\underline{t}_k$ or $I_{\ell}\subseteq I_k$. In other words, $I_\ell$ is either contained in $I_k$ or else lies completely to the left of $I_k$.
\item $I_{\ell}\subseteq I_k$ if and only if
\begin{equation} \label{eq:containment-condition}
\v_{\ell}=\v_k\quad\mbox{or}\quad \v_{\ell}\ \mbox{begins with}\ \v_k^-(\al_1\dots\al_{j_k}^-)^n\al_1\dots\al_{j_k}\ \mbox{for some $n\geq 0$}.
\end{equation}
\end{enumerate}
\end{proposition}

\begin{proof}
It suffices to show that \eqref{eq:containment-condition} implies $I_{\ell}\subseteq I_k$, and $\overline{t}_{\ell}>\underline{t}_k$ implies \eqref{eq:containment-condition}.

Assume first that \eqref{eq:containment-condition} holds. If $\v_{\ell}=\v_k$, then $\underline{t}_{\ell}>\underline{t}_k$ and $\overline{t}_{\ell}=\overline{t}_k$, so $I_{\ell}\subseteq I_k$. If $\v_{\ell}\neq \v_k$, then by iterating Lemma \ref{lem:v_k-decreasing}, $\v_{\ell}\prec \v_k$ and $\v_{\ell}$ does not extend $\v_k$. Hence, we also have $\v_{\ell}^*\lle \v_k^*$ since the map $\v\mapsto \v^*$ which sends a Lyndon word $\v$ to the smallest $\beta$-Lyndon word $\v^*$ greater than or equal to $\v$, is clearly nondecreasing. Therefore, {by Lemma \ref{lem:repeating-preserves-order}}, $\overline{t}_{\ell}\leq \overline{t}_k$. On the other hand, $\v_\ell$ begins with $\v_k^-(\al_1\dots\al_{j_k}^-)^n\al_1\dots\al_{j_k}$ for some $n\geq 0$, and this implies $\underline{t}_\ell\geq \underline{t}_k$. Thus, $I_\ell\subseteq I_k$.

Next, we show by induction that for all $\ell\geq k$,
$\overline{t}_{\ell}>\underline{t}_k$ implies \eqref{eq:containment-condition}.
This is trivial for $\ell=k$, so take $\ell'>k$ and assume that $\overline{t}_{\ell}>\underline{t}_k$ implies \eqref{eq:containment-condition} for all $k\leq \ell<\ell'$. Suppose $\overline{t}_{\ell'}>\underline{t}_k$, so
\begin{equation} \label{eq:v_l-v_k-inequality}
(\v_{\ell'}^*)^\f \succ \v_k^-(\al_1\dots\al_{j_k}^-)^\f.
\end{equation}
Assume $\v_{\ell'}\neq \v_k$, as otherwise there is nothing to show.
We consider two cases:

\medskip
{\em Case 1.} $\v_{\ell'}$ is $\beta$-Lyndon. Then $\v_{\ell'}^*=\v_{\ell'}$ and hence
\[
\v_{\ell'}^\f \succ \v_k^-(\al_1\dots\al_{j_k}^-)^\f.
\]
Since $\v_k$ and $\v_{\ell'}$ are both Lyndon, $\v_{\ell'}^\f$ cannot begin with $\v_k$, and hence it must begin with $\v_k^-$. Thus, in view of the constraint \eqref{eq:always-below-alpha}, there exists some integer $r\geq 0$ such that $\v_{\ell'}^\f$ begins with
\begin{equation} \label{eq:v_l-beginning}
\v_k^-(\al_1\dots\al_{j_k}^-)^r\al_1\dots\al_{j_k}.
\end{equation}
Suppose, by way of contradiction, that $\v_{\ell'}$ is a proper prefix of this last word. {If $\v_{\ell'}$ is a prefix of $\v_k^-(\al_1\dots\al_{j_k}^-)^r$, then $\si^{|\v_{\ell'}|}(\v_{\ell'}^\f)=\v_{\ell'}^\f$ begins both with \eqref{eq:v_l-beginning}, and with a proper suffix of \eqref{eq:v_l-beginning} which contains the word $\al_1\dots\al_{j_k}$. Hence, there is an earlier occurrence of $\al_1\dots\al_{j_k}$ in \eqref{eq:v_l-beginning}. But this is impossible by Lemmas \ref{lem:quasi-greedy expansion-alpha-q}, \ref{lem:always-below-alpha} and \ref{lem:header-suffix}. This leaves the case when
\[
\v_{\ell'}=\v_k^-(\al_1\dots\al_{j_k}^-)^r \al_1\dots\al_m \qquad\mbox{for some $1\leq m<j_k$}.
\]
We then have (by shifting \eqref{eq:v_l-beginning} to the left by $|\v_{\ell'}|$)
\[
\big(\v_k^-(\al_1\dots\al_{j_k}^-)\big)_{1:j_k-m}=\al_{m+1}\dots\al_{j_k}.
\]
This implies that $l_k=|\v_k|>j_k-m$ by Lemma \ref{lem:v_k-upper-bound}, and 
\[
v_{k,1}\dots v_{k,j_k-m}=\al_{m+1}\dots\al_{j_k}.
\]
Again using Lemma \ref{lem:v_k-upper-bound}, this forces
\[
v_{k,j_k-m+1}\dots v_{k,l_k}\lle \al_{j_k+1}\dots\al_{l_k+m}=v_{k,1}\dots v_{k,l_k+m-j_k},
\]
}
contradicting that $\v_k$ is Lyndon. Therefore, $\v_{\ell'}$ begins with $\v_k^-(\al_1\dots\al_{j_k}^-)^r\al_1\dots\al_{j_k}$.


\medskip
{\em Case 2.} $\v_{\ell'}$ is not $\beta$-Lyndon. Let $\ell\leq \ell'$ be the largest integer such that $\v_{\ell}=\v_{\ell'}$.
Then $\v_{\ell}\neq \v_{\ell-1}$, $\v_{\ell}\neq \v_k$ since $\v_{\ell'}\neq \v_k$, and of course $\v_{\ell}$ is not $\beta$-Lyndon. By Lemma \ref{lem:v_k-star},
\begin{equation} \label{eq:v-ell-form}
\v_{\ell}=\u(\al_1\dots\al_{j_{\ell}}^-)^r\al_1\dots\al_{j_{\ell}}
\end{equation}
for some word $\u$ not ending in $\al_1\dots\al_{j_{\ell}}^-$ and some $r\geq 0$. If $\u=\v_k^-$, then $\v_\ell$ extends $\v_k^-\al_1\dots\al_{j_k}$ {since $j_\ell>j_k$}. So assume $\u\neq\v_k^-$. Then there are two possibilities:

(a) $\u=\v_{\ell-1}^-$. Then $\v_{\ell'}^*=\v_\ell^*=\v_{\ell-1}^*$, so $(\v_{\ell-1}^*)^\f \succ \v_k^-(\al_1\dots\al_{j_k}^-)^\f$ by \eqref{eq:v_l-v_k-inequality}. This means $\overline{t}_{\ell-1}>\underline{t}_k$, so by the induction hypothesis, $\v_{\ell-1}$ begins with $\v_k^-(\al_1\dots\al_{j_k}^-)^n\al_1\dots\al_{j_k}$ for some $n\geq 0$, since $\v_{\ell-1}=\u^+\neq \v_k$. Now either (i) $\v_{\ell-1}=\v_k^-(\al_1\dots\al_{j_k}^-)^n\al_1\dots\al_{j_k}$, in which case {\eqref{eq:v-ell-form} implies}
\[
\v_\ell=\v_{\ell-1}^-(\al_1\dots\al_{j_{\ell}}^-)^r\al_1\dots\al_{j_{\ell}}=\v_k^-(\al_1\dots\al_{j_k}^-)^{n+1}(\al_1\dots\al_{j_{\ell}}^-)^r\al_1\dots\al_{j_{\ell}},
\]
and so $\v_\ell$ begins with $\v_k^-(\al_1\dots\al_{j_k}^-)^{n+1}\al_1\dots\al_{j_k}$ because $j_\ell>j_k$; or (ii) $\v_{\ell-1}$ properly extends $\v_k^-(\al_1\dots\al_{j_k}^-)^n\al_1\dots\al_{j_k}$, in which case $\u$, and hence $\v_\ell$, begins with ${\v_k^-}(\al_1\dots\al_{j_k}^-)^n\al_1\dots\al_{j_k}$.

(b) $\u\neq\v_{\ell-1}^-$. Then $\v_{\ell'}^*=\v_\ell^*=\u^+$, and hence,
\[
(\u^+)^\f\succ \v_k^-(\al_1\dots\al_{j_k}^-)^\f.
\]
This implies that $(\u^+)^\f$ begins with $\v_k^-(\al_1\dots\al_{j_k}^-)^{q}\al_1\dots\al_{j_k}$ for some ${q}\geq 0$. The same argument as in Case 1 then shows that $\u^+$ {begins with} $\v_k^-(\al_1\dots\al_{j_k}^-)^{q}\al_1\dots\al_{j_k}$. If $\u^+$ properly extends this word, then so does $\u$ and hence so does $\v_\ell$.
Otherwise, {\eqref{eq:v-ell-form} yields
\[
\v_\ell=\u(\al_1\dots\al_{j_{\ell}}^-)^r\al_1\dots\al_{j_\ell}=\v_k^-(\al_1\dots\al_{j_k}^-)^{q+1}(\al_1\dots\al_{j_{\ell}}^-)^r\al_1\dots\al_{j_\ell},
\]
}
so that $\v_\ell$ {begins with} $\v_k^-(\al_1\dots\al_{j_k}^-)^{{q}+1}\al_1\dots\al_{j_k}$, because $j_\ell>j_k$.

In both case (a) and case (b) we have shown that $\v_\ell$ {begins with} $\v_k^-(\al_1\dots\al_{j_k}^-)^n\al_1\dots\al_{j_k}$ for some $n\geq 0$. {Since $\v_{\ell'}=\v_\ell$}, this completes the proof.
\end{proof}

\section{Proof of Theorem \ref{thm:basic-interval-transitivity}} \label{sec:basic-interval-proof}

We are now ready to prove Theorem \ref{thm:basic-interval-transitivity}. The theorem involves three separate statements, which we prove in Propositions \ref{prop:non-transitivity-window}, \ref{prop:transitive-in-between} and \ref{prop:entropy-constant-in-window}, respectively. Our notation and assumptions are the same as in the previous section. {In particular, we assume $\cs\in\La$ and $\beta\in(\beta_\ell^\cs,\beta_*^\cs)$.}

\begin{lemma} \label{v_k-star-less-than-s}
For each $k\geq 1$, $\v_k^*\prec \cs$.
\end{lemma}

\begin{proof}
First, take $k=1$. If $\v_1$ is $\beta$-Lyndon, then $\v_1^*=\v_1\prec \cs$, using \eqref{eq:v1-small}. If $\v_1$ is not $\beta$-Lyndon, then $\v_1=\u(\al_1\dots\al_{j_1}^-)^r\al_1\dots\al_{j_1}$ for some word $\u$ and integer $r\geq 0$ by Lemma \ref{lem:v_k-star}, and $\v_1^*=\u^+$. This implies that $\v_1^*\neq \cs$, for else we would have
\[
\v_1=\u(\al_1\dots\al_{j_1}^-)^r\al_1\dots\al_{j_1}=\cs^-\L(\cs)^r\L(\cs)^+,
\]
contradicting \eqref{eq:v1-small}. But $\v_1^*\lle \cs$ by Lemma \ref{lem:existence-of-v-star}, since $\v_1\prec\cs$. Therefore, $\v_1^*\prec \cs$.

For $k\geq 2$, the statement now follows inductively by Lemma \ref{lem:v_k-decreasing}.
\end{proof}

\begin{definition}
We call an interval $I_k\in\II$ a {\em non-transitivity window} if it is a maximal element of $\II$ with respect to set inclusion; that is, if $I_k$ is not contained in any other interval $I_\ell\in \II$. We denote by $\II_{\max}$ the collection of all non-transitivity windows.
\end{definition}

By Proposition \ref{prop:non-transitivity-intervals}, the collection $\II_{\max}$ is pairwise disjoint. The name ``non-transitivity window" is made clear by the next proposition.

\begin{proposition} \label{prop:non-transitivity-window}
Let $I_k\in\II_{\max}$ be a non-transitivity window. Then for any $\beta$-Lyndon interval $[t_L,t_R]$ with $t_R\in I_k$, the subshift $\Kt_\beta(t_R)$ is not transitive.
\end{proposition}

\begin{proof}
Suppose $t_R\in I_k$. We first assume $I_k$ is open on the right, so
\begin{equation} \label{eq:greedy-expansion-sandwich}
\v_k^-(\al_1\dots\al_{j_k}^-)^\f=b(\underline{t}_k,\beta)\lle b(t_R,\beta)\prec b(\overline{t}_k,\beta)=(\v_k^*)^\f.
\end{equation}
Observe that, by definition of $\v_k^*$, there are no $\beta$-Lyndon words $\w$ such that $\v_k^\f\prec \w^\f\prec (\v_k^*)^\f$, and so we actually have $b(t_R,\beta)\prec \v_k^\f$. (Strict inequality, because if $\v_k^\f\prec (\v_k^*)^\f$, then $\v_k$ is not $\beta$-Lyndon.) But then, since $\v_k$ is Lyndon, $b(t_R,\beta)$ must begin with $\v_k^-$.

We claim that $\al_1\dots\al_{j_k}\in \cL(\Kt_\beta(t_R))$. To see this, consider two cases:
\begin{enumerate}[(i)]
\item If $t_R>\underline{t}_k$, then $b(t_R,\beta)$ begins with $\v_k^-(\al_1\dots\al_{j_k}^-)^r\al_1\dots\al_{j_k}$ for some $r\geq 0$, and since $b(t_R,\beta)\in\Kt_\beta(t_R)$, this implies that $\al_1\dots\al_{j_k}\in \cL(\Kt_\beta(t_R))$.
\item If $t_R=\underline{t}_k$, then $b(t_R,\beta)=\v_k^-(\al_1\dots\al_{j_k}^-)^\f$, and in this case the sequence
\[
\al_1\dots\al_{j_k}\v_k^-(\al_1\dots\al_{j_k}^-)^\f {=\al_1\dots\al_{j_k+l_k}^-(\al_1\dots\al_{j_k}^-)^\f}
\]
lies in $\Kt_\beta(t_R)$ {by Lemma \ref{lem:v_k-upper-bound}}, so again $\al_1\dots\al_{j_k}\in \cL(\Kt_\beta(t_R))$.
\end{enumerate}

From the construction of $\v_k$, \eqref{eq:greedy-expansion-sandwich} and the definition of $\Kt_\beta(t_R)$, it follows that any sequence in $\Kt_\beta(t_R)$ beginning with $\al_1\dots\al_{j_k}$ must be of the form
\begin{equation} \label{eq:repeating-pattern}
\al_1\dots\al_{j_k} \v_k^{p_1}\v_k^-(\al_1\dots\al_{j_k}^-)^{q_1}\,\al_1\dots\al_{j_k} \v_k^{p_2}\v_k^-(\al_1\dots\al_{j_k}^-)^{q_2}\dots,
\end{equation}
where $0\leq p_i,q_i\leq\infty$ for all $i$. If $k>1$, then no such sequence ends in $\cs^\f$, because $\v_k^\f\prec\cs^\f$ by Lemma \ref{v_k-star-less-than-s} and $\al_1\dots\al_{j_k}^-$ begins with $\L(\cs)^+$. Since $\cs^\f$ is always a valid sequence in $\Kt_\beta(t_R)$, we see that $\Kt_\beta(t_R)$ is not transitive in this case. 

Suppose now that $k=1$. Then we have to be a bit more careful since $\al_1\dots\al_{j_1}^-=\L(\cs)$. If $t_R>\underline{t}_1$, then we cannot have $q_i=\infty$ for any $i$ in \eqref{eq:repeating-pattern}, so again the legal sequences of the form \eqref{eq:repeating-pattern} cannot end in $\cs^\f$, and $\Kt_\beta(t_R)$ is not transitive. Suppose $t_R=\underline{t}_1$, so $b(t_R,\beta)=\v_1^-(\al_1\dots\al_{j_1}^-)^\f$. Since $b(t_R,\beta)$ is purely periodic, this implies that either $\v_1=\al_{m+1}\dots\al_{j_1}$ for some $m<j_1$, or else
\[
\v_1=\al_{m+1}\dots\al_{j_1}^-(\al_1\dots\al_{j_1}^-)^r\al_1\dots \al_{j_1}
\]
for some $m<j_1$ and $r\geq 0$. However, it can be seen from the inequalities in \eqref{eq:v_1-upper-bound} that neither of these is possible. Thus, we must have $t_R>\underline{t}_1$, and we are done.

Finally, we consider the case when $I_k$ is closed. Recall that this means that $\II=(I_1,\dots,I_k)$ is finite, $\v_k$ is $\beta$-Lyndon, and $\si^{j_k}(\al(\beta))=\v_k^\f$. If $t_R<\overline{t}_k$, the proof goes as before, so assume $t_R=\overline{t}_k$. Then $b(t_R,\beta)=\v_k^\f$. Now $\Kt_\beta(t_R)$ is not transitive because $\al_1\dots\al_{j_k}$ is a legal word that can only be followed by the sequence $\v_k^\f$.
\end{proof}

\begin{remark}
In the last case in the above proof, when $I_k$ is closed and $t=\overline{t}_k$, we showed that $\Kt_\beta(t)$ is not transitive. However, in this special case the set
\[
\mathcal{K}_\beta(t)=\{\z\in{A_\beta}^\N: b(t,\beta)\lle \si^n(\z)\prec\al(\beta)\ \forall\,n\geq 0\}
\]
is in fact itself a transitive subshift, as is easy to see. So the non-transitivity of $\Kt_\beta(t)$ in this case is somewhat artificial, caused by enlarging the set $\mathcal{K}_\beta(t)$ more than what is needed to obtain a subshift.
\end{remark}

{
To prove the next proposition, we first present a lemma.

\begin{lemma} \label{lem:it-does-not-come-back}
Assume $\al(\beta)$ is not periodic, and $\II=(I_1,I_2,\dots,I_{\ell})$ is finite. Let $n_0:=|\L(\cs)^+\v_1\v_2\dots\v_{\ell}|=j_{\ell}+|\v_{\ell}|$. (If $\II=\emptyset$, we set $n_0:=|\cs|$.)
\begin{enumerate}[(i)]
\item There is an infinite sequence $(n_i)_{i\in\N}$ with $n_0<n_1<n_2\dots$ such that $\al_{n_0+1}\dots\al_{n_i}$ is Lyndon for each $i$. 
\item The tail $\si^{n_0}(\al(\beta))$ does not contain the word $\al_1\dots\al_{n_0}$. 
\end{enumerate}
\end{lemma}

\begin{proof}
If $\II\neq\emptyset$, then $\si^{n_0}(\al(\beta))\lle\v_{\ell}^\f$. If $\II=\emptyset$, then $\si^{n_0}(\al(\beta))\lle \cs^-\L(\cs)^\f$. Either way, $\si^{n_0}(\al(\beta))$ begins with {a digit $d<M_\beta$. If $d=0$, then, since $\si^{n_0}(\al(\beta))\neq 0^\f$, there is an index $n_1>n_0$ such that $\al_{n_0+1}\dots\al_{n_1}$ is Lyndon. Otherwise, the one-digit word $d$ is Lyndon, and we set $n_1:=n_0+1$.
In either case,} $\si^{n_1}(\al(\beta))\succ (\al_{n_0+1}\dots\al_{n_1})^\f$, or else the word $\al_{n_0+1}\dots\al_{n_1}$ would have become $\v_{\ell+1}$ and generated an additional interval $I_{\ell+1}$ in $\II$. Hence, there is an index $n_2>n_1$ such that $\al_{n_0+1}\dots\al_{n_2}$ is also Lyndon. Repeating this argument yields (i).

To prove (ii), we write $(c_i):=\si^{n_0}(\al(\beta))$ and suppose by way of contradiction that 
\begin{equation} \label{eq:reoccurrence}
\al_{n_0+m+1}\dots\al_{2n_0+m}=c_{m+1}\dots c_{m+n_0}=\al_1\dots\al_{n_0}, 
\end{equation}
where we choose $m$ to be minimal, i.e. $c_{i+1}\dots c_{i+n_0}\neq \al_1\dots\al_{n_0}$ for all $0\leq i<m$. We first show that $c_1\dots c_{m+n_0}$ is Lyndon. This means verifying that
\begin{equation} \label{eq:c-Lyndon}
c_{i+1}\dots c_{m+n_0}\succ c_1\dots c_{m+n_0-i} \qquad \forall\,1\leq i<m+n_0.
\end{equation}
Observe that we always have weak inequality, since $c_1\dots c_{m+n_0}$ is a prefix of the Lyndon word $\al_{n_0+1}\dots\al_{n_p}$ for a sufficiently large $p$. If $i\leq m$, then equality in \eqref{eq:c-Lyndon} would contradict the minimality of $m$, since $c_{i+1}\dots c_{m+n_0}$ contains the word $\al_1\dots\al_{n_0}$. So it remains to deal with the case $i>m$.
Recall that
\begin{equation} \label{eq:beginning-copied}
c_{m+1}\dots c_{m+n_0}=\al_1\dots\al_{n_0}=\L(\cs)^+\v_1\dots\v_{\ell}.
\end{equation}
If $m+j_1\leq i<m+n_0$, then $c_{i+1}\dots c_{m+n_0}=\al_{i-m+1}\dots\al_{n_0}$ is part of the block $\v_1\dots\v_{\ell}$. 
Recall that we already have weak inequality: $c_1\dots c_{m+n_0-i}\lle c_{i+1}\dots c_{m+n_0}$. If $c_{i+1}\dots c_{m+n_0}$ begins in the middle of a word $\v_k$, then we immediately have strict inequality by the Lyndon property of $\v_k$, giving \eqref{eq:c-Lyndon}. Otherwise, $c_{i+1}\dots c_{m+n_0}=\v_k\dots\v_{\ell}$ for some $k$ in view of \eqref{eq:beginning-copied}. Suppose $c_{i+1}\dots c_{m+n_0}=c_1\dots c_{m+n_0-i}$. Then $\al_{n_0+1}\dots\al_{m+2n_0-i}=\v_k\dots\v_{\ell}$, so $\al(\beta)$ begins with $\L(\s^+)\v_1\dots \v_{\ell}\v_k\dots\v_{\ell}$, which can only happen if $\v_k=\v_{k+1}=\dots=\v_{\ell}$. Since $I_{\ell}$ is the last interval in $\mathcal{I}$, this implies $\si^{m+2n_0-l}(\al(\beta))\succ \v_{\ell}^\f$. (Else there would have been a word $\v_{\ell+1}$ and corresponding interval $I_{\ell+1}\in\II$.) But then $\si^{n_0}(\al(\beta))\succ \v_{\ell}^\f$ as well, contradicting that $I_{\ell}\in\mathcal{I}$.


Finally, for $m<i<m+j_1$ we have
\begin{align*}
c_{i+1}\dots c_{m+j_1}&=\al_{i-m+1}\dots \al_{j_1}\succ s_1\dots s_{j_1-(i-m)}\\
&\lge (\v_{\ell}^\f)_{1:j_1-(i-m)}\lge c_1\dots c_{j_1-(i-m)},
\end{align*}
where the first inequality follows since $\al_{i-m+1}\dots \al_{j_1}$ is a suffix of $\L(\cs)^+$, the second inequality follows by Lemma \ref{v_k-star-less-than-s}, and the last inequality follows from the definition of $\v_{\ell}$.

(Note: If the collection $\II$ is empty, then $n_0=j_1$ and we obtain directly the inequality $s_1\dots s_{j_1-(i-m)}\lge c_1\dots c_{j_1-(i-m)}$, because $\si^{j_1}(\al(\beta))\prec \cs^-\L(\cs)^\f$.)


Having verified \eqref{eq:c-Lyndon}, we conclude that $c_1\dots c_{m+n_0}$ is Lyndon. But then $\si^{2n_0+m}(\al(\beta))\succ (c_1\dots c_{n_0+m})^\f$ {(or else $I_{\ell}$ would not have been the last interval in $\II$)}. Now \eqref{eq:reoccurrence} gives
\[
\si^{2n_0+m}(\al(\beta))\succ (\al_{n_0+1}\dots \al_{2n_0+m})^\f=\al_{n_0+1}\dots\al_{n_0+m}(\al_1\dots\al_{n_0+m})^\f.
\]
There are two possibilities:
\begin{itemize}
\item $\al_{2n_0+m+1}\dots\al_{2n_0+2m}=\al_{n_0+1}\dots\al_{n_0+m}$, in which case we obtain $\si^{2n_0+2m}(\al(\beta))\succ (\al_1\dots\al_{n_0+m})^\f\lge \al(\beta)$; or 
\item $\al_{2n_0+m+1}\dots\al_{2n_0+2m}\succ\al_{n_0+1}\dots\al_{n_0+m}$, in which case $\si^{2n_0+m}(\al(\beta))\succ\si^{n_0}(\al(\beta))$, and so $\si^{n_0+m}(\al(\beta))\succ\al(\beta)$ because we can cancel the first $n_0$ digits in view of \eqref{eq:reoccurrence}. 
\end{itemize}
In either case, we obtain a contradiction with Lemma \ref{lem:quasi-greedy expansion-alpha-q}.
\end{proof}
}

\begin{proposition} \label{prop:transitive-in-between}
Assume $\cs=\s\in\La_1={F_e}$. Let $t_R\in(0,\tau(\beta))$ be a right endpoint of a $\beta$-Lyndon interval, and suppose $t_R$ does not lie in any non-transitivity window. Then $\Kt_\beta(t_R)$ is transitive.
\end{proposition}

(The statement of this proposition fails when $\cs\in\La_k$ with $k\geq 2$, as we will see in Section \ref{sec:higher-order-basic}.)

\begin{proof}
Let $\w$ be the $\beta$-Lyndon word such that $b(t_R,\beta)=\w^\f$.
We enumerate $\II_{\max}$ as $\II_{\max}=(I_1=I_{k_1},I_{k_2},\dots)$, where $k_1<k_2<\dots$, or as $\II_{\max}=(I_1=I_{k_1},I_{k_2},\dots,I_{k_N})$, where $k_1<k_2<\dots<k_N$, depending on whether $\II_{\max}$ is infinite or finite. Recall that the intervals $I_{k_1},I_{k_2},\dots$ proceed from right to left. There are three cases to consider:
\begin{enumerate}[(i)]
\item $t_R\geq \overline{t}_1$;
\item $\overline{t}_{k_{i+1}}\leq t_R<\underline{t}_{k_i}$ for some $i$;
\item $t_R<\underline{t}_k$ for all $k$.
\end{enumerate}
The argument differs somewhat for each of the three cases.

\medskip
{\em Case 1.} $t_R\geq \overline{t}_1$. Then $\w^\f\lge (\v_1^*)^\f\lge \al_{j_1+1}\al_{j_1+2}\dots$.
In fact, we can show that
\begin{equation} \label{eq:w-dominates-tail-of-alpha}
\w^\f\succ \al_{j_1+1}\al_{j_1+2}\dots,
\end{equation}
by considering the following cases:
\begin{itemize}
\item If $t_R>\overline{t}_1$, then $\w^\f\succ (\v_1^*)^\f$;
\item If $t_R=\overline{t}_1$ and $\v_1$ is {\em not} $\beta$-Lyndon, then $(\v_1^*)^\f\succ \v_1^\f\lge \al_{j_1+1}\al_{j_1+2}\dots$ by the definition of $\v_1$;
\item If $t_R=\overline{t}_1$ and $\v_1$ {\em is} $\beta$-Lyndon, then, since we assumed that $t_R\not\in I_1$, {it follows that the interval $I_1$ is half open and therefore, by the way we constructed the intervals $(I_k)$}, $\si^{j_1}(\al(\beta))\neq \v_1^\f$, which implies $(\v_1^*)^\f\lge\v_1^\f\succ \al_{j_1+1}\al_{j_1+2}\dots$.
\end{itemize}
By \eqref{eq:w-dominates-tail-of-alpha}, we have
\[
\Kt_\beta(t_R)=\{\z: \w^\f\lle \si^n(\z)\lle (\al_1\dots\al_{j_1}^-)^\f\ \forall\,n\geq 0\}.
\]
Thus, the word $\al_1\dots\al_{j_1}$ cannot occur in any sequence $\z\in\Kt_\beta(t_R)$; in particular, it cannot occur in $\w^\f$.

Let $\u\in\cL(\Kt_\beta(t_R))$ and $\z\in\Kt_\beta(t_R)$ be given. By Lemma \ref{lem:begin-with-s}, we can extend $\z$ to the left to a sequence $\z'\in\Kt_\beta(t_R)$ beginning with $\L(\s)$. {(This is where we use that $\s\in F_e$.)} 
Thus, without loss of generality, we will assume that $\z$ itself begins with $\L(\s)$.

If $\u$ does not end in $\al_1\dots\al_m$ for any $m$, set $\u':=\u\L(\s)$. Otherwise, let $m$ be the largest integer such that $\u$ ends in $\al_1\dots\al_m$, say $\u=u_1\dots u_l \al_1\dots\al_m$. Then $m<j_1$, and {we extend $\u$ to} $\u':=u_1\dots u_l\al_1\dots\al_{j_1}^-=u_1\dots u_l\L(\s)$.

Observe that $\w^\f\prec \s^-\L(\s)^\f$, so there exists $N_1\in\N$ such that $\w^\f\prec \s^-\L(\s)^{N_1}0^\f$. Furthermore, $\si^n(\w^\f)\lle \L(\s)^\f$ for all $n\geq 0$ {because $\w$ is $\beta$-Lyndon and does not contain the word $\al_1\dots\al_{j_1}=\L(\s)^+$.} In fact, the inequality is strict because $\w$ and $\s$ are both Lyndon, and $\w\neq \s$. Thus, since $\w^\f$ is periodic, there exists $N_2\in\N$ such that
\[
\si^n(\w^\f)\prec \L(\s)^{N_2}0^\f \qquad \forall\,n\geq 0.
\]
Now let $N_3:=\max\{N_1,N_2\}$, and set $\v:=\L(\s)^{N_3}$. Then $\u'\v\z\in\Kt_\beta(t_R)$. Therefore, $\Kt_\beta(t_R)$ is transitive.

\medskip
{\em Case 2.} $\overline{t}_{k_{i+1}}\leq t_R<\underline{t}_{k_i}$ for some $i$. For ease of notation, put $k:=k_i$ and $\ell:=k_{i+1}$, so we have
\[
\overline{t}_\ell\leq t_R<\underline{t}_\nu \qquad\forall\, 1\leq\nu<\ell,
\]
since $\underline{t}_\nu\geq \underline{t}_k$ for $k\leq\nu<\ell$. {Note that the intervals $I_k$ and $I_\ell$ are disjoint, and if $\ell\geq k+2$, then $I_\nu\subseteq I_k$ for all $k<\nu<\ell$.}
By the same reasoning as in Case 1, we have that
\[
\w^\f\succ \al_{j_\ell+1}\al_{j_\ell+2}\dots,
\]
and so
\[
\Kt_\beta(t_R)=\{\z: \w^\f\lle \si^n(\z)\lle (\al_1\dots\al_{j_\ell}^-)^\f\ \forall\,n\geq 0\}.
\]
Thus, the word $\al_1\dots\al_{j_\ell}$ never occurs in any sequence in $\Kt_\beta(t_R)$. On the other hand, for each $1\leq\nu<\ell$,
\[
\w^\f=b(t_R,\beta)\prec b(\underline{t}_\nu,\beta)=\v_\nu^-(\al_1\dots\al_{j_\nu}^-)^\f,
\]
so there is an integer $r_\nu\geq 0$ such that
\begin{equation} \label{eq:strictly-clearing-the-bar}
\w^\f\prec \v_{\nu}^-(\al_1\dots\al_{j_\nu}^-)^{r_\nu} 0^\f.
\end{equation}
Let $\u\in\cL(\Kt_\beta(t_R))$ and $\z\in\Kt_\beta(t_R)$ be given. As in Case 1, we may assume that $\z$ begins with the word $\L(\s)$. 
If no suffix of $\u$ is a prefix of $\al(\beta)$, set $m=0$. Otherwise, let $m$ be the largest integer such that $\u$ ends in $\al_1\dots\al_m$. In either case {$m<j_\ell=j_{k_{i+1}}$}, and we define 
\[
{\u':=\u\al_{m+1}\dots\al_{j_{k_{i+1}}}^-(\al_1\dots\al_{j_{k_i}}^-)^{r_{k_i}}(\al_1\dots\al_{j_{k_{i-1}}}^-)^{r_{k_{i-1}}}\dots (\al_1\dots\al_{j_1}^-)^{r_1}.}
\]
{(We only use those $r_\nu$'s for which $I_\nu$ is a maximal interval in $\mathcal{I}$; i.e. $I_\nu\in\II_{\max}$.)}
It is not difficult to check using \eqref{eq:strictly-clearing-the-bar} that $\u'\z\in\Kt_\beta(t_R)$. (The idea is that {$\al_1\dots\al_{j_{k_{i+1}}}^-$ ends in $\v_{k_i}^-$. So we append a high enough power of $\al_1\dots\al_{j_{k_i}}^-$} to it to ensure that $\si^n(\u'\z)\lge \w^\f$ for all $n$, and repeat until we reach $\nu=1$.) Hence, $\Kt_\beta(t_R)$ is transitive.

\medskip
{\em Case 3.} $t_R<\underline{t}_k$ for all $k$. This case is the most involved.

(a) Assume first that $\II$ is infinite. Then as in Case 2 there exist integers $r_1,r_2,\dots$ such that
\[
\w^\f\prec \v_k^-(\al_1\dots\al_{j_k}^-)^{r_k} 0^\f, \qquad k=1,2,\dots.
\]
Let $\u\in\cL(\Kt_\beta(t_R))$ and $\z\in\Kt_\beta(t_R)$ be given.
If $\u$ does not end in $\al_1\dots\al_m$ for any $m$, set $\u':=\u$. Otherwise, let $m$ be the largest integer such that $\u$ ends in $\al_1\dots\al_m$. Here we do not have an upper bound for $m$, however there is an index $k$ such that $j_{k-1}\leq m<j_{k}$, since $j_k\to\f$ as $k\to\f$. {Let $1=k_1<k_2<\dots<k_i<k$ be the indices such that $I_{k_t}\in\II_{\max}$, $t=1,\dots,i$.} We can then extend $\u$ to a word $\u'\in \cL(\Kt_\beta(t_R))$ ending in the block
\[
\al_1\dots\al_{j_k}^-{(\al_1\dots\al_{j_{k_i}}^-)^{r_{k_i}}(\al_1\dots\al_{j_{k_{i-1}}}^-)^{r_{k_{i-1}}}\dots (\al_1\dots\al_{j_1}^-)^{r_1}}.
\]
Similarly, as before, we may assume $\z$ begins with $\L(\s)$. Then $\u'\z'\in\Kt_\beta(t_R)$ just in Case 2.

(b) Assume next that $\II$ is finite: $\II=(I_1,I_2,\dots,I_\ell)$. This case is more difficult because $\u$ can end in an arbitrarily long prefix of $\al(\beta)$, but there are only finitely many $j_k$'s, so the construction of case (a) above needs to be modified. Set $n_0:=|\L(\s)^+\v_1\v_2\dots\v_{\ell}|=j_{\ell}+|\v_{\ell}|$.

(i) Suppose first that $\al(\beta)$ is not periodic. By Lemma \ref{lem:it-does-not-come-back} (i), there is an infinite, strictly increasing sequence $(n_i)_{n\in\N}$ such that $\al_{n_0+1}\dots\al_{n_i}$ is Lyndon for each $i$. 
It follows that the tail $\si^{n_0}(\al(\beta))$ cannot be periodic. 
If now $\w^\f\succ \si^{n_0}(\al(\beta))$, then
\[
\Kt_\beta(t_R)=\{\z: \w^\f\lle \si^n(\z)\lle (\al_1\dots\al_{n_0}^-)^\f\ \forall\,n\geq 0\},
\]
and the proof proceeds essentially as in Case 2. Otherwise, since $\si^{n_0}(\al(\beta))$ is not periodic, there is a sufficiently large index $p$ such that {$n_p>|\w|$ and}
\begin{equation} \label{eq:finite-cutoff}
\w^\f\prec \al_{n_0+1}\dots \al_{n_{p}}^-0^\f.
\end{equation}

From Lemma \ref{lem:it-does-not-come-back} (ii), it follows that
\begin{equation} \label{eq:all-tails-below-prefix}
\si^n(\al(\beta))\lle (\al_1\dots\al_{n_0}^-)^\f \qquad\forall\,n\geq n_0.
\end{equation}
Now let $\u\in\cL(\Kt_\beta(t_R))$ and $\z\in\Kt_\beta(t_R)$ be given. As in Case 2, we can extend $\z$ to a sequence $\z'\in\Kt_\beta(t_R)$ beginning with $\L(\s)$. We extend $\u$ first to a sequence $\u'$ ending in $\al_1\dots\al_{n_{p}}^-$, where ${p}$ satisfies \eqref{eq:finite-cutoff}, in the usual way by pasting $\u$ and $\al_1\dots\al_{n_{p}}^-$ together along their longest overlap. If $\u'$ does not end in a prefix of $\w$, then $\u'\z'\in\Kt_\beta(t_R)$ and we are done.

Otherwise, let $m$ be the largest integer such that $\u'$ ends in $w_1\dots w_m$, where we write $(w_i):=\w^\f$. {Note that $m\leq |\w|<n_p$.} Then by \eqref{eq:finite-cutoff} and the Lyndon property of $\al_{n_0+1}\dots\al_{n_{p}}$,
\[
w_1\dots w_m\lle \al_{n_0+1}\dots\al_{n_0+m}\lle \al_{n_{p}-m+1}\dots\al_{n_{p}}^-=w_1\dots w_m,
\]
hence $w_1\dots w_m=\al_{n_0+1}\dots\al_{n_0+m}$. By \eqref{eq:finite-cutoff} and \eqref{eq:all-tails-below-prefix}, it follows that
\[
w_{m+1}w_{m+2}\dots \prec \al_{n_0+m+1}\dots\al_{n_{p}}^-0^\f \prec (\al_1\dots\al_{n_0}^-)^\f,
\]
so there is some integer $r_*\geq 1$ such that $w_{m+1}w_{m+2}\dots\prec (\al_1\dots\al_{n_0}^-)^{r_*} 0^\f$.
Recall again that $\al_1\dots\al_{n_0}^-=\L(\s)^+\v_1\dots\v_{\ell}^-$. As in Case 2 there exist integers $r_1,\dots,r_{\ell}$ such that
\begin{equation} \label{eq:large-exponents}
\w^\f\prec \v_k^-(\al_1\dots\al_{j_k}^-)^{r_k} 0^\f, \qquad k=1,\dots,\ell.
\end{equation}
{Recall that $1=k_1<k_2<\dots<k_N\leq \ell$ are the indices such that $I_{k_i}\in\II_{\max}$, $i=1,\dots,N$.}
We now extend $\u'$ further by setting
\[
\u'':=\u'(\al_1\dots\al_{n_0}^-)^{r_*}{(\al_1\dots\al_{j_{k_N}}^-)^{r_{k_N}}(\al_1\dots\al_{j_{k_{N-1}}}^-)^{r_{k_{N-1}}}\dots (\al_1\dots\al_{j_1}^-)^{r_1}}.
\]
Then $\u''\z'\in\Kt_\beta(t_R)$. Hence, $\Kt_\beta(t_R)$ is transitive.

(ii) Suppose next that $\al(\beta)$ is periodic. Then
\[
\al(\beta)=(\al_1\dots\al_{n_0})^\f=(\L(\s)^+\v_1\dots\v_{\ell}^-)^\f.
\]
Let $\u\in\cL(\Kt_\beta(t_R))$. If $\u$ does not end in a prefix of $\al(\beta)$, we set $\u':=\u\al_1\dots\al_{n_0}$. Otherwise, there is some $r\geq 0$ and $0\leq m<n_0$ such that $\u$ ends in $(\al_1\dots\al_{n_0})^r\al_1\dots\al_m$. In that case we extend $\u$ slightly further to a word $\u'$ ending in $(\al_1\dots\al_{n_0})^{r+1}$. Observe that $\u'$ then ends in $\L(\s)^+\v_1\dots\v_{\ell}^-$. From here we proceed as above: We choose exponents $r_1,\dots,r_{\ell}$ satisfying \eqref{eq:large-exponents}, and set
\[
\u'':=\u'{(\al_1\dots\al_{j_{k_N}}^-)^{r_{k_N}}\dots (\al_1\dots\al_{j_{k_{N-1}}}^-)^{r_{k_{N-1}}}\dots (\al_1\dots\al_{j_1}^-)^{r_1}}.
\]
This word $\u''$ can then be connected directly to any sequence $\z'\in\Kt_\beta(t_R)$ beginning with $\L(\s)$.
\end{proof}

\begin{proposition} \label{prop:entropy-constant-in-window}
Let $I_k=[\underline{t}_k,\overline{t}_k)\in\II_{\max}$ be a non-transitivity window. Then $h(\Kt_\beta(t))$ is constant on $I_k$.
\end{proposition}

\begin{proof}
This will follow from Lemmas \ref{lem:EBLI-constant-entropy} and \ref{lem:window-is-EBLI} in Section \ref{sec:EBLI}.
\end{proof}

\begin{proof}[Proof of Theorem \ref{thm:basic-interval-transitivity}]
The theorem follows from Propositions \ref{prop:non-transitivity-window}, \ref{prop:transitive-in-between} and \ref{prop:entropy-constant-in-window}.
\end{proof}

The work we did in this section and the last, also has an important consequence for the bifurcation sets $\EE_\beta$ and $\BB_\beta$:

\begin{theorem} \label{thm:bifurcation-in-basic-interval}
Let $\beta\in(\beta_\ell^{\cs},\beta_*^{\cs})$, where $\cs\in\La$. If $I_k$ is a non-transitivity window, then $\dim_H(\EE_\beta\cap I_k)>0$. In particular, if the collection $\II$ is nonempty, then $\dim_H(\EE_\beta\backslash\BB_\beta)>0$.
\end{theorem}

\begin{proof}
Let $I_k=[\underline{t}_k,\overline{t}_k)$ be a non-transitivity window, so
\[
b(\underline{t}_k,\beta)=\v_k^-(\al_1\dots\al_{j_k}^-)^\f, \qquad b(\overline{t}_k,\beta)=(\v_k^*)^\f.
\]
Writing $B_k:=\al_1\dots\al_{j_k}$ for brevity, consider points $t$ with greedy expansion
\[
b(t,\beta)=\v_k^-(B_k^-)^{r_1}B_k \v_k^-(B_k^-)^{r_2}B_k \v_k^-(B_k^-)^{r_3}B_k\dots.
\]
Note that this expression is indeed a greedy $\beta$-expansion since
\[
B_k\v_k^-\prec B_k\v_k=\al_1\dots\al_{j_k}\al_{j_k+1}\dots\al_{j_k+l_k}.
\]
Furthermore, observe that $\si^n(b(t,\beta))\lge b(t,\beta)$ for all $n\geq 0$ if and only if $(r_1,r_2,\dots)\lge (r_n,r_{n+1},\dots)$ for all $n\geq 1$. Taking $r_1=3$ and $r_i\in\{1,2\}$ arbitrarily for $i\geq 2$ we see that there are uncountably many points of $\EE_\beta$ in $I_k$. In fact, the set
\[
C:=\{\v_k^-(B_k^-)^3 B_k\v_k^-(B_k^-)^{r_2}B_k\v_k^-(B_k^-)^{r_3}B_k\dots: r_i\in\{1,2\}\ \forall\,i\geq 2\}
\]
has positive entropy, hence positive Hausdorff dimension in $\Sigma_\beta$. Therefore, by \eqref{eq:Hausdorff-dimension-preserved},
\[
\dim_H(\EE_\beta\cap I_k)\geq \dim_H \pi_\beta(C)>0.
\]
The second statement follows from the first and Proposition \ref{prop:entropy-constant-in-window}, which implies that there are no points of $\BB_\beta$ in $I_k$.
\end{proof}

\section{Higher order basic intervals} \label{sec:higher-order-basic}

In this section we generalize Theorem \ref{thm:basic-interval-transitivity} to the higher order basic interval $[\beta_\ell^{\cs},\beta_*^{\cs}]$, where $\cs\in\La$. We do this by combining the construction from Section \ref{sec:basic-interiors} with the renormalization method of Section \ref{sec:relative-exceptional}.

Recall that $\mathcal{T}_R(\beta)$ is the set of all right endpoints of $\beta$-Lyndon intervals in $[0,\tau(\beta)]$. The following theorem establishes the final case of Theorem \ref{thm:general-transitivity}.

\begin{theorem} \label{thm:general-basic-intervals}
Let $[\beta_\ell,\beta_*]$ be a basic interval generated by a word $\cs=\r_1\bullet\dots\bullet \r_n\in\La$, where {$\r_1\in F_e$ and $\r_i\in F^*$ for $i=2,\dots,n$}, and let $\beta\in(\beta_\ell,\beta_*)$. Then there is a (finite or infinite, possibly empty) collection $\II$ of intervals such that
\begin{enumerate}[{\rm(i)}]
\item For any $t_R\in\mathcal{T}_R(\beta)$, $\Kt_\beta(t_R)$ is transitive if and only if $b(t_R,\beta)\prec \r_1^-\L(\r_1)^\f$ and $t_R\not\in\bigcup_{I\in\II}I$;
\item For any $t_R\in\mathcal{T}_R(\beta)\backslash \bigcup_{I\in\II}I$, $\Kt_\beta(t_R)$ has a transitive subshift $\mathcal{K}_\beta'(t_R)$ of full entropy and full Hausdorff dimension that contains the sequence $b(t_R,\beta)$;
\item The collection
\[
\{\mathcal{K}_\beta'(t_R): t_R\in\mathcal{T}_R(\beta)\backslash \textstyle{\bigcup_{I\in\II}I}\}
\]
is a strictly descending collection of subshifts;
\item The entropy function $t\mapsto h(\Kt_\beta(t))$ is constant throughout each interval $I\in\II$.
\end{enumerate}
\end{theorem}

We begin by proving an extension of Lemma \ref{lem:begin-with-s}.

\begin{lemma} \label{lem:weaker-left-extension}
Let $\cs=\r_1\bullet\dots\bullet \r_n$, where {$\r_1\in F_e$ and $\r_i\in F^*$ for $i=2,\dots,n$}, and let $\beta\in(\beta_\ell^{\cs},\beta_*^{\cs}]$. Let $[t_L,t_R]$ be a $\beta$-Lyndon interval such that $b(t_R,\beta)\prec \r_1^-\L(\r_1)^\f$. If $\z\in\Kt_\beta(t_R)$, then $\z$ can be extended to the left to a sequence $\z'\in \Kt_{\beta}(t_R)$ beginning with $\r_1$, and also to a sequence $\z''\in \Kt_{\beta}(t_R)$ beginning with $\L(\r_1)$.
\end{lemma}

The proof goes exactly like that of Lemma \ref{lem:begin-with-s}, replacing $\s$ with $\r_1$ and observing that $\al(\beta)$ begins with $\L(\cs)^+$, which in turn begins with $\L(\r_1)^+$, followed by either $\r_1$ or $\r_1^-$.

\bigskip

Now let $\cs=\r_1\bullet\dots\bullet\r_n\in\La$ with $n\geq 2$, and consider $\beta\in(\beta_\ell^{\cs},\beta_*^{\cs})$.
We follow the same construction as in Section \ref{sec:basic-interiors}. Thus, $j_1=|\cs|$, the first special word $\v_1$ begins with $\al_{j_1+1}$, etc. The procedure from Section \ref{sec:basic-interiors} generates a (finite or infinite) collection $\II=(I_1,I_2,\dots)$ of intervals.

Recall that the transitivity result of Proposition \ref{prop:transitive-in-between} required that $\cs\in\F$. For general $\cs\in\La$, we have instead the following fact:

\begin{proposition} \label{prop:higher-basic-interval-transitivity}
With $\beta$ as above, let $[t_L,t_R]$ be a $\beta$-Lyndon interval. Then $\Kt_\beta(t_R)$ is transitive if and only if $b(t_R,\beta)\prec \r_1^-\L(\r_1)^\f$ and $t_R\not\in\bigcup_{I\in\II}I$.
\end{proposition}

\begin{proof}
Write $\r:=\r_1$ for brevity. If $t_R\in\bigcup_{I\in\II}I$, then $\Kt_\beta(t_R)$ is not transitive by Proposition \ref{prop:non-transitivity-window}. We now show that $\Kt_\beta(t_R)$ is not transitive when $b(t_R,\beta)\lge \r^-\L(\r)^\f$. This follows essentially the same way as in the proof of Proposition \ref{prop:general-transitivity}. Let $\w$ be the $\beta$-Lyndon word such that $b(t_R,\beta)=\w^\f$. Observe that for any sequence $\z\in\Kt_\beta(t_R)$, we have
\[
\Phi_\r(0^\f)=\r^-\L(\r)^\f\lle \w^\f \lle \si^n(\z) \lle \al(\beta) \lle \L(\cs)^+\cs^-\L(\cs)^\f \lle \L(\r)^+\r^\f=\Phi_\r(1^\f).
\]
Thus, if $\z\in\Kt_\beta(t_R)$ begins with $\r^-$ or $\L(\r)^+$, then $\z=\Phi_\r(\hat{\z})$ for some sequence $\hat{\z}$ by Lemma \ref{lem:four-blocks}. Furthermore, since $\w^\f\prec \cs^-\L(\cs)^\f$ and $\cs^-$ begins with the word $\r^-$, $\w^\f$ also begins with $\r^-$. The rest of the proof follows that of Proposition \ref{prop:general-transitivity}.

\medskip

Next, we show that $\Kt_\beta(t_R)$ is transitive when $b(t_R,\beta)\prec \r^-\L(\r)^\f$ and $t_R\not\in\bigcup_{I\in\II}I$. This proceeds almost exactly as in the proof of Proposition \ref{prop:transitive-in-between}, with one essential difference: Since $\w^\f\prec\r^-\L(\r)^\f$, we can find an integer $M$ such that $\w^\f\prec\r^-\L(\r)^M 0^\f$. In each of the three cases in the proof of Proposition \ref{prop:transitive-in-between}, we can assume that $\z$ begins with $\L(\r)$ by Lemma \ref{lem:weaker-left-extension}. We now extend $\z$ further by preceding it with the factor $\L(\r)^M$, calling the new sequence again $\z$. The rest of the proof is now the same as before (changing $\s$ to $\cs$ throughout).
\end{proof}

Recall that $X(\s)=\Phi_\s\big(\{0,1\}^\N\big)$ and $X^*(\s)=\Phi_\s\big(\{0,1\}^*\big)$ for any Lyndon word $\s$.

\begin{lemma} \label{lem:v_k-four-blocks}
Suppose $(\v_k^*)^\f\succ \r_1^-\L(\r_1)^\f$. Then $\v_k\in X^*(\r_1)$, and $\v_k^-(\al_1\dots\al_{j_k}^-)^\f\lge \r_1^-\L(\r_1)^\f$.
\end{lemma}

\begin{proof}
We prove this by induction on $k$. Take first $k=1$, and assume $(\v_1^*)^\f\succ \r^-\L(\r)^\f$, where we write $\r:=\r_1$. 
{Note that $\v_1=\al_{j_1+1}\dots\al_{j_1+l_1}\lle (\cs^-\L(\cs)^\f)_{1:l_1}$ since $\beta\in(\beta_\ell^\cs,\beta_*^\cs)$, and $\cs^-$ begins with $\r^-$.
Suppose first that $\v_1$ is $\beta$-Lyndon. Then $\v_1^\f\succ \r^-\L(\r)^\f$, so $\v_1^\f$ begins with $\r^-$. Furthermore, $\si^n(\v_1^\f)\prec\al(\beta)\lle \L(\cs)^+\cs^\f\lle \L(\r)^+\r^\f=\Phi_\r(1^\f)$. Hence,} $\v_1^\f\in X(\r)$ by Lemma \ref{lem:four-blocks}. {Since $\v_1$ is Lyndon, it follows that $\v_1\in X^*(\r)$ and moreover, the last block of $\v_1$ is $\r$ or $\L(\r)^+$. Thus,} $\v_1^-\in X^*(\r)$ as well, {and $\v_1^-$ ends in a block $\r^-$ or $\L(\r)$. Observing that
\[
\al_1\dots\al_{j_1}^-=\L(\cs)=\Phi_\r\big(\L(\r_2\bullet\dots\bullet\r_n)\big)\in X^*(\r)
\]
and this word begins with a block $\L(\r)^+$,
it follows that $\v_1^-(\al_1\dots\al_{j_1}^-)^\f\in X^*(\r)$. This implies $\v_1^-(\al_1\dots\al_{j_1}^-)^\f\lge \r^-\L(\r)^\f$.

Assume next that $\v_1$ is not $\beta$-Lyndon. Then $\v_1=\u(\al_1\dots\al_{j_1}^-)^r\al_1\dots\al_{j_1}$ for some word $\u$ and integer $r\geq 0$, and $\v_1^*=\u^+$ by Lemma \ref{lem:v_k-star}. Here $\u^+$ is $\beta$-Lyndon and $(\u^+)^\f\succ \r^-\L(\r)^\f$, so $\u\in X^*(\r)$ and $\u$ ends in a block $\r^-$ or $\L(\r)$ by the same argument as above. As in the previous case, it follows that $\v_1^-(\al_1\dots\al_{j_1}^-)^\f=\u(\al_1\dots\al_{j_1}^-)^\f\in X(\r)$ and $\v_1^-(\al_1\dots\al_{j_1}^-)^\f\lge \r^-\L(\r)^\f$.}


Next, let $k\geq 2$ and assume the lemma holds for the words $\v_1,\dots,\v_{k-1}$. Suppose $(\v_k^*)^\f\succ \r^-\L(\r)^\f$. The same argument we used for $\v_1$ shows that $\v_k\in X^*(\r)$. Now since $\v_l^\f{\lge}\v_k^\f\succ \r^-\L(\r)^\f$ for all $l<k$, the induction hypothesis implies that $\v_l\in X^*(\r)$ for all $l<k$, and thus, $\al_1\dots\al_{j_k}=\L(\cs)^+\v_1\dots\v_{k-1}\in X^*(\r)$. It follows that $\v_k^-(\al_1\dots\al_{j_k}^-)^\f\in X(\r)$, and hence, $\v_k^-(\al_1\dots\al_{j_k}^-)^\f\lge \r^-\L(\r)^\f$.
\end{proof}

\begin{proof}[Proof of Theorem \ref{thm:general-basic-intervals}]
By Theorem \ref{thm:basic-interval-transitivity}, the theorem holds for the case $n=1$, i.e. $\cs=\r\in{F_e}$; we simply set $\mathcal{K}_\beta'(t_R):=\Kt_\beta(t_R)$ for $t_R\in\mathcal{T}_R(\beta)\backslash \bigcup_{I\in\II}I$.

Now let $\cs=\r_1\bullet\dots\bullet\r_m$, where $m\geq 2$ and {$\r_1\in F_e$ and $\r_i\in F^*$ for $i=2,\dots,m$}. Let $\beta\in(\beta_\ell^{\cs},\beta_*^{\cs})$, and construct the Lyndon words $\v_1,\v_2,\dots$ and collection of intervals $\II=(I_1,I_2,\dots)$ as outlined above. Then statement (i) follows from Proposition \ref{prop:higher-basic-interval-transitivity}.

We define a new base $\beta'\in[\beta_\ell^{\cs},\beta_*^{\cs})$ as follows. Set $\r:=\r_1$. If $\si^n(\al(\beta)){\succ}\r^-\L(\r)^\f$ for all $n\geq j_1$ (and hence for all $n\geq 0$), we set $\beta':=\beta$ and $n_0:=\f$. Otherwise, let
\[
n_0:=\min\{n\geq j_1: \si^n(\al(\beta)){\lle}\r^-\L(\r)^\f\}.
\]
{Assume first that $n_0<\f$.}
Since the words $\v_1,\v_2,\dots$ are Lyndon, there is an integer $k_0\geq 0$ such that $n_0=j_{k_0+1}$. We let $\beta'$ be the base given by
\[
\al(\beta')=(\al_1\dots\al_{j_{k_0+1}}^-)^\f=\big(\L(\cs)^+\v_1\dots\v_{k_0}^-\big)^\f,
\]
where we interpret the last expression as $\al(\beta')=\L(\cs)^\f$ if $k_0=0$. Let $t_*$ be the point given by $b(t_*,\beta)=\r^-\L(\r)^\f$. By Lemma \ref{lem:v_k-four-blocks}, any interval $I_k\in\II$ lies either completely in $(0,t_*)$ or in $[t_*,\tau(\beta)]$. By definition of $n_0$, the intervals $I_1,\dots,I_{k_0}$ lie to the right of the critical point $t_*$, and $I_{k_0+1},I_{k_0+2},\dots$ lie to the left of $t_*$. 

If $n_0=\f$, then each $I_k$ lies to the right of $t_*$, and we set $k_0:=\f$ for convenience.

Due to the special way in which we constructed the intervals $I_1,I_2,\dots$ for the case of periodic $\al(\beta)$ in Section \ref{sec:basic-interiors}, the non-transitivity windows for $\beta'$ are exactly the intervals $I_1'=[\underline{t}_1',\overline{t}_1'),\dots,I_{k_0}'=[\underline{t}_{k_0}',\overline{t}_{k_0}')$ given by
\[
b(\underline{t}_k',\beta')=\v_k^-(\al_1\dots\al_{j_k}^-)^\f, \qquad b(\overline{t}_k',\beta')=(\v_k^*)^\f, \qquad k=1,2,\dots,k_0.
\]

Let $\w$ be a $\beta$-Lyndon word such that $\w^\f\succ \r^-\L(\r)^\f$ and $t_R:=\pi_\beta(\w^\f)\not\in\bigcup_{k=1}^{k_0}I_k$. 
We first claim that $\w$ is also $\beta'$-Lyndon. This is obvious if $\beta'=\beta$, so assume $\beta'\neq\beta$ so that $n_0<\f$. Suppose $\si^l(\w^\f)\lge \al(\beta')=\big(\L(\cs)^+\v_1\dots\v_{k_0}^-\big)^\f$ for some $l$. Since $\w$ is $\beta$-Lyndon and the words $\v_1,\dots,\v_{k_0}$ decrease lexicograpically, this can only happen if $\w=\v_{k_0}^-\L(\cs)^+\v_1\dots\v_{k_0-1}=\v_{k_0}^-\al_1\dots\al_{j_{k_0}}$. But then
\[
b(\underline{t}_{k_0},\beta)=\v_{k_0}^-(\al_1\dots\al_{j_{k_0}}^-)^\f\prec \w^\f \prec (\v_{k_0}^*)^\f=b(\overline{t}_{k_0},\beta),
\]
so $t_R=\pi_\beta(\w^\f)\in I_{k_0}$, a contradiction. Hence, $\w$ is $\beta'$-Lyndon, and there is a point $t_R'$ such that $b(t_R',\beta')=\w^\f$.

Furthermore, $\si^{n_0}(\al(\beta)){\lle\r^-\L(\r)^\f}\prec\w^\f$ if $n_0$ is finite, and then
\[
\Kt_\beta(t_R)=\{\z\in{A_\beta}^\N: \w^\f\lle \si^n(\z)\lle (\al_1\dots\al_{j_{{k_0+1}}}^-)^\f\ \forall\,n\geq 0\}.
\]
It follows that $\Kt_\beta(t_R)=\Kt_{\beta'}(t_R')$.

Now define $\cs_i:=\r_1\bullet\dots\bullet\r_i$ for $i=1,2,\dots,m$, so in particular, $\cs_m=\cs$. Since ${b(\tau(\beta),\beta)}=\cs^-\L(\cs)^\f=\cs_m^-\L(\cs_m)^\f$, there is a unique $i\in\{1,2,\dots,m-1\}$ such that
\[
\cs_i^-\L(\cs_i)^\f\prec \w^\f \prec \cs_{i+1}^-\L(\cs_{i+1})^\f.
\]
Set $\bR:=\cs_i$ and $\bR':=\r_{i+1}\bullet\dots\bullet\r_m$. It follows (by Lemma \ref{lem:four-blocks} if $n_0=\f$; by Lemma \ref{lem:v_k-four-blocks} if $n_0<\f$) that $\al(\beta')\in X(\bR)$, and hence $\al(\beta')=\Phi_\bR(\al(\hat{\beta}))$ for some base $\hat{\beta}$. Note that $\cs=\bR\bullet\bR'$. Since $\beta'\in [\beta_\ell^{\cs},\beta_*^{\cs})$, we have $\hat{\beta}\in [\beta_\ell^{\bR'},\beta_*^{\bR'})$. We make the following observations.
\begin{itemize}
\item If $n_0=\f$ and $\II$ is infinite, then 
\[
\al(\beta')=\L(\cs)^+\v_1\v_2\v_3\dots \qquad\mbox{and} \qquad \al(\hat{\beta})=\L(\bR')^+\hat{\v}_1\hat{\v}_2\hat{\v}_3\dots,
\] 
where for each $k$, $\hat{\v}_k$ is a Lyndon word such that $\Phi_\bR(\hat{\v}_k)=\v_k$.

\item If $n_0=\f$ and $\II$ is finite, then for some $k\geq 0$, 
\[
\al(\beta')=\L(\cs)^+\v_1\dots\v_k \y \qquad\mbox{and} \qquad \al(\hat{\beta})=\L(\bR')^+\hat{\v}_1\dots\hat{\v}_k \hat{\y},
\] 
where $\y$ and $\hat{\y}$ are sequences satisfying $\si^j(\y)\succ\y$ for all $j\geq 0$, $\y=\Phi_\bR(\hat{\y})$, and hence also $\si^j(\hat{\y})\succ\hat{\y}$ for all $j\geq 0$.

\item If $n_0<\f$, then 
\[
\al(\beta')=\big(\L(\cs)^+\v_1\dots\v_{k_0}^-\big)^\f \qquad\mbox{and} \qquad \al(\hat{\beta})=\big(\L(\bR')^+\hat{\v}_1\dots\hat{\v}_{k_0}^-\big)^\f, 
\]
where again, $\v_k=\Phi_\bR(\hat{\v}_k)$ for $k=1,2,\dots,k_0$. (If $k_0=0$, we have $\al(\beta')=\L(\cs)^\f$ and $\al(\hat{\beta})=\L(\bR')^\f$.)
\end{itemize}

The words $\hat{\v}_1,\hat{\v}_2,\dots$ determine a (finite or infinite) ordered collection $\widehat{\II}=(\hat{I}_1,\hat{I}_2,\dots)$ of intervals in the same way that $\v_1,\v_2,\dots$ determine $I_1,I_2,\dots$. In fact, if $\hat{I}_k=[\underline{\hat{t}}_k,\overline{\hat{t}}_k)$, then
\begin{equation} \label{eq:endpoint-mapping}
b(\underline{t}_k,\beta)=\Phi_\bR\big(b(\underline{\hat{t}}_k,\hat{\beta})\big), \qquad b(\overline{t}_k,\beta)=\Phi_\bR\big(b(\overline{\hat{t}}_k,\hat{\beta})\big).
\end{equation}

Since $\w^\f\succ \bR^-\L(\bR)^\f$, we have $\w=\Phi_\bR(\hat{\w})$ for some $\hat{\beta}$-Lyndon word $\hat{\w}$, as in the proof of Proposition \ref{prop:transitive-subshifts}, and so there is a point $\hat{t}_R$ such that $b(\hat{t}_R,\hat{\beta})=\hat{\w}^\f$. Therefore, since $t_R'\not\in \bigcup_{k=1}^{k_0}I_k'$ and 
\[
b(t_R',\beta')=\w^\f=\Phi_\bR(\hat{\w}^\f)=\Phi_\bR(b(\hat{t}_R,\hat{\beta})), 
\]
we obtain by \eqref{eq:endpoint-mapping} that $\hat{t}_R\not\in\bigcup_{k=1}^{k_0}\hat{I}_k$. Furthermore, since $b(t_R',\beta')=\w^\f\prec \cs_{i+1}^-\L(\cs_{i+1})^\f$, it follows that
\[
b(\hat{t}_R,\hat{\beta})=\Phi_{\bR}^{-1}\big(b(t_R',\beta')\big)\prec \Phi_{\bR}^{-1}\big(\cs_{i+1}^-\L(\cs_{i+1})^\f\big)=\r_{i+1}^-\L(\r_i)^\f.
\]
{Hence, recalling that $\hat{\beta}\in [\beta_\ell^{\bR'},\beta_*^{\bR'})$ and $\bR'=\r_{i+1}\bullet\dots\bullet\r_m$, we can apply Proposition \ref{prop:higher-basic-interval-transitivity} to $\hat{\beta}$ {and $\hat{t}_R$} and conclude that} $\Kt_{\hat\beta}(\hat{t}_R)$ is transitive. We now define a subshift $\mathcal{K}_\beta'(t_R)$ of $\Kt_\beta(t_R)$ by
\[
\mathcal{K}_\beta'(t_R):=\{\sigma^n(\Phi_{\r}(\hat{\z})): \hat{\z}\in \Kt_{\hat\beta}(\hat{t}_R),\ n\geq 0\}.
\]
As in the proof of Proposition \ref{prop:transitive-subshifts}, $\mathcal{K}_\beta'(t_R)$ is transitive and of full entropy and full Hausdorff dimension in $\Kt_{\beta'}(t_R')=\Kt_\beta(t_R)$, and contains the sequence $b(t_R,\beta)=b(t_R',\beta')$. Moreover, the collection
\[
\{\mathcal{K}_\beta'(t_R): t_R\in\mathcal{T}_R(\beta)\backslash \textstyle{\bigcup_{I\in\II}I}\}
\]
is a strictly descending collection of subshifts.

Finally, Proposition \ref{prop:entropy-constant-in-window} implies that $h(\Kt_\beta(t))$ is constant on each interval in $\II$.
\end{proof}

\section{Gaps between $\beta$-Lyndon intervals} \label{sec:gaps}

We saw in Subection \ref{subsec:dense-intervals} that when $\beta\in \overline{E}$, the $\beta$-Lyndon intervals are dense in $[0,\tau(\beta)]$. However, Example \ref{ex:non-dense-intervals} showed that this need not be the case when $\beta$ lies in the interior of a basic interval. In this section we show that having gaps between $\beta$-Lyndon intervals is in fact typical.

Below we fix a basic interval $[\beta_\ell,\beta_*]=[\beta_\ell^{\cs},\beta_*^{\cs}]$ generated by a Lyndon word $\cs\in\La$.

\begin{definition} \label{def:beta-Lyndon-gap}
We say an open interval $(t_1,t_2)$ is a {\em $\beta$-Lyndon gap} if
\begin{enumerate}[(i)]
\item $(t_1,t_2)$ does not intersect any $\beta$-Lyndon intervals; and
\item any open interval $(u_1,u_2)$ properly containing $(t_1,t_2)$ intersects at least one $\beta$-Lyndon interval.
\end{enumerate}
\end{definition}

Note that this definition is indifferent about whether the endpoints $t_1$ and $t_2$ belong to $\beta$-Lyndon intervals.

We first show that $\beta$-Lyndon gaps always exist when $\beta$ lies in a basic interval.
Most importantly, we show that the entropy $h(\Kt_\beta(t))$ stays constant across each $\beta$-Lyndon gap. {Together, these facts allow us to show that $\EE_\beta\backslash\BB_\beta$ is infinite.} 

\begin{proposition} \label{cor:gaps-exist}
Let $\beta\in(\beta_\ell^{\cs},\beta_*^{\cs})$, where $\cs\in\La$. Then there are infinitely many $\beta$-Lyndon gaps, and the left endpoint of each of these lies in $\EE_\beta$.
\end{proposition}

\begin{proof}
Write $\cs=s_1\dots s_m$ and $\L(\cs)=c_1\dots c_m$, and let $j$ be the integer such that $\L(\cs)=s_{j+1}\dots s_m s_1\dots s_j$. Set $\u:=s_{j+1}\dots s_m$. Since $\beta<\beta_*$, we have $\sigma^m(\alpha(\beta))\prec \cs^-\L(\cs)^\f$, so there is an integer $N$ such that $\sigma^m(\alpha(\beta))\prec \cs^-\L(\cs)^N 0^\f$. Choose $M\geq N$, and let $t_M$ be the point given by
\[
b(t_M,\beta)=\cs^-\L(\cs)^M \u^- \L(\cs)^\f.
\]
This is a valid greedy $\beta$-expansion, because it does not contain the word $\L(\cs)^+=\alpha_1\dots\alpha_m$, where $(\al_i):=\al(\beta)$. Also, since $\u^-\prec \L(\cs)$, it is clear that $b(t_M,\beta)\prec \cs^-\L(\cs)^\f=b(\tau(\beta),\beta)$, so $t_M<\tau(\beta)$. Furthermore, $b(t_M,\beta)\succ \sigma^m(\alpha(\beta))$ by the choice of $M$ and $N$. Finally, for $j\leq i<m$, the properties of Lyndon words yield
\[
s_{i+1}\dots s_m^- c_1\dots c_i \lge s_1\dots s_{m-i} c_1\dots c_i \succ s_1\dots s_{m-i} s_{m-i+1}\dots s_m^-={\cs^-},
\]
so that
\[
\sigma^{(M+1)m+i-j}(b(t_M,\beta))=s_{i+1}\dots s_m^-\L(\cs)^\f \succ b(t_M,\beta),
\]
and also,
\[
\sigma^{(M+1)m-j}(b(t_M,\beta))=s_1\dots s_j \u^-\L(\cs)^\f=\cs^-\L(\cs)^\f \succ b(t_M,\beta).
\]
It follows that $\sigma^n(b(t_M,\beta))\lge b(t_M,\beta)$ for all $n\geq 0$; in other words, $t_M\in\EE_\beta$. Hence, by Lemma \ref{lem:E-beta-next-to-beta-Lyndon} (and its proof), $t_M$ is the limit of an increasing sequence of endpoints of $\beta$-Lyndon intervals.

On the other hand, if $k>(M+1)m+|\u|$, it is easy to see that $b_1\dots b_k^+$ {is not $\beta$-Lyndon} (where we write $b(t_M,\beta)=b_1b_2b_3\dots$): If it were, then $b_1\dots b_k^+$ would have to end in $\L(\cs)^+$, but $\L(\cs)^+\cs^-\L(\cs)^M \u^-0^\f\succ\al(\beta)$ by the choice of $M$. Thus, $t_M$ is not a decreasing limit of endpoints of $\beta$-Lyndon intervals. Finally, $t_M$ is not itself a right endpoint of a $\beta$-Lyndon interval because $b(t_M,\beta)$ is aperiodic. Hence, $t_M$ is a left endpoint of a $\beta$-Lyndon gap.

Since different choices of $M$ give rise to different sequences $b(t_M,\beta)$, there are infinitely many $\beta$-Lyndon gaps.
\end{proof}

\begin{proposition} \label{prop:beta-Lyndon-gap-entropy}
Let $(t_1,t_2)$ be a $\beta$-Lyndon gap. Then $h(\Kt_\beta(t_1))=h(\Kt_\beta(t_2))$.
\end{proposition}

\begin{proof}
By Lemmas \ref{lem:E-beta-next-to-beta-Lyndon} and \ref{lem:E-beta-characterizations}, $(t_1,t_2]\cap\EE_\beta=\emptyset$. Thus, if $t\in(t_1,t_2)$ then $\mathcal{K}_\beta(t)=\mathcal{K}_\beta(t_2)$, and so $\Kt_\beta(t)\backslash \Kt_\beta(t_2)$ is at most countable. Hence, $h(\Kt_\beta(t))=h(\Kt_\beta(t_2))$. Letting $t\searrow t_1$ and using the continuity of $t\mapsto h(\Kt_\beta(t))$, the proposition follows.
\end{proof}

\begin{corollary} \label{cor:infinite-bifurcation-difference}
If $\beta\in(\beta_\ell^{\cs},\beta_*^{\cs})$ for some $\cs\in\La$, then $\EE_\beta\backslash\BB_\beta$ is infinite.
\end{corollary}

\begin{proof}
This follows directly from Propositions \ref{cor:gaps-exist} and \ref{prop:beta-Lyndon-gap-entropy}, since the latter implies that the left endpoint of a $\beta$-Lyndon gap does not lie in $\BB_\beta$.
\end{proof}

\begin{corollary} \label{cor:countably-infinite}
Let $\beta\in(\beta_\ell^{\cs},\beta_*^{\cs})$ for some $\cs\in\La$, and suppose $\al(\beta)$ is eventually periodic. If the collection $\II$ of non-transitivity windows is empty, then $\EE_\beta\backslash\BB_\beta$ is countably infinite.
\end{corollary}

\begin{proof}
By the assumption and Proposition \ref{prop:transitive-in-between}, $\Kt_\beta(t_R)$ is transitive for every $\beta$-Lyndon interval $[t_L,t_R]$. Furthermore, since $\al(\beta)$ is eventually periodic, the subshift $\Sigma_\beta$ is sofic and hence $\Kt_\beta(t_R)$ is sofic as well. Thus, if $[t_L,t_R]$ and $[t_L',t_R']$ are two $\beta$-Lyndon intervals with $t_R<t_R'$, then $h(\Kt_\beta(t_R))>h(\Kt_\beta(t_R'))$ by \cite[Corollary 4.4.9]{Lind_Marcus_1995}. So the only intervals on which $\dim_H K_\beta(t)$ is constant are the $\beta$-Lyndon intervals, which contain no points of $\EE_\beta$ except the right endpoints, and $\beta$-Lyndon gaps, which contain no points of $\EE_\beta$ except possibly the left endpoints. In other words, each interval on which $\dim_H K_\beta(t)$ is constant contains at most one point from $\EE_\beta$. Hence, $\EE_\beta\backslash\BB_\beta$ is countable, and by Corollary \ref{cor:infinite-bifurcation-difference}, it is infinite.
\end{proof}

Example \ref{ex:v_k} (c) presents a $\beta$ satisfying the hypotheses of Corollary \ref{cor:countably-infinite}.

\section{Properties of extended $\beta$-Lyndon intervals} \label{sec:EBLI}

Recall the definition of an extended $\beta$-Lyndon interval (EBLI) from Section \ref{sec:introduction}.

\begin{lemma} \label{lem:EBLI-left-endpoint}
Let $[t_L^*,t_R]$ be an EBLI that is not a $\beta$-Lyndon interval. Then $t_L^*\in \EE_\beta$.
\end{lemma}

\begin{proof}
We have to check that
\begin{equation} \label{eq:b-shifts-bigger}
\si^n(b(t_L^*,\beta))\lge b(t_L^*,\beta) \qquad \forall\,n\geq 0.
\end{equation}
Let $\w$ be the $\beta$-Lyndon word generating the EBLI, so $b(t_L^*,\beta)=\w^-(\al_1\dots\al_m^-)^\f$ and $b(t_R,\beta)=\w^\f$, where $m:=\min\{n:\si^n(\al(\beta))\lle \w^\f\}$.
Write $\w=w_1\dots w_l$. We first establish that
\begin{equation} \label{eq:alpha-and-w-inequality}
w_{i+1}\dots w_l^-(\al_1\dots\al_m^-)^\f \lle (\al_1\dots\al_m^-)^\f \qquad \forall\,0\leq i<l.
\end{equation}
If $l-i\leq m$, then $w_{i+1}\dots w_l^-\lle \al_1\dots \al_{l-i}^-$ since $\w$ is $\beta$-Lyndon, which yields \eqref{eq:alpha-and-w-inequality} for this case.

If $l-i>m$, then we claim that $w_{i+1}\dots w_{i+m}\lle \al_1\dots\al_m^-$. Suppose to the contrary that $w_{i+1}\dots w_{i+m}\lge \al_1\dots\al_m$. We have $\si^m(\al(\beta))\lle \w^\f\lle \si^{i+m}(\w^\f)$ since $\w$ is Lyndon, and hence it follows that
\[
\al(\beta)=\al_1\dots\al_m\si^m(\al(\beta))\lle w_{i+1}\dots w_{i+m}\si^{i+m}(\w^\f)=\si^i(\w^\f),
\]
contradicting that $\w$ is $\beta$-Lyndon. Hence, $w_{i+1}\dots w_{i+m}\lle \al_1\dots\al_m^-$. We can then cancel the first $m$ digits on both sides of \eqref{eq:alpha-and-w-inequality} and continue inductively until we are left with a word of length $\leq m$ in front of $(\al_1\dots\al_m^-)^\f$ in the left hand side, at which point we know the inequality to be true from the first case above.

For $1\leq j<l$ we have $w_{j+1}\dots w_l^-\lge w_1\dots w_{l-j}$ (since $\w$ is Lyndon), and so \eqref{eq:alpha-and-w-inequality} with $i=l-j$ gives
\[
\si^j(b(t_L^*,\beta))=w_{j+1}\dots w_l^-(\al_1\dots\al_m^-)^\f\lge \w^-(\al_1\dots\al_m^-)^\f=b(t_L^*,\beta).
\]
It remains to check that
\begin{equation} \label{eq:tail-after-w}
\al_{j+1}\dots\al_m^-(\al_1\dots\al_m^-)^\f\lge \w^-(\al_1\dots\al_m^-)^\f \qquad \forall\, 0\leq j<m.
\end{equation}
This is obvious for $j=0$, so assume $1\leq j<m$. If $l\leq m-j$, then $\al_{j+1}\dots\al_{j+l}\lge\w$ since $\si^j(\al(\beta))\succ\w^\f$, hence \eqref{eq:tail-after-w} holds for this case.

So suppose $l>m-j$. Then $\si^j(\al(\beta))\succ\w^\f$ implies $w_1\dots w_{m-j}\lle \al_{j+1}\dots\al_m$. If we had equality here then by the Lyndon property of $\w$ we would get
\[
\si^j(\al(\beta))=w_1\dots w_{m-j}\si^m(\al(\beta))\lle w_1\dots w_{m-j}\w^\f\lle \w^\f,
\]
contradicting the minimality of $m$. Thus, $w_1\dots w_{m-j}\lle \al_{j+1}\dots\al_m^-$. Together with \eqref{eq:alpha-and-w-inequality} (taking $i=m-j$), this gives \eqref{eq:tail-after-w}.
\end{proof}

\begin{lemma} \label{lem:no-overlap-EBLI}
Any two EBLIs are either non-overlapping or else one contains the other.
\end{lemma}

\begin{proof}
Any two $\beta$-Lyndon intervals are disjoint, and a $\beta$-Lyndon interval $[t_L,t_R]$ and an EBLI $[u_L^*,u_R]$ {with $t_R\neq u_R$} cannot overlap because the left endpoint $u_L^*$ lies in $\EE_\beta$ by Lemma \ref{lem:EBLI-left-endpoint}, and therefore does not lie in $[t_L,t_R]$.

It remains to consider the case of two EBLIs 
\[
I_\v=\big[\pi_\beta\big(\v^-(\al_1\dots\al_{m(\v)}^-)^\f\big),\pi_\beta(\v^\f)\big], \qquad I_\w=\big[\pi_\beta\big(\w^-(\al_1\dots\al_{m(\w)}^-)^\f\big),\pi_\beta(\w^\f)\big].
\]
Suppose they overlap without one containing the other. Without loss of generality,
\begin{equation} \label{eq:overlapping-EBLI}
\v^-\big(\al_1\dots\al_{m(\v)}^-\big)^\f\prec \w^-\big(\al_1\dots\al_{m(\w)}^-\big)^\f\prec \v0^\f\prec \v^\f\prec \w^\f,
\end{equation}
where the second inequality holds because $\w^-(\al_1\dots\al_{m(\w)}^-)^\f$ does not lie in the $\beta$-Lyndon interval $[\v0^\f,\v^\f]$, as observed at the beginning of the proof.
{Recall that $m(\v)=\min\{n\geq 0:\si^n(\al(\beta))\lle \v^\f\}$ and $m(\w)=\min\{n\geq 0:\si^n(\al(\beta))\lle \w^\f\}$. Thus, $\v^\f\prec \w^\f$ implies $m(\v)\geq m(\w)$ and, since $\v$ and $\w$ are Lyndon, $\v^\f\prec \w0^\f$. So $\v^\f$ begins with $\w^-$.} In fact $\v$ extends $\w^-$: It cannot be a prefix of $\w^-$ because $\w^-(\al_1\dots\al_{m(\w)}^-)^\f\prec \v0^\f$. Say $\v=\w^-u_1\dots u_l$. Then from \eqref{eq:overlapping-EBLI},
\begin{equation} \label{eq:u-w-sandwich}
u_1\dots u_l^-\big(\al_1\dots\al_{m(\v)}^-\big)^\f\prec \big(\al_1\dots\al_{m(\w)}^-\big)^\f\prec u_1\dots u_l0^\f.
\end{equation}
Hence the sequence $(\al_1\dots\al_{m(\w)}^-)^\f$ begins with $u_1\dots u_l^-$, and the first inequality in \eqref{eq:u-w-sandwich} implies
\[
\big(\al_1\dots\al_{m(\v)}^-\big)^\f\prec \si^l\big(\big(\al_1\dots\al_{m(\w)}^-\big)^\f\big).
\]
But this is impossible: {Since $m(\v)\geq m(\w)$ and $\al_{l+1}\dots \al_{m(\w)}^-\prec\al_1\dots\al_{m(\w)-l}$ for all $l<m(\w)$ by Lemma \ref{lem:quasi-greedy expansion-alpha-q}, we have}
\[
\si^l\big(\big(\al_1\dots\al_{m(\w)}^-\big)^\f\big)\lle \big(\al_1\dots\al_{m(\w)}^-\big)^\f \lle \big(\al_1\dots\al_{m(\v)}^-\big)^\f.
\]
This contradiction completes the proof.
\end{proof}

\begin{proposition} \label{prop:dense-EBLIs}
The union of all EBLIs is dense in $[0,\tau(\beta)]$.
\end{proposition}

\begin{proof}
We know from Proposition \ref{prop:dense-intervals} and Theorems \ref{thm:beta-in-relative-exceptional-set} and \ref{thm:infinitely-Farey} that the $\beta$-Lyndon intervals are dense in $[0,\tau(\beta)]$ when $\beta\in \overline{E}\cup \bigcup_{\cs\in\La}\overline{E^\cs}\cup E_\f$. So it remains to show that the EBLIs are dense for $\beta$ in the interior of a basic interval, i.e. $\beta\in(\beta_\ell^\cs,\beta_*^\cs)$ with $\cs\in\La$.

Let $t\in{(0,\tau(\beta))}$ with greedy $\beta$-expansion $b(t,\beta)=b_1b_2\dots$. Let $\mathcal{N}$ denote the set of all indices $n\in\N$ such that {$b_1\dots b_n\in L_e$. Note that $\mathcal{N}\neq\emptyset$: If $b_1>0$, then $b_1\in L_e$; otherwise, since $(b_i)\neq 0^\f$, there is a smallest $n\geq 2$ such that $b_n>0$ and then $b_1\dots b_n\in L_e$.} There are now three possibilities:
\begin{enumerate}[(1)]
\item $\mathcal{N}$ is finite and $b_1\dots b_n$ is $\beta$-Lyndon for each $n\in\mathcal{N}$;
\item $\mathcal{N}$ is infinite and $b_1\dots b_n$ is $\beta$-Lyndon for each $n\in\mathcal{N}$; or
\item $b_1\dots b_n$ is not $\beta$-Lyndon for some $n\in\mathcal{N}$.
\end{enumerate}
In the first case, letting $m:=\max\mathcal{N}$, we have $\si^{m}(b(t,\beta))\lle b(t,\beta)$ and so $t$ lies in the $\beta$-Lyndon interval $[b_1\dots b_{m},(b_1\dots b_{m})^\f]$.

In the second case, $t$ does not lie in any $\beta$-Lyndon interval but is the limit of left endoints of $\beta$-Lyndon intervals $[b_1\dots b_n,(b_1\dots b_n)^\f], n\in\mathcal{N}$.

So it remains to deal with the third case. We show that in this case $t$ lies in some $EBLI$. Let $N$ be the smallest $n\in\mathcal{N}$ for which $b_1,\dots b_n$ is {\em not} $\beta$-Lyndon, and put $\w:=b_1\dots b_{N}$. We first aim to determine $\w^*$. Observe that, since $t<\tau(\beta)=\pi_\beta(\cs^-\L(\cs)^\f)$, there is at least one $\beta$-Lyndon word, namely $\cs$, that is greater than $\w$; hence $\w^*$ is well defined {(see Lemma \ref{lem:existence-of-v-star})}, and we can write $\w^*=b_1\dots b_l^+$ for some $l<N$.

Since $\w$ is not $\beta$-Lyndon, there is an index $k<N$ such that $b_{k+1}\dots b_N b_1\dots b_k\lge\al_1\dots \al_N$. Let $k_1$ be the smallest such $k$. In particular $b_{k_1+1}\dots b_N\lge \al_1\dots\al_{N-k_1}$. But, since $b(t,\beta)$ is a greedy $\beta$-expansion, we also have the reverse inequality. Hence,
\begin{equation} \label{eq:k_1-facts}
b_{k_1+1}\dots b_N=\al_1\dots\al_{N-k_1}, \qquad\mbox{and} \qquad b_1\dots b_{k_1}\lge\al_{N-k_1+1}\dots\al_N.
\end{equation}
Clearly, then, $l\leq k_1$ as otherwise $b_1\dots b_l^+$ would include the forbidden word $\al_1\dots\al_{l-k_1}^+$.

{By the minimality of $k_1$, we must have $b_{k_1}{<M_\beta}$.}
Suppose $b_1\dots b_{k_1}^+$ is $\beta$-Lyndon. Then $\w^*=b_1\dots b_{k_1}^+$, and 
\[
m:=\min\{n\geq 0: \si^n(\al(\beta))\lle (\w^*)^\f\}\leq N-k_1, 
\]
because $\al_{N-k_1+1}\dots\al_N\lle b_1\dots b_{k_1}\prec \w^*$. So in this case, we obtain
\begin{align*}
(\w^*)^-(\al_1\dots\al_m^-)^\f &=b_1\dots b_{k_1}(\al_1\dots\al_m^-)^\f \lle b_1\dots b_{k_1}(\al_1\dots\al_{N-k_1}^-)^\f\\
&=b_1\dots b_{k_1}(b_{k_1+1}\dots b_N^-)^\f\prec b(t,\beta),
\end{align*}
and hence, $t$ lies in the EBLI generated by $\w^*$.

If $b_1\dots b_{k_1}^+$ is not $\beta$-Lyndon, we repeat the process with $k_1$ in place of $N$. Precisely, there is a smallest integer $k_2<k_1$ such that $b_{k_2+1}\dots b_{k_1}^+b_1\dots b_{k_2}\lge \al_1\dots\al_{k_1}$. We claim that in fact
\begin{equation} \label{eq:plus-block}
b_{k_2+1}\dots b_{k_1}^+=\al_1\dots\al_{k_1-k_2},
\end{equation}
and hence $b_1\dots b_{k_2}\lge \al_{k_1-k_2+1}\dots \al_{k_1}$. To see the claim, suppose by way of contradiction that $b_{k_2+1}\dots b_{k_1}\lge \al_1\dots\al_{k_1-k_2}$. Then this must hold with equality since $b_{k_2+1}\dots b_{k_1}$ is part of $b(t,\beta)$. {We also have $b_{k_1+1}\dots b_N=\al_1\dots\al_{N-k_1}$ from \eqref{eq:k_1-facts}, so
\[
b_{k_2+1}\dots b_{k_1}b_{k_1+1}\dots b_N=\al_1\dots\al_{k_1-k_2}\al_1\dots\al_{N-k_1}.
\]
By Lemmas \ref{lem:quasi-greedy expansion-alpha-q} and \ref{lem:greedy-expansion}, this implies $\al_1\dots\al_{N-k_1}=\al_{k_1-k_2+1}\dots\al_{N-k_2}$ and, by the minimality of $k_1$, $b_1\dots b_{k_2}\prec \al_{N-k_2+1}\dots\al_N$. Since furthermore, $b_1\dots b_{k_1}\lge\al_{N-k_1+1}\dots\al_N$ and $k_1>k_2$, it follows that
\[
\al_{N-k_2+1}\dots\al_N\succ \al_{N-k_1+1}\dots\al_{N-(k_1-k_2)},
\]
and hence,
\begin{align*}
\al_1\dots\al_{N-(k_1-k_2)}&=\al_{k_1-k_2+1}\dots\al_{N-k_2}\al_{N-k_1+1}\dots\al_{N-(k_1-k_2)}\\
&\prec \al_{k_1-k_2+1}\dots\al_{N-k_2}\al_{N-k_2+1}\dots\al_N,
\end{align*}
contradicting that $\si^{k_1-k_2}(\al(\beta))\lle \al(\beta)$. This proves \eqref{eq:plus-block}.
}


{Next, we claim that $k_1-k_2\leq N-k_1$. For, if this were not the case, then \eqref{eq:plus-block} would yield} $b_{k_2+1}\dots b_{k_2+N-k_1}=\al_1\dots\al_{N-k_1}$, so by \eqref{eq:k_1-facts},
\[
b_{k_2+N-k_1+1}\dots b_N \lle \al_{N-k_1+1}\dots\al_{N-k_2}\lle b_1\dots b_{k_1-k_2},
\]
contradicting that $b_1\dots b_N$ is Lyndon. 

Suppose now that $b_1\dots b_{k_2}^+$ is $\beta$-Lyndon. Since $\al_{k_1-k_2+1}\dots \al_{k_1}\lle b_1\dots b_{k_2}$, we have
\[
\si^{k_1-k_2}(\al(\beta))=\al_{k_1-k_2+1}\al_{k_1-k_2+2}\dots \prec (b_1\dots b_{k_2}^+)^\f,
\]
so $m:=\min\{n\geq 0: \si^n(\al(\beta))\lle (\w^*)^\f\}\leq k_1-k_2$. Hence,
\begin{align*}
(\w^*)^-(\al_1\dots\al_m^-)^\f & \lle b_1\dots b_{k_2}(\al_1\dots\al_{k_1-k_2}^-)^\f
=b_1\dots b_{k_1}(\al_1\dots\al_{k_1-k_2}^-)^\f \\
&\lle b_1\dots b_{k_1}(\al_1\dots\al_{N-k_1}^-)^\f \prec b_1\dots b_{k_1}\al_1\dots\al_{N-k_1}0^\f\\
&=b_1\dots b_N 0^\f\lle b(t,\beta),
\end{align*}
where the first equality uses \eqref{eq:plus-block}, and the second inequality uses that $k_1-k_2\leq N-k_1$. So, $t$ lies in the EBLI generated by $\w^*$.

If $b_1\dots b_{k_2}^+$ is not $\beta$-Lyndon, we repeat the above process with $k_2$ replacing $k_1$ and the Lyndon word $b_1\dots b_{k_1}^+$ replacing the Lyndon word $b_1\dots b_N$. Continuing this way, we obtain a sequence $N>k_1>k_2>\dots$. The process must stop after a finite number of steps, say $p$, and then $\w^*=b_1\dots b_{k_p}^+$ and $t$ lies in the EBLI generated by $\w^*$.
\end{proof}

\begin{lemma} \label{lem:EBLI-constant-entropy}
Assume $\beta\in(\beta_\ell^\cs,\beta_*^\cs)$ for some $\cs\in\La$. Then the entropy function $t\mapsto h(\Kt_\beta(t))$ is constant in each EBLI that lies fully in $[0,\tau(\beta)]$.
\end{lemma}

\begin{proof}
The entropy is obviously constant on each $\beta$-Lyndon interval, so consider an EBLI $I_\w=\big[\pi_\beta\big(\w^-(\al_1\dots\al_m^-)^\f\big),\pi_\beta(\w^\f)\big]$ inside $[0,\tau(\beta)]$. Since $\beta$-Lyndon intervals are dense in $[0,\tau(\beta)]$ for $\beta\in \overline{E}\cup \bigcup_{\cs\in\La}\overline{E^\cs}\cup E_\f$ and EBLIs are non-overlapping by Lemma \ref{lem:no-overlap-EBLI}, it must be the case that $\beta\in(\beta_\ell^\cs,\beta_*^\cs)$ for some $\cs\in\La$. Moreover, $\w^\f\lle b(\tau(\beta),\beta)=\cs^-\L(\cs)^\f \prec \cs^\f$. So $\w$ and $\cs$ are two distinct $\beta$-Lyndon words, and $m:=\min\{n:\si^n(\al(\beta))\lle\w^\f\}\geq |\cs|$, because $\al(\beta)$ begins with $\L(\cs)^+$.

Let $t_L^*$ and $t_R$ be the left and right endpoints of $I_\w$. It suffices to show that
\[
h(\Kt_\beta(t_L^*))=h(\Kt_\beta(t_R)).
\]
Suppose $\z\in \Kt_\beta(t_L^*)\backslash \Kt_\beta(t_R)$. Then
\begin{equation} \label{eq:z-boundaries}
b(t_L^*,\beta)=\w^-(\al_1\dots\al_m^-)^\f\lle \si^n(\z) \lle \al(\beta) \qquad\forall\,n\geq 0,
\end{equation}
and there is some integer $n_0$ such that $\si^{n_0}(\z)\prec b(t_R,\beta)=\w^\f$. We claim that $\z$ must end in a sequence of the form
\begin{equation} \label{eq:repeating-pattern-2}
\w^-(\al_1\dots\al_m^-)^{p_1}\,\al_1\dots\al_m \w^{q_1}\w^-(\al_1\dots\al_m^-)^{p_2}\,\al_1\dots\al_m \w^{q_2}\dots,
\end{equation}
where $0\leq p_i,q_i\leq\infty$ for all $i$. Since $\si^{n_0}(\z)\prec \w^\f$, there is some $n_1\geq n_0$ such that $\si^{n_1}(\z)$ begins with $\w^-$. By \eqref{eq:z-boundaries}, it then follows that $\si^{n_1}(\z)$ begins with $\w^-(\al_1\dots\al_m^-)^{p_1}\al_1\dots\al_m$ for some $p_1\leq\infty$. But now recall that $\si^m(\al(\beta))\lle \w^\f$. Thus, if $p_1<\infty$, then again using \eqref{eq:z-boundaries}, the word $\al_1\dots\al_m$ must be followed by $\w^{q_1}\w^-$ for some $q_1\leq\infty$. Repeating this argument, we see that $\z$ ends in a sequence of the form \eqref{eq:repeating-pattern-2}.

Now let $\mathcal{X}$ denote the set of all sequences of the form \eqref{eq:repeating-pattern-2}, together with their shifts. Then $\mathcal{X}$ is a subshift of $\Kt_\beta(t_L^*)$. Also, let $\mathcal{Y}$ denote the set of all infinite concatenations of the words $\cs$ and $\w$ and their shifts, that is,
\[
\mathcal{Y}:=\left\{\si^n\big(\cs^{p_1}\w^{q_1}\cs^{p_2}\w^{q_2}\dots\big): \ n\geq 0,\ 1\leq p_i,q_i\leq\f\ \forall\,i\right\}.
\]
If $\z\in\mathcal{Y}$, then $\si^n(\z)\lge \w^\f$ for all $n\geq 0$ since $\w$ and $\cs$ are both Lyndon and $\w^\f\prec \cs^\f$. Furthermore, $\si^n(\z)\lle \al(\beta)$ for all $n\geq 0$ since $\cs$ and $\w$ are both $\beta$-Lyndon. Hence, $\mathcal{Y}$ is a subshift of $\Kt_\beta(t_R)$. Since neither $\w$ nor $\cs$ is a power of the other (else one of them wouldn't be Lyndon) and $|\cs|\leq m$, a comparison of the definitions of $\mathcal{X}$ and $\mathcal{Y}$ shows that $h(\mathcal{Y})\geq h(\mathcal{X})$. As a result,
\[
h(\Kt_\beta(t_L^*))\leq \max\{h(\Kt_\beta(t_R)),h(\mathcal{X})\}\leq \max\{h(\Kt_\beta(t_R)),h(\mathcal{Y})\}=h(\Kt_\beta(t_R)).
\]
The reverse inequality is obvious. Therefore, the proof is complete.
\end{proof}

\begin{lemma} \label{lem:window-is-EBLI}
The closure of every non-transitivity window (that is, every interval in the collection $\II_{\max}$) is an EBLI.
\end{lemma}

\begin{proof}
Let $I_k=[\underline{t}_k,\overline{t}_k)\in\II_{\max}$ be a non-transitivity window for $\beta\in(\beta_\ell^\cs,\beta_*^\cs)$ with $\cs\in\La$. Set $m:=j_k$. If $\v_k$ is $\beta$-Lyndon, then $\v_k^*=\v_k$, and so we have $b(\underline{t}_k,\beta)=\v_k^-(\al_1\dots\al_m^-)^\f$ and $b(\overline{t}_k,\beta)=\v_k^\f$. By Lemma \ref{v_k-star-less-than-s}, $\v_k^\f\prec \cs^\f$, and since $\al(\beta)$ begins with $\L(\cs)^+$, this implies that $\si^n(\al(\beta))\succ \v_k^\f$ for all $n<m$. Therefore, $\overline{I_k}$ is an EBLI.

Suppose now that $\v_k$ is not $\beta$-Lyndon. Since $I_k$ is a non-transitivity window, $I_k\not\subseteq I_{k-1}$ and so $\v_k\neq \v_{k-1}$. Thus, by Lemma \ref{lem:v_k-star}, there exist a nonnegative integer $r$ and a word $\u$ not ending in $\al_1\dots\al_{j_k}^-$ such that $\v_k=\u(\al_1\dots\al_{j_k}^-)^r\al_1\dots\al_{j_k}$, and
\begin{equation*}
\v_k^*=\begin{cases}
\v_{k-1}^* & \mbox{if $k\geq 2$ and $\u=\v_{k-1}^-$},\\
\u^+ & \mbox{otherwise}.
\end{cases}
\end{equation*}
If $k\geq 2$ and $\u=\v_{k-1}^-$, then, since $j_k>j_{k-1}$, it follows that $\v_k$ begins with $\v_{k-1}^-\al_1\dots\al_{j_{k-1}}$ and so $I_k\subseteq I_{k-1}$ by Proposition \ref{prop:non-transitivity-intervals}. This contradicts our assumption that $I_k$ is a non-transitivity window. Hence, this case cannot happen and $\v_k^*=\u^+$. Now observe that
\[
\v_k^-(\al_1\dots\al_m^-)^\f=\u(\al_1\dots\al_m^-)^{r+1}(\al_1\dots\al_m^-)^\f=(\v_k^*)^-(\al_1\dots\al_m^-)^\f.
\]
Hence, $\overline{I_k}$ is the EBLI generated by $\v_k^*$.
\end{proof}

\begin{lemma} \label{lem:transitive-at-EBLI-right-endpoint}
Let $\mathcal{T}_R'$ denote the set of all right endpoints of {\em maximal} EBLIs in $[0,\tau(\beta)]$. Then for each $t\in\mathcal{T}_R'$, $\Kt_\beta(t)$ has a transitive subshift $\mathcal{K}_\beta'(t)$ of full entropy containing the sequence $b(t,\beta)$. Moreover, these subshifts can be chosen so that $\{\mathcal{K}_\beta'(t): t\in \mathcal{T}_R'\}$ is a strictly descending collection of subshifts, and they are sofic if $\al(\beta)$ is eventually periodic.
\end{lemma}

\begin{proof}
A maximal EBLI either contains a non-transitivity window or else does not intersect any non-transitivity windows, in view of Lemmas \ref{lem:no-overlap-EBLI} and \ref{lem:window-is-EBLI}. Thus, the result follows from Proposition \ref{prop:transitive-in-between} and Theorems \ref{thm:beta-in-relative-exceptional-set}, \ref{thm:infinitely-Farey} and \ref{thm:general-basic-intervals}.
\end{proof}

\begin{corollary} \label{cor:entropy-different}
Assume $\al(\beta)$ is eventually periodic. Then the entropy function $t\mapsto h(\Kt_\beta(t))$ takes distinct values on distinct maximal EBLIs in $[0,\tau(\beta)]$.
\end{corollary}

\begin{proof}
Immediate from Lemma \ref{lem:transitive-at-EBLI-right-endpoint} and \cite[Corollary 4.4.9]{Lind_Marcus_1995}.
\end{proof}

\begin{corollary} \label{cor:maximal-EBLI}
Assume $\al(\beta)$ is eventually periodic. Then an EBLI is maximal if and only if it is not properly contained in the closure of a non-transitivity window.
\end{corollary}

\begin{proof}
The forward direction is immediate from Lemma \ref{lem:window-is-EBLI}. For the converse, assume an EBLI $I_\w=[t_L^*,t_R]$ is not properly contained in the closure of a non-transitivity window. Suppose there is another EBLI, say $I_\v=[u_L^*,u_R]$ that properly contains $I_\w$. Then $\Kt_\beta(t_R)$ (resp. $\Kt_\beta(u_R)$) has a transitive sofic subshift $\mathcal{K}_\beta'(t_R)$ (resp. $\mathcal{K}_\beta'(u_R)$) of full entropy that contains the sequence $b(t_R,\beta)$ (resp. $b(u_R,\beta)$), because neither $t_R$ nor $u_R$ lies in a non-transitivity window. Since $\mathcal{K}_\beta'(u_R)$ is a proper subshift of $\mathcal{K}_\beta'(t_R)$, this implies $h(\Kt_\beta(t_R))>h(\Kt_\beta(u_R))$, contradicting that the entropy is constant on $[u_L^*,u_R]$.
\end{proof}

\begin{proof}[Proof of Theorem \ref{thm:entropy-plateaus}]
Recall the decomposition \eqref{eq:decomposition}.
Assume $\al(\beta)$ is eventually periodic. Then $\beta\not\in E$ by Proposition \ref{prop:Sturmian}. Likewise, $\beta\not\in E^\cs$ for any $\cs\in\La$ because if $\beta\in E^\cs$, then $\al(\beta)=\Phi_\cs(\al(\hat{\beta}))$ for some $\hat{\beta}\in E$, and $\al(\hat{\beta})$ is not eventually periodic, so $\al(\beta)$ isn't either. Furthermore, $\beta\not\in E_\f$ (see the remark following Proposition \ref{prop:E-inf-bifurcation}). Thus, by \eqref{eq:decomposition}, $\beta\in I^\cs=[\beta_\ell^\cs,\beta_*^\cs]$ for some $\cs\in\La$.

If $\beta\in\{\beta_\ell^\cs,\beta_*^\cs\}$, then the $\beta$-Lyndon intervals are dense in $[0,\tau(\beta)]$ by Proposition \ref{prop:dense-intervals} and Theorem \ref{thm:beta-in-relative-exceptional-set}, so the EBLIs in $[0,\tau(\beta)]$ are just the $\beta$-Lyndon intervals and the conclusion follows from Theorem \ref{thm:E-minus-B}.

Finally, if $\beta\in(\beta_\ell^\cs,\beta_*^\cs)$, the theorem follows from Proposition \ref{prop:dense-EBLIs}, Lemma \ref{lem:EBLI-constant-entropy} and Corollary \ref{cor:entropy-different}.
\end{proof}

\begin{proof}[Proof of Corollary \ref{cor:symbolic-plateaus}]
We first show that $\b'$ is well defined. First, suppose $\si^n(\b)\succ\b$ for some $n$, and take $n$ minimal with this property. Then $b_n{<b_1}$, and setting $\b':=(b_1\dots b_n)^\f$ we see that $0^\f\prec\si^n(\b')\lle \b'$ for all $n\geq 0$, and it is easy to see that no sequence between $\b'$ and $\b$ satisfies these inequalities. Furthermore, $\Sigma_{\a,\b}=\Sigma_{\a,\b'}$ for all $\a$. Second, suppose $\si^n(\b)\lle \b$ for all $n\geq 0$, but $\b=b_1\dots b_n0^\f$ with $b_n{>0}$. Setting $\b':=(b_1\dots b_n^-)^\f$ we see that $0^\f\prec\si^n(\b')\lle \b'$ for all $n\geq 0$ by the assumption $\b\succ 10^\f$. Again, no sequence between $\b'$ and $\b$ satisfies these inequalities. Moreover, $\Sigma_{\a,\b'}\subseteq\Sigma_{\a,\b}$ and $\Sigma_{\a,\b}\backslash\Sigma_{\a,\b'}$ is at most countable for each $\a$.
Thus, in both cases, $\b'$ is well defined, $h(\Sigma_{\a,\b})=h(\Sigma_{\a,\b'})$ for all $\a$, and $\b'=\al(\beta)$ for some $\beta{>1}$ by Lemma \ref{lem:quasi-greedy expansion-alpha-q}.

Now assume $\b'$ is eventually periodic. For simplicity we relabel $\b'$ as $\b$ in the rest of the proof.
If $[\pi_\beta(\w0^\f),\pi_\beta(\w^\f)]$ is a plateau of $t\mapsto h(\Kt_\beta(t))$, then $[\w^-\b,\w^\f]$ is a plateau of $\a\mapsto h(\Sigma_{\a,\b})$ since $\Sigma_{\a,\b}=\Sigma_{\w0^\f,\b}$ for all $\a\in (\w^-\b,\w0^\f]$. In the same way, if $[\pi_\beta\big(\w^-(\al_1\dots\al_m^-)^\f\big),\pi_\beta(\w^\f)]$ is a plateau of $t\mapsto h(\Kt_\beta(t))$, then $[\w^-(\al_1\dots\al_m^-)^\f,\w^\f]$ is a plateau of $\a\mapsto h(\Sigma_{\a,\b})$.

It remains to show that no interval of the form $[\u^-\b,\u0^\f]$, with $\u$ an allowed word in $\Sigma_\beta$ { not ending in $0$}, is a plateau of $\a\mapsto h(\Sigma_{\a,\b})$. (Since $\pi_\beta(\u^-\b)=\pi_\beta(\u0^\f)$, the interval $[\u^-\b,\u0^\f]$ collapses to a point when projecting under $\pi_\beta$, so it is at least conceivable that $[\u^-\b,\u0^\f]$ is a plateau.) This is equivalent to showing that the point $t_0:=\pi_\beta(\u0^\f)$ already lies in a plateau of $t\mapsto h(\Kt_\beta(t))$, i.e. in an EBLI.

If $\u$ is not Lyndon, then $t_0$ lies inside some $\beta$-Lyndon interval, hence in an EBLI. If $\u$ is $\beta$-Lyndon, then $t_0$ is the left endpoint of a $\beta$-Lyndon interval, hence is contained in an EBLI. Finally, suppose $\u$ is Lyndon but not $\beta$-Lyndon. Then the proof of Proposition \ref{prop:dense-EBLIs} shows that $t_0$ lies in an EBLI. Hence, in all cases, the interval $[\u^-\b,\u0^\f]$ already lies in one of the plateaus $\tilde{I}_\w$.
\end{proof}

\begin{remark} \label{rem:zero-plateau}
The $0$-plateau of $\a\mapsto h(\Sigma_{\a,\b})$ is $[b(\tau(\beta),\beta),{M_\beta}^\f]$ as long as $b(\tau(\beta),\beta)$ does not end in $0^\f$. The only exceptional case is when $\beta=\beta_r^\cs$ for $\cs\in\La$, in which case $b(\tau(\beta),\beta)=\cs0^\f$ and the $0$-plateau of $\a\mapsto h(\Sigma_{\a,\b})$ is $[\cs^-\b,{M_\beta}^\f]$. (An equivalent result was proved also by Labarca and Moreira \cite{Labarca-Moreira-2006}.)
\end{remark}

\section{Proofs of the other main theorems} \label{sec:other-proofs}

\begin{proof}[Proof of Theorem \ref{thm:general-transitivity}]
Recall the decomposition \eqref{eq:decomposition}.
For $\cs\in{F_e}$, the left endpoint $\beta_\ell^\cs$ of $I^{\cs}$ lies in $E_L$; for $\cs=\s\bullet\r\in\La$ with $\s\in{F_e}$ and {$\r\in\La^*$ (where $\La^*$ was defined in \eqref{eq:Lambda-star}}), $\beta_\ell^\cs$ lies in $\overline{E^\s}$; and for all $\cs\in\La$, the right endpoint $\beta_*^\cs$ of $I^{\cs}$ lies in $\overline{E^\cs}$.
Thus, the theorem follows from Theorems \ref{thm:transitive-in-E_L}, \ref{thm:basic-interval-right-endpoint}, \ref{thm:beta-in-relative-exceptional-set}, \ref{thm:infinitely-Farey} and \ref{thm:general-basic-intervals}.
\end{proof}


\begin{proof}[Proof of Theorem \ref{thm:size-of-E-minus-B}]
We first show that $\dim_H(\EE_\beta\backslash\BB_\beta)>0$ for almost all $\beta{>1}$. Fix for the moment a basic interval $I^\cs=[\beta_\ell^\cs,\beta_*^\cs]$ where $\cs\in\La$, and take $\beta\in(\beta_\ell^\cs,\beta_*^\cs)$. Then $\al(\beta)$ begins with $\L(\cs)^+$. Let $k$ be the integer such that $\si^{|\cs|}(\al(\beta))$ begins with {$0^k d$ for some $d>0$}. If there is a string of more than $k$ consecutive zeros in $\al(\beta)$ beyond $\L(\cs)^+{0^k d}$, then it follows from the construction in Section \ref{sec:basic-interiors} that the collection $\II$ of non-transitivity windows is non-empty. It is well known (see \cite[p.~678]{Schmeling-97}) that for Lebesgue-almost all $\beta{>1}$ the sequence $\al(\beta)$ contains arbitrarily long strings of consecutive zeros. Hence for almost all $\beta\in I^\cs$ there is at least one non-transitivity window, and for such $\beta$, $\dim_H(\EE_\beta\backslash\BB_\beta)>0$ by Theorem \ref{thm:bifurcation-in-basic-interval}. Since the basic intervals $I^\cs$, $\cs\in\La$ cover the interval $(1,{\f)}$ up to a set of Lebesgue measure zero (see \cite{Allaart-Kong-2021}), we conclude that $\dim_H(\EE_\beta\backslash\BB_\beta)>0$ for almost all $\beta{>1}$.

(In fact the above argument can easily be extended to show that for almost all $\beta{>1}$, there are infinitely many non-transitivity windows.)

Next, we show that for each $k\in\{0,1,2,\dots\}\cup\{\aleph_0\}$ there are infinitely many $\beta$'s such that $|\EE_\beta\backslash\BB_\beta|=k$. For finite $k$ this is an immediate consequence of Theorem \ref{thm:E-minus-B}, because for each $k\in\N$ the set $\La_k$ is infinite. For $k=\aleph_0$ it follows from Corollary \ref{cor:countably-infinite}, since for each $\cs\in\La$ we can find a base $\beta\in (\beta_\ell^\cs,\beta_*^\cs)$ such that $\II=\emptyset$ and $\al(\beta)$ is eventually periodic. (For instance, let $m:=|\cs|$ and take $\al(\beta)=\L(\cs)^+0^{m+1}(10)^\f$.)

Finally, we show that, if $\EE_\beta\backslash\BB_\beta$ is uncountable, then its Hausdorff dimension is strictly positive. The intersection of $\EE_\beta$ with at least one complementary interval of $\BB_\beta$, say $[t_1,t_2)$, must be uncountable. Since the intersection of $\EE_\beta$ with each $\beta$-Lyndon interval or $\beta$-Lyndon gap contains at most one point, the interval $[t_1,t_2)$ must contain infinitely many $\beta$-Lyndon intervals. Let $[t_L,t_R]$ and $[u_L,u_R]$ be two of them, ordered so that $t_R<u_R$. Let $\v$ and $\w$ be the $\beta$-Lyndon words such that $b(t_R,\beta)=\v^\f$ and $b(u_R,\beta)=\w^\f$. Now consider all sequences of the form
\[
\v^3 \w\v^{k_1}\w\v^{k_2}\w\v^{k_3}\w\dots, \qquad k_i\in\{1,2\}\ \forall\,i.
\]
Since $\v^\f\prec \w^\f$, it is easy to see that each such sequence is the $\beta$-greedy expansion of some point $t\in(t_R,u_R)\cap\EE_\beta$. Hence, as in the proof of Theorem \ref{thm:bifurcation-in-basic-interval}, it follows that
\[
\dim_H(\EE_\beta\backslash\BB_\beta) \geq \dim_H\big(\EE_\beta\cap[t_1,t_2)\big) \geq \dim_H\big(\EE_\beta\cap(t_R,u_R)\big)>0.
\]
This completes the proof.
\end{proof}

\begin{proof}[Proof of Theorem \ref{thm:local-dimension}]
(i) Let $\beta{>1}$ and $t\in\BB_\beta$. Fix $\ep>0$. Then by Theorem \ref{thm:main},
\begin{gather*}
\dim_H\big(\EE_\beta\cap[t,1]\big)=\dim_H K_\beta(t),\\
\dim_H\big(\EE_\beta\cap[t+\ep,1]\big)=\dim_H K_\beta(t+\ep).
\end{gather*}
Since $t\in\BB_\beta$, $\dim_H K_\beta(t+\ep)<\dim_H K_\beta(t)$. Hence, using the stability of Hausdorff dimension, it follows that $\dim_H\big(\EE_\beta\cap[t,t+\ep)\big)=\dim_H\big(\EE_\beta\cap[t,1]\big)$. Thus,
\begin{align*}
\dim_H K_\beta(t)&=\dim_H\big(\EE_\beta\cap[t,1]\big)=\dim_H\big(\EE_\beta\cap[t,t+\ep)\big)\\
&\leq \dim_H\big(\EE_\beta\cap(t-\ep,t+\ep)\big)\leq \dim_H\big(\EE_\beta\cap[t-\ep,1]\big)\\
&=\dim_H K_\beta(t-\ep),
\end{align*}
so letting $\ep\searrow 0$ and using the continuity of $t\mapsto \dim_H K_\beta(t)$ gives (i).

(ii) Assume $\al(\beta)$ is eventually periodic; then $\Kt_\beta(t_R)$ (or $\mathcal{K}_\beta'(t_R)$, as appropriate) is sofic for every $\beta$-Lyndon interval $[t_L,t_R]$. Take $t\in\BB_\beta$. Since $\BB_\beta\subseteq\EE_\beta$, we have
\begin{equation} \label{eq:BB-local-upper}
\dim_H\big(\BB_\beta\cap(t-\ep,t+\ep)\big)\leq \dim_H\big(\EE_\beta\cap(t-\ep,t+\ep)\big).
\end{equation}
On the other hand, it follows from our results on $\EE_\beta\backslash\BB_\beta$ and the argument in the proof of Proposition \ref{prop:E-B-equality} that the only intervals with which $\EE_\beta\backslash\BB_\beta$ can have an uncountable intersection are the non-transitivity windows, if they are present. Since $\al(\beta)$ is eventually periodic, the construction in Section \ref{sec:basic-interiors} implies that there are at most finitely many non-transitivity windows. Furthermore, since $t\in\BB_\beta$, $t$ cannot be a left endpoint of a non-transitivity window by Proposition \ref{prop:entropy-constant-in-window}. Thus, for all sufficiently small $\ep>0$, the interval $(t,t+\ep)$ does not intersect any non-transitivity window. Hence it contains at most countably many points of $\EE_\beta\backslash\BB_\beta$, and so
\begin{equation} \label{eq:BB-local-lower}
\dim_H\big(\BB_\beta\cap[t,t+\ep)\big)\geq \dim_H\big(\EE_\beta\cap[t,t+\ep)\big).
\end{equation}
From \eqref{eq:BB-local-upper}, \eqref{eq:BB-local-lower} and part (i) of the theorem, (ii) follows.
\end{proof}

\begin{remark}
The reason why we cannot extend Theorem \ref{thm:local-dimension} (ii) to all $\beta\in(1,2]$ has to do with a technical limitation. All our upper bounds for the size of $\EE_\beta\backslash\BB_\beta$ depend on \cite[Corollary 4.4.9]{Lind_Marcus_1995}, which states that, if $\mathcal{X}$ is a transitive sofic subshift and $\mathcal{Y}$ is a proper subshift of $\mathcal{X}$, then $h(\mathcal{X})>h(\mathcal{Y})$. Without the assumption of $\mathcal{X}$ being sofic, this conclusion may fail. When $\al(\beta)$ is not eventually periodic, the subshift $\Kt_\beta(t_R)$ is not necessarily sofic, and therefore the argument in the proof of Proposition \ref{prop:E-B-equality} does not work. We suspect that nonetheless, Theorem \ref{thm:local-dimension} (ii) holds for all $\beta\in(1,2]$.
\end{remark}

\section*{Acknowledgments}
{The authors wish to thank Wolfgang Steiner for pointing out a mistake in an earlier draft.}
The first author was partially supported by Simons Foundation grant \#709869.
The second author was supported by NSFC No.~11971079.

%
\end{document}